\documentclass[11pt]{amsart}

\usepackage{DKstyle}
\usepackage{bm}
\usepackage{float}

\usepackage{comment}

\usepackage{caption}
\usepackage[labelfont=rm]{subcaption}

\usepackage[pagebackref=false, colorlinks]{hyperref}

\usepackage{cleveref}

\newcommand{\B}{\mathrm{B}}

\def\DI{\operatorname{\mathbf{DI}}}
\newcommand{\Mod}[1]{\ (\mathrm{mod}\ #1)}
\newcommand{\vertiii}[1]{{\left\vert\kern-0.25ex\left\vert\kern-0.25ex\left\vert #1
    \right\vert\kern-0.25ex\right\vert\kern-0.25ex\right\vert}}
 \makeatletter
\newcommand*{\rom}[1]{\expandafter\@slowromancap\romannumeral #1@}
\makeatother
\theoremstyle{plain}

\newcommand{\ang}{\measuredangle}

\newcommand{\abs}[1]{\lvert#1\rvert} 

\newcommand{\ttr}{\mathtt{r}}

\hypersetup{pdffitwindow=true,linkcolor=blue,citecolor=blue,urlcolor=blue,menucolor=blue}

\DeclareMathOperator{\inj}{inj}
\DeclareMathOperator{\sys}{sys}
\DeclareMathOperator{\Leb}{Leb}


\date{}

\dedicatory{}
\commby{}

\title{Zero-one laws for uniform approximation via Gaussian and Eisenstein integers}

\author[Ren\'e Pfitscher]{Ren\'e Pfitscher}
\address{Ren\'e Pfitscher, School of Mathematical Sciences, University of Science and Technology of China (USTC), 230026, Hefei, China
\newline \rule[0ex]{0ex}{0ex} \hspace{8pt}{\tt pfitscher@ustc.edu.cn}}

\author[Anurag Rao]{Anurag Rao}
\address{Anurag Rao, SIMIS, Fudan University, 200433, Shanghai, China
\newline \rule[0ex]{0ex}{0ex} \hspace{8pt}{\tt arao@simis.cn}}

\author[Shucheng Yu]{Shucheng Yu}
\address{Shucheng Yu, School of Mathematical Sciences, University of Science and Technology of China (USTC), 230026, Hefei, China
\newline \rule[0ex]{0ex}{0ex} \hspace{8pt}{\tt yusc@ustc.edu.cn}}

\author[Han Zhang]{Han Zhang}
\address{Han Zhang, School of Mathematical Sciences, Shenzhen University, 518060,  Shenzhen, China
\newline \rule[0ex]{0ex}{0ex} \hspace{8pt}{\tt hzmath@szu.edu.cn}}

\thanks{R.P. and S.Y. are supported by the
National Key R\&D Program of China No. 2024YFA1015100. H.Z. is supported by  the National Natural Science Foundation of China (No. 12501250)}

\date{}

\begin{document}
\begin{abstract}

We establish two distinct zero-one laws for the uniform Diophantine approximation of complex numbers by quotients of Gaussian integers and by quotients of Eisenstein integers. Using tools from homogeneous dynamics, we study this problem by reducing it to a shrinking target problem on certain homogeneous spaces of $\SL_2(\C)$. The main novel ingredients include measure estimates on a certain family of neighborhoods of the corresponding critical loci, as well as  new disjointness statements to control the short-range mixing contribution. Due to the  different nature of the critical loci in the Gaussian and Eisenstein cases, these measure estimates are obtained by rather different arguments. 
\end{abstract}

\maketitle

\setcounter{tocdepth}{2} 
\tableofcontents

\section{Introduction}\label{introSEC}

Let $d$ be a positive square-free integer and $K=\Q(\sqrt{-d})$ 
be the corresponding imaginary quadratic field. In this paper we study the uniform Diophantine approximation problem of approximating complex numbers by quotients of Gaussian integers and by quotients of Eisenstein integers, that is, by quotients of algebraic integers in $\Q(\sqrt{-1})$ and $\Q(\sqrt{-3})$, respectively.

The starting point of our investigation is the following  Dirichlet-type approximation theorem in this complex setting: \begin{Thm}
\label{thm:dirichlet}
Let $K=\Q(\sqrt{-d})$ be an imaginary quadratic field with $\cO_K$ its ring of integers. Then there exists some constant $c_K>0$ such that for all $z\in \C$ and for all $t>  c_K$, there exists a pair $(p,q)\in \cO_K^2$ satisfying  
\begin{align}\label{equ:dirichlet}
|qz-p|< \frac{c_K}{t},\qquad 0<|q|\leq t, 
\end{align}
where $|\cdot|$ is the usual absolute value of complex numbers. In addition, for any $0<c<c_K$ the set of $z\in \C$ for which the system of inequalities 
\begin{align}\label{eq: dirichlet impro}
|qz-p|\leq \frac{c}{t},\qquad 0<|q|\leq t
\end{align}
is solvable in $(p,q)\in \cO_K^2$ for all $t$ sufficiently large is Lebesgue null.  Moreover,  for any $d\neq 1,3$, we have the following bound for $c_K$:
\begin{align}\label{equ:bdck}
\frac{|D_K|}{4+2|D_K|^{1/2}}\leq c_K\leq \frac{2}{\pi}|D_K|^{1/2},
\end{align}
where $D_K$ is the \textit{discriminant} of $K$, that is, $D_K=-4d$ if $d\equiv 1, 2\Mod{4}$ and $D_K=-d$ if $d\equiv 3\Mod{4}$. 
\end{Thm}

The constant $c_K$ that appears in Theorem \ref{thm:dirichlet} is known as the \textit{Dirichlet constant}. When $K=\Q(\sqrt{-1})$, \eqref{equ:dirichlet} and \eqref{eq: dirichlet impro} in Theorem \ref{thm:dirichlet} were originally proved by Chevallier \cite[Theorem 5]{Chevallier2021}, 
who also derived  the explicit constant $c_{\Q(\sqrt{-1})}=\frac{\sqrt{2}}{3-\sqrt{3}}$. For a general imaginary quadratic field $K$, the existence of such a  constant $c_K$ follows from standard techniques of homogeneous dynamics using Dani's correspondence and ergodicity, following the ideas developed in \cite[Sec. 8.7]{KleinbockMargulis1999}; see Section \ref{sec:thm11} for more details. 

{From the bound \eqref{equ:bdck} we immediately get that
\begin{align*}
\frac12\leq \liminf_{d\to\infty}\frac{c_K}{|D_K|^{1/2}} \leq \limsup_{d\to\infty}\frac{c_K}{|D_K|^{1/2}}\leq \frac{2}{\pi}=0.6366\cdots.
\end{align*}
It is an interesting question to get exact values or more precise asymptotic estimates on the constant $c_K$.  Our bound \eqref{equ:bdck} on $c_K$ is obtained by relating it to a certain \textit{critical radius} of a family of lattices in $\C^2$ parameterized by the homogeneous space $\SL_2(\C)/\SL_2(\cO_K)$ (see \eqref{equ:costrel}) and then proving bounds on this critical radius (see Proposition \ref{prop:bdonrk}).  We mention that for some small values of $d$, our analysis indeed gives  better lower bounds (see \eqref{equ:bdckmp}); we state the bound \eqref{equ:bdck} here only for simplicity of presentation. }

{
For the two special cases when $K=\Q(\sqrt{-1})$ and $K=\Q(\sqrt{-3})$, Minkowski \cite{MinkowskiDA} explicitly determined the corresponding \textit{critical locus}, i.e. the set of lattices attaining the critical radius. In particular, it follows from his result, together with the aforementioned relation between $c_K$ and the critical radius, that 
\begin{align}\label{equ:ckvalue}
c_{\Q(\sqrt{-1})}=\frac{\sqrt{2}}{3-\sqrt{3}} \quad \text{and}\quad  c_{\Q(\sqrt{-3})}=1.
\end{align}
}

Fix an imaginary quadratic field $K$ and let $\psi_{1}(t)=1/t$. Inspired by Davenport and Schmidt’s notion of improving Dirichlet’s theorem,
for any $0<c<c_K$, we say that $z$ is $(c,K)$-Dirichlet improvable if, for all sufficiently large $t$, there exist $p,q\in\cO_K$ such that $|qz-p|<c/t$ and $0<|q|<t$. We denote this set by $\DI_K(c\psi_1)$.
Let $\DI_K(c\psi_1)\subset \C/\cO_K$ be the collection of all $(c,K)$-Dirichlet improvable complex numbers. This set is well defined, since it follows directly from the definition that $z$ is $(c,K)$-Dirichlet improvable if and only if $z+z'$ is $(c,K)$-Dirichlet improvable for every $z'\in \cO_K$. The set of all \textit{Dirichlet improvable} complex numbers with respect to $K$ is then defined to be
$$\DI_K:=\bigcup_{c\in(0,c_K)} \DI_K(c\psi_1). $$
It is then an immediate consequence of Theorem \ref{thm:dirichlet} that $\Leb(\DI_K)=0$, where $\Leb$ denotes the  Lebesgue measure on $\C$ normalized so that 
$
\Leb(\C/\cO_K)=1.
$ 
This is analogous to the well-known results  in real Euclidean spaces established by Davenport and Schmidt \cite{hDwS69b,hDwS70a}, and then Kleinbock and Weiss \cite{KleinbockWeiss2008}.

Rather than these special functions $c\psi_1$ with $c\in (0,c_K)$, one would like to consider a more general function $\psi$ and investigate those $z\in \C$ satisfying \eqref{eq: dirichlet impro} for $\psi$ in place of $c\psi_1$. More precisely, let $\psi: [1,\infty)\to (0,1)$ be a continuous decreasing function. 
We say $z\in \mathbb{C}$ is \textit{$(\psi,K)$-Dirichlet} 
if for all sufficiently large $t$, the following system of inequalities is solvable in $(p,q)\in \cO_K^2$:
\begin{align}\label{equ:unifapp}
|qz-p|< \psi(t),\qquad 0<|q|< t. 
\end{align}

Let $\mathbf{DI}_K(\psi)\subset \C/\cO_K$ be the set of $(\psi,K)$-Dirichlet complex numbers. Similarly, $\mathbf{DI}_K(\psi)$ is well-defined since for any $z\in \C$ and $z'\in \cO_K$,  $z$ is $(\psi,K)$-Dirichlet if and only if $z+z'$ is $(\psi,K)$-Dirichlet. We would like to understand the size of $\mathbf{DI}_K(\psi)$ for a given approximating function $\psi$. This problem is often referred to as \textit{uniform Diophantine approximation}.

When $K=\Q(\sqrt{-1})$ or $K=\Q(\sqrt{-3})$, the following theorem establishes zero-one laws for the Lebesgue measure of $\DI_K(\psi)$ for a general approximating function $\psi$ satisfying some mild assumptions.

\begin{Thm}\label{thm: zero-one laws}
Let $K=\Q(\sqrt{-1})$ or $K=\Q(\sqrt{-3})$ with $c_K$ as in \eqref{equ:ckvalue}. Let $\psi :[1,\infty)\to (0,1)$ be a continuous decreasing function such that $\psi(t)<\frac{c_K}{t}$ for all $t\geq 1$. Let $$F_{\psi}(t):=\frac{c_K-t\psi(t)}{c_K}.$$
Assume further that 
\begin{align}\label{equ:lorecondpsi}
    \exists\, C_{\psi}\geq 1\ \text{such that }\ F_{\psi}(t_2)\leq C_{\psi}F_{\psi}(t_1)\qquad \text{whenever $1\leq t_1\leq t_2\leq 2 t_1$}.
\end{align}
Then the following zero-one laws hold.
\begin{enumerate}
 \item[(i)]  If $K=\Q(\sqrt{-1})$, then
\begin{align}\label{equ:01lawgauss}
\mathrm{Leb}(\DI_K(\psi))=\left\lbrace\begin{array}{cc}
0 & \text{if } \sum_kk^{-1}F_{\psi}(k)^4 =\infty,\\
1 &  \text{if } \sum_kk^{-1}F_{\psi}(k)^4<\infty.
\end{array}\right. 
\end{align}  
\item[(ii)]
If $K=\Q(\sqrt{-3})$, then
\begin{align}\label{equ:01laweisenstein}
\mathrm{Leb}(\DI_K(\psi))=\left\lbrace\begin{array}{cc}
0 & \text{if } \sum_kk^{-1}F_{\psi}(k)^2\log\left(1+\tfrac{1}{F_{\psi}(k)}\right) =\infty,\\
1 &   \text{if } \sum_kk^{-1}F_{\psi}(k)^2\log\left(1+\tfrac{1}{F_{\psi}(k)}\right) <\infty.
\end{array}\right. 
\end{align}
\end{enumerate}    
\end{Thm}

Condition \eqref{equ:lorecondpsi} is a local regularity assumption on $F_\psi$, weaker than requiring $t\mapsto t\psi(t)$ to be increasing, which gives $C_\psi=1$. The factor $2$ is inessential and may be replaced by any fixed constant greater than $1$, by iteration. As in \cite[Remark 1.1]{StrombergssonYu2026}, this condition provides the local control needed to compare the series criterion with the measures of the shrinking targets and to pass from the continuous-time Dani correspondence to a discrete-time shrinking-target problem.

It is also worth noting that the two series appearing in the statement of this theorem are different. This difference stems from the distinct structures of the critical loci for $\Q(\sqrt{-1})$ and $\Q(\sqrt{-3})$, as determined by Minkowski~\cite{MinkowskiDA}.

Let us briefly review the literature on uniform Diophantine approximation for a general approximating function $\psi$. This problem  originated from the pioneering work of Kleinbock and Wadleigh~\cite{KleinbockWadleigh2018}. In this work, they ingeniously utilized the regular continued fraction theory to establish the zero-one law for the Lebesgue measure of $\psi$-Dirichlet real numbers. This zero-one law was later generalized to higher-dimensional real matrix spaces with respect to weighted quasi-norms of any dimension by Kleinbock, Str\"{o}mbergsson and Yu \cite{KleinbockStrombergssonYu2022}, with an additional assumption on the function $\psi$ in the divergent case  (see \cite[(1.11)]{KleinbockStrombergssonYu2022}). Recently, Str\"{o}mbergsson and Yu \cite{StrombergssonYu2026} removed this additional condition by a geometric  disjointness argument, thereby establishing the full zero-one law for uniform approximation in real matrix spaces of any dimension. In a different but related context, Kleinbock and Rao \cite{KR22} studied zero-one laws for uniform approximation in the real line with respect to the Euclidean norm and the notion of Dirichlet-improvability with respect to more general norms in Euclidean space.

Another notable result is the work of Kleinbock and Wadleigh \cite{KleinbockWadleigh2019} on a zero-one law for uniform approximation in real matrix spaces in the inhomogeneous setting. The main result in \cite{KleinbockWadleigh2019} was further generalized in the work of Kim and Kim \cite{KimKim2022} through the consideration of Hausdorff measure and dimension.

The theory of uniform Diophantine approximation over $\R$ is well developed, whereas the analogous problem over $\C$ remains less understood. Here one approximates a complex number by quotients of algebraic integers in an imaginary quadratic field. Much of the existing work seeks complex continued-fraction algorithms with properties analogous to those of regular continued fractions. For example, Lakein \cite{Lakein73} and Schmidt \cite{A.Schmidt75} studied convergents arising from several such algorithms for imaginary quadratic fields. Dani and Nogueira \cite{DN14} later treated these expansions in a unified framework for Gaussian integers.

 One such algorithm is Hurwitz's continued-fraction expansion \cite{Hurwitz87} for approximation by Gaussian integers $\Z[i]$. It has been studied extensively and used to establish metric properties of sets determined by the resulting partial quotients; see, for example,~\cite{HX22,HX25,BGRH25}. Continued fractions have also been used to characterize badly approximable complex numbers with respect to an imaginary quadratic field~\cite{Hines19,Gerardo20}. Subsequent work~\cite{EK10,EGL16,KleinbockLy2016} studied the Hausdorff dimension of sets of badly approximable numbers. See~\eqref{equ:badapp} below for the definition in the complex setting.

Additionally, the  \textit{asymptotic approximation} theory in this complex setting has already been studied extensively. For a given approximating function $\psi$, asymptotic approximation, unlike  uniform approximation,  investigates those $z\in \C$ for which \eqref{equ:unifapp} has solutions for arbitrarily large  $t$. LeVeque \cite{LeV52} first derived a Khintchine-type theorem for asymptotic approximation of a complex number with Gaussian integers. LeVeque's result was generalized to an arbitrary imaginary quadratic field by Sullivan \cite{Sul82} using hyperbolic geometry. This was further extended by Ly \cite{Ly16} to complex matrix spaces of any dimension based on tools from homogeneous dynamics. Recently, a Schmidt-type counting result for asymptotic approximation in this complex setting was obtained by Alam and Str\"ombergsson \cite{AS24}, relying on an adelic Rogers' formula.

Although the works above study continued fractions and asymptotic approximation in depth, none addresses zero--one laws for uniform approximation in this setting. To the best of our knowledge, Theorem \ref{thm: zero-one laws} establishes the first such  zero-one laws. In contrast to the Khintchine-type theorem for asymptotic approximation in the complex setting, where the zero-one law is determined by the same series criterion for any imaginary quadratic field $K$ (see \cite[Theorem 5]{Sul82}), the zero-one laws for uniform approximation in this setting differ significantly for $K=\Q(\sqrt{-1})$ and $K=\Q(\sqrt{-3})$. {This reflects the subtlety of the uniform approximation theory in the complex setting and suggests that this problem may depend on deep arithmetic properties of the individual imaginary quadratic field. }

\subsection{Proof strategy for the zero-one laws}
{Unlike some previous works \cite{HX22,BGRH25,HX25} on Diophantine approximation over $K=\Q(\sqrt{-1})$, which rely on  Hurwitz continued fractions, our proof of  Theorem \ref{thm: zero-one laws} uses tools from homogeneous dynamics}. This enables us to consider uniform Diophantine approximation in $\C$ using not only Gaussian integers, but also Eisenstein integers. 

The connection between Diophantine approximation and homogeneous dynamics is known as \textit{Dani's correspondence}. This was discovered by Dani \cite{Dani1985}, and it enables one to ``observe" the Diophantine properties of a given number via the behavior of a certain associated orbit in homogeneous spaces. Dani's correspondence was later refined by Kleinbock and Margulis \cite{KleinbockMargulis1999} for  asymptotic Diophantine approximation, and then was adjusted in the work of Kleinbock-Str\"{o}mbergsson-Yu \cite{KleinbockStrombergssonYu2022} for uniform approximation in the real setting. 

In this article, we adapt the approach of \cite{KleinbockStrombergssonYu2022} to the complex setting. More precisely, let $K=\Q(\sqrt{-1})$ or $\Q(\sqrt{-3})$. It is well-known that the corresponding \textit{Bianchi group} $\Gamma_K=\SL_2(\cO_K)$ is a lattice in $G=\SL_2(\C)$ (see e.g. \cite{MR03}). We let $X=G/\Gamma_K$ and $\mu_K$ be the unique $G$-invariant probability measure on $X$. Note that the homogeneous space $X$ naturally parameterizes the space of free, discrete, cocompact  $\cO_K$-submodules of $\C^2$ admitting an ordered basis with determinant $1$ via the map $g\G_K\mapsto g\cO_K^2$.  Consider the function $\Delta:X\to \R$ defined by 
\begin{align}\label{def:delta} 
\Delta(\Lambda)=\sup_{\bm{v}\in \Lambda\smallsetminus\{\bm{0}\}} \log\left( \tfrac{\sqrt{c_K}}{\norm{\bm{v}}} \right), 
\end{align}
where $c_K$ is the Dirichlet constant as in Theorem \ref{thm:dirichlet} and $\|\cdot\|$ is the supremum norm on $\C^2$. Indeed, by relating the constant $c_K$ to a certain critical radius for lattices in  $X$, we will see that $\Delta(\Lambda)\geq 0$ for all $\Lambda\in X$ and the set $\Delta^{-1}\{0\}\subset X$ is referred to as \textit{the critical locus} with respect to the norm $\|\cdot\|$. By Dani's correspondence (see Proposition \ref{prop:dani}), the proof of Theorem \ref{thm: zero-one laws}  reduces to a shrinking target problem on $X$.  The corresponding shrinking targets are described in terms of the level sets $\Delta^{-1}[0,r]$ for small parameters $r>0$. It is therefore necessary to understand the asymptotic behavior of the  measures of these  level sets, which is the content of the following  result.

\begin{Thm}\label{thm: measure estimates}
The following measure estimates hold.
\begin{align}\label{equ:meestboth}
\mu_K(\Delta^{-1}[0,r])\asymp \left\lbrace\begin{array}{cc}
r^5 & \text{ if }\, K=\Q(\sqrt{-1}),\\
r^3\log\left(\frac{1}{r}\right) & \text{ if }\,  K=\Q(\sqrt{-3}),
\end{array}\right. \qquad \text{as $r\to 0^+$}.
\end{align}
\end{Thm}
Theorem \ref{thm: measure estimates} provides asymptotic measure estimates for the neighborhood of the critical locus determined by $\Delta$. This result is of independent interest. The fact that the asymptotic formulas vary is once again due to the different structures of the critical loci in these two cases.

{Similar measure estimates also arise naturally in the study of the uniform approximation problem in the real setting. More precisely, one is led  to  estimate the measure of  neighborhoods of certain critical loci in the homogeneous space $\SL_d(\R)/\SL_d(\Z)$. }
For $d=2$, this was obtained in the works \cite{StrombergssonVenkatesh2005,KelmerYu2020} based on Siegel's transform. For any $d\geq 2$,
this was done in \cite{KleinbockStrombergssonYu2022}, where the strategy is to bound these neighborhoods both from above and below by appropriate measurable sets in $X$, whose measures can be readily estimated. Achieving this is intricate and demands an understanding of the explicit structure of these neighborhoods. 

{For the proof of measure estimates in Theorem \ref{thm: measure estimates}, we adapt the approach in \cite{KleinbockStrombergssonYu2022}. However, the complex setting introduces new difficulties in characterizing the structure of $\Delta^{-1}[0,r]$ for  small $r>0$. In the real setting, the measure estimate relies crucially on the simple fact that a non-zero real number has either a positive or a negative sign. This dichotomy is not available in the complex setting, where the possible directions of  a non-zero complex number are parameterized by the unit circle. This extra freedom of parameters leads to several complexities in our analysis. Moreover, since the critical loci differ in the cases of $\Q(\sqrt{-1})$ and $\Q(\sqrt{-3})$, the structure of $\Delta^{-1}[0,r]$ is anticipated to be rather different, suggesting additional challenges. Despite these difficulties, we succeed in obtaining matching upper and lower bounds for $\mu_K(\Delta^{-1}[0,r])$ in both the Gaussian and Eisenstein cases. This is based on different key observations specific to each corresponding case, and uses a refined form of the triangle inequality (see Lemma \ref{lemma: basic inequality}). } The details are given in \Cref{sec: measure estimates}.

{Theorem \ref{thm: measure estimates} provides sharp measure estimates for the target sets arising in our shrinking target problem. To solve this shrinking target problem, however, one also needs dynamical input. A now-standard approach in homogeneous dynamics is to use effective single and double equidistribution results for suitable expanding unipotent orbits in the corresponding homogeneous space.}

{
In our setting, namely the homogeneous space
$\SL_2(\C)/\SL_2(\cO_K)$,
such effective equidistribution results are standard. The single equidistribution theorem can be obtained either by Margulis' thickening argument together with effective mixing estimates \cite{KleinbockMargulis1996,EskinMcMullen1993}, or by spectral methods \cite{Sodergren2012,Edwards17}. In this paper we use the version proved by Edwards \cite[Theorem $1'$]{Edwards17}. For the double equidistribution theorem, we follow the strategy of Kleinbock-Shi-Weiss \cite{KleinbockShiWeiss2017}.}

{These two effective equidistribution theorems are used in two main ways: first, to relate the probabilities of the hitting events to the measures of the corresponding shrinking targets; and second, to establish the quasi-independence of these hitting events, which is needed for a suitable divergence Borel-Cantelli lemma. However, when applying these effective results, we encounter smoothing issues similar to those arising in the real setting; see \cite[Remark 8]{KleinbockStrombergssonYu2022}. These issues prevent us from obtaining satisfactory control over certain short-range mixing terms associated with the hitting events.}

To overcome this difficulty, we follow the strategy of the recent work of Str\"ombergsson and Yu \cite{StrombergssonYu2026} and establish a disjointness statement for our shrinking targets in an appropriate range. In the case $K=\Q(\sqrt{-3})$, the critical locus has a structure analogous to that in the real setting, and the required disjointness statement can therefore be proved in essentially the same way. In contrast, when $K=\Q(\sqrt{-1})$, this argument breaks down because the critical locus exhibits a different geometric behavior. Nevertheless, we are still able to prove the desired disjointness statement by using the crucial observation that, modulo a compact torus, the critical locus in this case is contained in a union of two closed geodesic orbits; see Lemma \ref{rmk:cloges} below. We prove the desired disjointness statement for both cases in Proposition \ref{prop:disjoint}. Using these disjointness arguments together with the two effective equidistribution theorems, we are able to establish the two zero-one laws stated in Theorem \ref{thm: zero-one laws}.

\vspace{5pt}

\noindent{\bf Notation and Conventions.} 
Throughout the paper, the notation $\norm{\cdot}$ denotes the supremum norm, either on $\C^2$ or on $\mathrm{M}_2(\C)$, depending on the context. For $r>0$, let $\B_{\C^2}(r)$ be the open $r$-ball centered at the origin in $\C^2$, and let $\B_G(r)$ be the open $r$-ball centered at the identity element in $G=\SL_2(\C)$, both with respect to $\norm{\cdot}$.
For two positive quantities $A$ and $B$, we write ``$A\ll B$" to denote $A\leq cB$ for some constant $c>0$, and we write
``$A\asymp B$" to denote $A\ll B\ll A$. We will also write ``$O(A)$" to denote any real number $E$ satisfying $|E|\leq cA$ for some constant $c>0$. We will sometimes use subscripts to indicate the dependence of the implicit constant (i.e., the constant $c$ above) on parameters.

\vspace{5pt}

\noindent{\bf Structure of the paper.}  In Section \ref{sec: geometry of numbers}, we introduce preliminaries on geometry of numbers and, following the work of Chevallier and Minkowski, provide an explicit description of the critical locus in the Gaussian and Eisenstein cases. We also prove Theorem \ref{thm:dirichlet} and discuss the relations between Dirichlet improvability and bad approximability in the Gaussian and Eisenstein cases. In Section \ref{sec: measure estimates}, we prove the desired measure estimates in Theorem \ref{thm: measure estimates} for both the Gaussian and Eisenstein cases. In Section \ref{sec: effective equidistribution}, we recall the effective single equidistribution theorem (see Theorem \ref{thm: single equidistribution}), and give a self-contained proof of the effective double equidistribution theorem (Theorem \ref{thm:effdoubequi}). In Section \ref{sec: proof of zero-one laws}, we finish the proof of the zero-one laws in Theorem \ref{thm: zero-one laws}. In particular, {we obtain the desired disjointness statements in Proposition \ref{prop:disjoint} which are then used to handle the divergence cases of the zero-one laws.}

\vspace{5pt}

\noindent\textbf{AI Declaration.}
The first version of the manuscript was produced by the human authors and was first uploaded to arXiv in July 2026. ChatGPT (OpenAI, GPT-5.6 Sol and GPT-6 Astra) was subsequently used in several rounds of proofreading and review of the mathematical arguments. In addition to suggestions on exposition, it helped identify points requiring correction or further justification and proposed revisions to certain statements and arguments. In particular, the current formulation of the disjointness statement was suggested by ChatGPT to account for the fact that the constructed target sets are not nested; the original formulation was not sufficient to obtain the estimates needed in the proof. The authors independently evaluated and verified the suggested revisions, made the final decisions, and take full responsibility for the content and correctness of the paper.

\vspace{5pt}

\noindent{\bf Acknowledgements.} H.Z. would like to thank Yubin He for valuable discussions on Hurwitz continued fraction theory.

\section{Geometry of numbers for imaginary quadratic fields}\label{sec: geometry of numbers}

Let $K=\Q(\sqrt{-d})$ and let $\cO_K$  be its ring of integers, that is, $\cO_K=\Z[\sigma]$, where  $\sigma=\sqrt{-d}$ if $d\equiv 1, 2 \Mod{4}$ and $\sigma=\frac{1+\sqrt{-d}}{2}$ if $d\equiv 3\Mod{4}$.  Let $G=\SL_2(\C)$ and $\G_K=\SL_2(\cO_K)$. {For later purposes, we introduce the following group elements in $G$. 
Let 
\begin{align}\label{equ:gsdef}
g_s:=\begin{pmatrix}
        e^s & 0\\
         0 & e^{-s}
    \end{pmatrix},\qquad  (s\in \R)
\end{align}
and let 
\begin{align}\label{equ:uzpm}
    u(z):=\begin{pmatrix}
        1 & z\\
         0 & 1
    \end{pmatrix},\quad u^-(z):=\begin{pmatrix}
        1 & 0\\
         z & 1
    \end{pmatrix},\qquad (z\in \C).
\end{align}}

The homogeneous space $X:=G/\G_K$ parameterizes the following space
\begin{align*}
\left\{g\cO_K^2: g\in G\right\},
\end{align*}
which is the space of {free, discrete, cocompact} $\cO_K$-submodules of $\C^2$ that admit an ordered basis in $\C^2$ with determinant $1$. {Using the identification between $\C^2$ and $\R^4$, elements of $X$ can be viewed as lattices in $\R^4$. More precisely, as a lattice in $\R^4$, the integral lattice $\cO_K^2$ has a basis 
$$\left\{(1,0,0,0)^t, (\Re(\sigma),\Im(\sigma), 0,0)^t, (0,0,1,0)^t, (0,0,\Re(\sigma),\Im(\sigma))^t\right\}.
$$ 
Since elements of $G$ preserve the covolume,  we have 
\begin{align}\label{equ:covol}
\mathrm{covol}(\Lambda)=\mathrm{covol}(\cO_K^2)=\Im(\sigma)^2=\frac{|D_K|}{4},\qquad \forall\, \Lambda\in X.
\end{align}
Here $D_K$ is the \textit{discriminant} of $K$; namely, $D_K=-4d$ if $d\equiv 1, 2\Mod{4}$ and $D_K=-d$ if $d\equiv 3\Mod{4}$.} Further note that $X$ is noncompact in view of  Mahler's compactness criterion in this setting \cite[Theorem~1.1]{KleinbockShiTomanov17}.

Let $\|\cdot\|$ be the supremum norm on $\C^2$. For any $r>0$, let
$$
\B_{\C^2}(r):=\{\bm{v}\in \C^2: \|\bm{v}\|<r\}.
$$ 
For any lattice $\Lambda\in X$, define 
\begin{align*}
\sys(\Lambda):=\min\{\|\bm{v}\|: \bm{v}\in \Lambda\smallsetminus\{\bm{0}\}\}. 
\end{align*}
The  \textit{critical radius} of the space $X$ with respect to the supremum norm is defined by 
\begin{align}\label{def:rklambda}
r_K:=\inf\{r>0: \Lambda\cap \B_{\C^2}(r)\neq \{\bm{0}\},\, \forall\, \Lambda\in X\}=\sup\{\sys(\Lambda): \Lambda\in X\}.
\end{align}
Thus, if $r> r_K$ then $B_{\C^2}(r)$ intersects every lattice $\Lambda \in X$ in a nonzero point.
On the other hand, if $r < r_K$ there must exist some $\Lambda \in X$ for which $\Lambda \cap B_{\C^2}(r) = \{\bm{0}\}$.
The corresponding \textit{critical locus} is defined as 
\begin{align}\label{def:criticallocus}
\mathfrak{C}_K:=\{\Lambda\in X: \Lambda\cap \B_{\C^2}(r_K)=\{\bm{0}\}\}. 
\end{align}
It is nonempty and compact by Mahler's criterion. We define the continuous function
\begin{align}\label{def:deltaalt}
\Delta: X\to [0,\infty),\qquad \Lambda\mapsto 
\log\left(\tfrac{r_K}{\text{sys}(\Lambda)}\right).
\end{align}
We will see below that $c_K=r_K^2$ (see \eqref{equ:costrel}); thus this definition is equivalent to the definition given in \eqref{def:delta}. By definition we have $\Delta^{-1}\{0\}=\mathfrak{C}_K$. Note also that for any $r>0$,
\begin{align}\label{equ:svec}
\Delta^{-1}[0,r]=\{ \Lambda\in X: \sys(\Lambda)\geq r_Ke^{-r}\}.
\end{align}

\begin{Lem}\label{lem:criticallocus}
The map $\Delta: X\to [0,\infty)$ is proper. In particular, for any $r\geq 0$, $\Delta^{-1}[0,r]$ is a {nonempty} compact subset of $X$. 
\end{Lem}

\begin{proof}
Since $\Delta$ is continuous, for any $r\geq 0$, the set $\Delta^{-1}[0,r]$ is a closed subset of $X$. Moreover, by Mahler's compactness criterion in this setting \cite[Theorem~1.1]{KleinbockShiTomanov17} and the expression \eqref{equ:svec} we see that $\Delta^{-1}[0,r]$ is also a precompact subset of $X$. Thus $\Delta^{-1}[0,r]$ is a compact subset of $X$. 

To show nonemptiness, we only need to show $\mathfrak{C}_K=\Delta^{-1}\{0\}$ is nonempty; {the nonemptiness of $\Delta^{-1}[0,r]$ for $r>0$ follows from \eqref{equ:svec} and the definition of $r_K$. } Take a decreasing sequence $\{r_n\}_{n\in \N}$ of positive numbers  with $\lim_{n\to\infty}r_n=0$. For each $n\in \N$, take $\Lambda_n\in \Delta^{-1}[0, r_n]$. By compactness of $\Delta^{-1}[0, r_1]$, by passing to a subsequence we have $\lim_{n\to\infty}\Lambda_n=\Lambda$ for some $\Lambda\in \Delta^{-1}[0, r_1]$. Since $\Delta$ is continuous, we have $\Delta(\Lambda)=\lim_{n\to\infty}\Delta(\Lambda_n)\leq \lim_{n\to\infty}r_n=0$. This implies that $\Lambda\in \Delta^{-1}\{0\}$ and thus $\Delta^{-1}\{0\}$ is nonempty. 
\end{proof}

\subsection{Critical locus for Gaussian and Eisenstein integers}\label{sec:criloc}
In this section we give the explicit description of the critical locus when $K=\Q(\sqrt{-1})$ and $K=\Q(\sqrt{-3})$. 
Let 
\begin{align}\label{equ:defcR}
\cR:=\left\{R_{\theta}:=\begin{pmatrix}
e^{i\theta} & 0\\
0 &e^{-i\theta}\end{pmatrix}: \theta\in [0, 2\pi)\right\},
\end{align}
and let 
$$
w_0:=\begin{pmatrix} 0 & 1\\ -1 & 0\end{pmatrix}.
$$ 
Note that since $\|\cdot\|$ is invariant under  left multiplication by $R_{\theta}$ and $w_0$,  it is invariant under the (disconnected) subgroup 
$
\tilde{\cR}:=\la \cR, w_0\ra
$
generated by $\cR$ and $w_0$. Hence any $\Delta$-level set is also $\tilde{\cR}$-invariant. In particular, $\mathfrak{C}_K=\Delta^{-1}\{0\}$ is  $\tilde{\cR}$-invariant. 
\begin{Thm}[{ \cite[Chapter 6]{MinkowskiDA}}]\label{thm:crilocus}
\begin{enumerate}
    \item[(i)]
If $K=\Q(\sqrt{-1})$, then we have
\begin{align}\label{equ:-1}
r_K=\sqrt{\tfrac{\sqrt{2}}{3-\sqrt{3}}}= 1.0561 \cdots\quad \text{and}\quad \mathfrak{C}_K=\tilde{\cR} h_0 \mathcal{O}_K^2, 
\end{align}
where 
\begin{align} \label{eq:h_0}
h_0:=r_K\begin{pmatrix}
1 &  \omega\\
-i\omega e^{-\frac{\pi}{12}i} & e^{-\frac{\pi}{12}i}\end{pmatrix},\qquad \text{with $\omega=\frac12+\left(1-\tfrac{\sqrt{3}}{2}\right)i$.}
\end{align}
\item[(ii)]
If $K=\Q(\sqrt{-3})$, then we have
\begin{align}\label{equ:-3}
r_K=1\quad \text{and}\quad \mathfrak{C}_K=\tilde{\cR} \left\{ u(z)\mathcal{O}_K^2: z\in \cF\right\},
\end{align}
where $\cF\subset \C$ is a closed fundamental domain for $\C/\mathcal{O}_K$, with boundary identifications understood, and $u(z)$ is as given in \eqref{equ:uzpm}.
\end{enumerate}
\end{Thm}

\begin{proof}
This was essentially proved by Minkowski in~\cite[Satz~LXII and LXIV, pp.~218 and 233]{MinkowskiDA}, but some notation and conventions require clarification.
   Let $K$ be an imaginary quadratic field. Minkowski considered the space of free cocompact discrete $\cO_K$-submodules of $\C^2$ which can be parameterized by the homogeneous space $\GL_2(\C)/\GL_2(\cO_K)$ via $g\GL_2(\cO_K)\mapsto g\cO_K^2$. Following Minkowski's notation, we say $\Lambda\in \GL_2(\C)/\GL_2(\cO_K)$ is \textit{admissible} if $\Lambda\cap \B_{\C^2}(1)=\{\bm{0}\}$. Define
    \begin{align*}
\cV_K:=\inf\{\mathrm{covol}(\Lambda): \text{$\Lambda\in \GL_2(\C)/\GL_2(\cO_K)$ is admissible}\}
        \end{align*}
and
\begin{align*}
\tilde{\mathfrak{C}}_K:=\{\Lambda \in \GL_2(\C)/\GL_2(\cO_K): \text{$\Lambda$ is admissible and $\mathrm{covol}(\Lambda)=\cV_K$}\}.
\end{align*}
 Here the covolume is computed viewing $\Lambda\in \GL_2(\C)/\GL_2(\cO_K)$ as a lattice in $\R^4$ via the natural identification between $\C^2$ and $\R^4$.
 Let $\cO_K^{\times}$ be the set of units in $\cO_K$.  Using the observation that lattices in $X=\SL_2(\C)/\SL_2(\cO_K)$ are of covolume $|D_K|/4$ with $D_K$ the discriminant of $K$ {(see \eqref{equ:covol})}, and the fact that lattices $\Lambda=g\cO_K^2\in \GL_2(\C)/\GL_2(\cO_K)$  are representable as elements in $X$ if and only if $\det(g)\in \cO_K^{\times}$, we see that 
 these two objects are related to our critical radius and locus via the following relations:

 \begin{align*}
     r_K=\left(\frac{|D_K|}{4\cV_K}\right)^{1/4}
 \quad
 \text{and} 
 \quad
     \mathfrak{C}_K=\left\{r_K\Lambda: \Lambda=g\cO_K^2\in \tilde{\mathfrak{C}}_K,\, {r_K^2}\det(g)\in \cO_K^{\times} \right\}.
 \end{align*}
 
 Now for the case when $K=\Q(\sqrt{-1})$, Minkowski computed that $\cV_K=\left(\frac{3-\sqrt{3}}{\sqrt{2}}\right)^2$ and 
 \begin{align}\label{eq:MinkowskiGLlocus}
     \tilde{\mathfrak{C}}_K=   \left\lbrace \begin{pmatrix} e^{is} & 0 \\ 0 & e^{it} \end{pmatrix}
    M\mathcal{O}_K^2 : s, t \in \mathbb{R}  \right\rbrace,\qquad \text{with $M:=\begin{pmatrix} 1 & -i-j^2 \\ 1-ij^2 & 1 \end{pmatrix}$}.
 \end{align}
 Here $j:=e^{\frac{2\pi}{3}i}=-\frac12+\frac{\sqrt{3}}{2}i$.
 Note also that in this case $D_K=-4$ and $\cO_K^{\times}=\{\pm 1,\pm i\}$ and so $r_K=\cV_K^{-1/4}=\sqrt{\tfrac{\sqrt{2}}{3-\sqrt{3}}}$ as claimed in \eqref{equ:-1}. Moreover, define 
 $$
 g_{s,t}=r_K\begin{pmatrix} e^{it} & 0\\ 0 & e^{it} 
\end{pmatrix} \begin{pmatrix} e^{is} & 0\\ 0 & e^{-is} 
\end{pmatrix}M,\qquad (s,t\in \R).
 $$
Then we have 
 \begin{align*}
     \mathfrak{C}_K=\left\{g_{s,t}\cO_K^2: s,t\in \R,\, \det(g_{s,t})\in \{\pm 1,\pm i\}\right\}.
 \end{align*}
 To further compute the set $\mathfrak{C}_K$, note that by a direct computation we have 
 $
 \det (M)= r_K^{-2}e^{-\frac{\pi}{12}i}$
 and thus 
 $$
 \det(g_{s,t})=e^{2it-\frac{\pi }{12}i},\qquad (s,t\in \R).
 $$ 
 Thus in order to have $\det(g_{s,t})\in \{\pm 1,\pm i\}=\{e^{\frac{k\pi}{2}i}: k=0,1,2,3\}$, we need {$2t-\frac{\pi}{12}\equiv \frac{k\pi}{2}\Mod{2\pi}$} for some $k=0,1,2,3$, i.e. {$t\equiv \frac{\pi}{24}+\frac{k\pi}{4}\Mod{\pi}$} for some $k=0,1,2,3$. We thus have 
\begin{align*}
\mathfrak{C}_K=\bigcup_{k=0}^3 \cR r_Ke^{(\frac{\pi}{24}+\frac{k\pi}{4})i} M\mathcal{O}_K^2=\bigcup_{k=0}^1 \cR r_Ke^{(\frac{\pi}{24}+\frac{k\pi}{4})i} M\mathcal{O}_K^2.
\end{align*}
Here for the second identity we used that $e^{\frac{\pi}{2} i}\cO_K^2=i\cO_K^2=\cO_K^2$.
Now note that 
\begin{align*}
    e^{\frac{\pi}{4}i}\begin{pmatrix} e^{-\frac{\pi}{4}i} & 0\\ 0 & e^{\frac{\pi}{4}i}\end{pmatrix}M\begin{pmatrix} 1& 0\\ 0 &-i\end{pmatrix}=\begin{pmatrix} 1 & 0\\ 0 & i\end{pmatrix}M\begin{pmatrix} 1& 0\\ 0 &-i\end{pmatrix}= w_0Mw_0^{-1}.
\end{align*}
From this we see that 
\begin{align*}
    \mathfrak{C}_K=\tilde{\cR}r_Ke^{\frac{\pi}{24}i} M\mathcal{O}_K^2.
\end{align*}
Finally, recall $h_0$ is given as in \eqref{eq:h_0}. One verifies the following identity
\begin{align*}
h_0=r_Ke^{\frac{\pi}{24}i}\begin{pmatrix}e^{\frac{\pi}{8}i} & 0\\ 0 & e^{-\frac{\pi}{8}i}\end{pmatrix} M\begin{pmatrix} 1& 0 \\
-i & 1\end{pmatrix},
\end{align*}
from which we get
\begin{align*}
\mathfrak{C}_K=\tilde{\cR} {r_K} e^{\frac{\pi}{24}i}M\mathcal{O}_K^2=\tilde{\cR} h_0\cO_K^2
\end{align*}
as desired.

For $K=\Q(\sqrt{-3})$, Minkowski computed that $\cV_K=\frac{3}{4}$ and 
\begin{align*}
    \tilde{\mathfrak{C}}_K= \{e^{it}Ru(z)\cO_K^2: t\in \R,\ R\in \tilde{\cR},\ z\in \cF\}
\end{align*}
with $\cF\subset \C$ a fundamental domain for $\C/\cO_K$. Note that in this case $D_K=-3$ and hence $r_K=1$. {We then get 
\begin{align*}
    \mathfrak{C}_K=\{e^{it}Ru(z)\cO_K^2: e^{2it}\in \cO_K^{\times},\ R\in \tilde{\cR},\ z\in \cF\}.
\end{align*} 
For any $t\in \R$ with $e^{2it}\in \cO_K^{\times}$ and for any $z\in \C$ and $R\in \tilde{\cR}$, using that $\begin{pmatrix}
   1 & 0\\ 
    0 &  e^{-2it}
\end{pmatrix}\cO_K^2=\cO_K^2$ we have
\begin{align*}
    e^{it}Ru(z)\cO_K^2=Ru(z)R_t\cO_K^2=RR_tu(e^{-2it}z)\cO_K^2.
\end{align*}
Here recall that $R_t=\begin{pmatrix}
e^{i t} & 0\\
0 &e^{-i t}\end{pmatrix}$ is as in \eqref{equ:defcR}.
Since $e^{-2it}\in \cO_K^{\times}$, the map $z\mapsto e^{-2it}z$ sends a fundamental domain for $\C/\cO_K$ to another fundamental domain for $\C/\cO_K$. Together with the above description of $\mathfrak{C}_K$ and the fact that $RR_t\in \tilde{\cR}$, this gives the second half of \eqref{equ:-3} as desired.}
\end{proof}

For later use, it will be more convenient to describe the critical locus as orbits of the connected group $\cR$. 

When $K=\Q(\sqrt{-1})$, we have
\begin{align}\label{equ:ckr-1}
\mathfrak{C}_K=\cR h_0 \mathcal{O}_K^2 \bigcup \cR h_1 \mathcal{O}_K^2,
\end{align}
with
\begin{align}\label{equ:h1}
    h_1  = w_0 h_0 w_0^{-1} = 
    r_K\begin{pmatrix} e^{-\frac{i\pi}{12}} & i\omega e^{-\frac{i\pi}{12}}\\  -\omega & 1 
\end{pmatrix}.
\end{align}
One can verify that the above union is in fact disjoint.

When $K=\Q(\sqrt{-3})$, we have
\begin{align}\label{equ:ckr-3}
\mathfrak{C}_K=\cR\left\{ u(z)\mathcal{O}_K^2 : z\in \cF\right\} \cup \cR \left\{ u^-(z) \mathcal{O}_K^2 : z\in \cF\right\},
\end{align}
where $u^-(z)=\begin{pmatrix} 1 & 0\\ z & 1\end{pmatrix}$ is as in \eqref{equ:uzpm}.

We have the following remarkable property showing that every lattice in $\mathfrak{C}_{\Q(\sqrt{-1})}$  lies on some periodic $\{g_s\}$-orbit.
\begin{Lem}\label{rmk:cloges}
    Let $K=\Q(\sqrt{-1})$. {Then for any $\Lambda\in \mathfrak{C}_K$ we have 
    \begin{align}\label{equ:cloges}
        g_{\tau}\Lambda=\Lambda,\qquad \text{with $\tau=\log(2+\sqrt{3})$}.
    \end{align}}
\end{Lem}
\begin{proof}
In view of \eqref{equ:ckr-1}, since $\{g_s\}_{s\in \R}$ commutes with $\cR$, it suffices to check \eqref{equ:cloges} for $\Lambda=h_0\cO_K^2$ and for $\Lambda=h_1\cO_K^2$.
Let $\tau=\log(2+\sqrt{3})$. When $\Lambda=h_0\cO_K^2$,  we have the following matrix identity:
    \begin{align*}
\begin{pmatrix} 2+\sqrt{3} & \\
& 2-\sqrt{3}\end{pmatrix} h_0=h_0\begin{pmatrix}
4-i & 2\\
2i & i\end{pmatrix}.
\end{align*}
Since $\begin{pmatrix}
4-i & 2\\
2i & i\end{pmatrix}\in \G_K$, this implies  that $g_{\tau}h_0\cO_K^2=h_0\cO_K^2$. When $\Lambda=h_1\cO_K^2$, using the relations $h_1=w_0h_0w_0^{-1}$, $w_0^{-1}g_sw_0=g_{-s}$ and $w_0\cO_K^2=\cO_K^2$, we get
\begin{align}\label{equ:h1qperiod}
    g_{-\tau}h_1\cO_K^2=g_{-\tau}w_0h_0w_0^{-1}\cO_K^2=w_0g_{\tau}h_0\cO_K^2=w_0h_0\cO_K^2=h_1\cO_K^2.  
\end{align}
Multiplying both sides above  by $g_{\tau}$ we get $g_{\tau} h_1\cO_K^2=h_1\cO_K^2$ as desired. 
\end{proof}

When $K=\Q(\sqrt{-3})$, {we have $\cO_K=\Z[j]$ and} we will take 
\begin{align}\label{equ:fdeis}
\cF:=\left\{z\in \C: |z|\leq  \inf_{v\in \Z[j]\smallsetminus\{0\}}|z-v|\right\}
\end{align}
to be the closed fundamental domain for $\C/\cO_K$. More explicitly, $\cF$ is a closed \textit{Dirichlet region} for $\Z[j]$ (viewed as a discrete subgroup of $\C$), namely,  the closed  regular hexagon centered at $0$ with vertices $\pm \frac{\sqrt{3}}{3}i$, $\pm \frac12\pm \frac{\sqrt{3}}{6}i$. It has the following simple but crucial property:  
\begin{align}\label{equ:keyprop}
|z|\leq \frac{\sqrt{3}}{3}\quad  \text{for all $z\in \cF$}.
\end{align}

\subsection{General bounds on \texorpdfstring{$r_K$}{rK}}
For a general imaginary quadratic field $K$, we have the following lower and upper bounds on $r_K$. 
\begin{Prop}\label{prop:bdonrk}
Let $d\neq 1,3$ be a positive square-free integer and let $K=\Q(\sqrt{-d})$ be the corresponding imaginary quadratic field. Let $D_K$ be the discriminant of $K$.  Let $b_d:=\frac12+\frac{\sqrt{d}}{2}i$ if $d\equiv 1, 2\Mod{4}$ and $b_d:=\frac12+\frac{d-1}{4\sqrt{d}}i$ if $d\equiv 3\Mod{4}$. Let 
$$
\kappa_d:=\inf\{|x|: x\in \cO_K\smallsetminus\{\pm 1, 0\}\}. 
$$ 
 We then have 
 \begin{align}\label{equ:uppbdrk}
\max\left\{\sqrt{\tfrac{|D_K|}{4+2|D_K|^{1/2}}}, \sqrt{\kappa_d \min\{|b_d|,1\}}\right\}\leq r_K\leq \sqrt{\frac{2}{\pi}}|D_K|^{1/4}.
\end{align}
\end{Prop}

\begin{remark}
Note that by our assumption that $d\neq 1,3$, we have $\cO_K^{\times}=\{\pm 1\}$. Thus $\kappa_d$ is the smallest absolute value of a nonzero non-unit element in $\cO_K$. More precisely, we have $\kappa_{2}=\kappa_{7}=\sqrt{2}$ with the minimal value obtained by $\sqrt{2}i$ and $\frac{1+\sqrt{7}i}{2}$ respectively, $\kappa_{11}=\sqrt{3}$ with the minimal value obtained by $\frac{1+\sqrt{11}i}{2}$; for all other cases $\kappa_d=2$ with the minimal value obtained by $2$. 
\end{remark}
\begin{remark}
Let $f_d:=\max\{\sqrt{\tfrac{|D_K|}{4+2|D_K|^{1/2}}}, \sqrt{\kappa_d \min\{|b_d|,1\}}\}$ be the above lower bound. {(See Table \ref{tab:hd_values} for the first $10$ values of $f_d$.)} Note that $|b_d|<1$ only when $d=2, 7, 11$, with $|b_{2}|=\frac{\sqrt{3}}{2}\approx 0.866$, $|b_{7}|=\frac{2}{\sqrt{7}}\approx 0.7559$ and $|b_{11}|=\frac{3}{\sqrt{11}}\approx 0.9045$. Thus we have 
$$
r_K\geq \sqrt{\kappa_d \min\{|b_d|,1\}}\geq \sqrt{\sqrt{2}|b_{7}|}=\frac{2^{3/4}}{7^{1/4}}= 1.03395\cdots>1.
$$
This, combined with Theorem \ref{thm:crilocus},  implies that except for the Eisenstein  case, the standard integral lattice $\cO_K^2$ (or any shearing of it by upper or lower triangular unipotent matrices in $G$) never lands in the critical locus. 

On the other hand, recall $D_K=-d$ if $d\equiv 3\Mod{4}$ and $D_K=-4d$ if $d\equiv 1, 2\Mod{4}$. Since $\sqrt{\kappa_d \min\{|b_d|,1\}}$ is bounded, and $\sqrt{{\tfrac{|D_K|}{4+2|D_K|^{1/2}}}}\sim \frac{\sqrt{2}}{2}|D_K|^{1/4}\to\infty$ as $d\to\infty$, we have $f_d=\sqrt{{\tfrac{|D_K|}{4+2|D_K|^{1/2}}}}$ for all $d$ sufficiently large. Indeed, by a direct inspection this identity holds  for all $d\geq 31$. We  then have the following asymptotic estimate on $r_K$:
\begin{align*}
0.7071\cdots=\frac{\sqrt{2}}{2}\leq \liminf_{d\to\infty}\frac{r_K}{|D_K|^{1/4}}\leq \limsup_{d\to\infty}\frac{r_K}{|D_K|^{1/4}}\leq \sqrt{\frac{2}{\pi}}=0.7978\cdots.
\end{align*}

\begin{table}[ht]
\centering
\begin{tabular}{c|c|c|c}
$d$ & $f_d$ (exact) & $f_d$ (approx) & Maximising term \\
\hline
2 & $\sqrt{\frac{\sqrt{6}}{2}}$ & 1.1067 & second \\
5 & $\sqrt{2}$ & 1.4142 & second \\
6 & $\sqrt{2}$ & 1.4142 & second \\
7 & $\frac{2^{3/4}}{7^{1/4}}$ & 1.0339 & second \\
10 & $\frac{\sqrt{10}}{\sqrt{1+\sqrt{10}}}$ & 1.5500 & first \\
11 & $\frac{3^{3/4}}{11^{1/4}}$ & 1.2517 & second \\
13 & $\frac{\sqrt{13}}{\sqrt{1+\sqrt{13}}}$ & 1.6801 & first \\
14 & $\frac{\sqrt{14}}{\sqrt{1+\sqrt{14}}}$ & 1.7183 & first \\
15 & $\sqrt{2}$ & 1.4142 & second \\
17 & $\frac{\sqrt{17}}{\sqrt{1+\sqrt{17}}}$ & 1.8216 & first
\end{tabular}
\caption{Values of $f_d$ for selected $d$. ``first'' refers to $\sqrt{{\tfrac{|D_K|}{4+2|D_K|^{1/2}}}}$, ``second'' refers to $\sqrt{\kappa_d\min\{|b_d|,1\}}$.}
\label{tab:hd_values}
\end{table}
\end{remark}

\begin{proof}[Proof of Proposition \ref{prop:bdonrk}]
{Recall from \eqref{equ:covol} that for any $\Lambda=g\cO_K^2\in X$ with $g\in G$ we have $\mathrm{covol}(\Lambda)=|D_K|/4$. } By direct computation we also have $\vol(\B_{\C^2}(r))=\pi^2r^4$. By Minkowski's convex body theorem, $\Lambda\cap \B_{\C^2}(r)\neq \{\bm{0}\}$ as long as 
$$
\vol(\B_{\C^2}(r))>2^4\mathrm{covol}(\Lambda)=4|D_K|\quad \Leftrightarrow \quad r>\sqrt{\frac{2}{\pi}}|D_K|^{1/4}.
$$
This shows that $r_K\leq \sqrt{\frac{2}{\pi}}|D_K|^{1/4}$ as desired. 

Next, we prove the lower bound $r_K\geq\sqrt{ \frac{|D_K|}{4+2|D_K|^{1/2}}}$. Let $\delta_d=1$ if $d\equiv 1, 2\Mod{4}$ and let $\delta_d=\frac12$ if $d\equiv 3\Mod{4}$. Let $\lambda_d=\frac{\delta_d\sqrt{d}}{1+\delta_d\sqrt{d}}$. Note that $\delta_d\sqrt{d}=\frac{1}{2}|D_K|^{1/2}$, thus it suffices to show $r_K\geq\sqrt{\lambda_d\delta_d}d^{1/4}$. For this, it suffices to find a lattice $\Lambda\in X$  such that $\sys(\Lambda)\geq \sqrt{\lambda_d\delta_d}d^{1/4}$. 
We take $\Lambda_1=\mathrm{g}_1\cO_K^2$ with
$$
\mathrm{g}_1=\left(\begin{smallmatrix}
\sqrt{\delta_d/\lambda_d}d^{1/4} & \sqrt{\delta_d\lambda_d} d^{1/4} i\\ 0 & \sqrt{\lambda_d/\delta_d}d^{-1/4}\end{smallmatrix}\right).
$$ 
We first prove the following claim:  We have 
\begin{align}\label{equ:lobd}
\inf_{(x,y)\in \cO_K^2\smallsetminus\{\bm{0}\}, |y|<\delta_d\sqrt{d}}|x+i\lambda_dy|\geq \lambda_d.
\end{align}
The proof of this claim is based on the following simple observation:
\begin{align*}
\forall\, x\in \cO_K\ :\  |x|<\delta_d\sqrt{d}\quad \Rightarrow\quad x\in \cO_K\cap \R= \Z.
\end{align*}
This observation is true since otherwise $x$ would have a non-trivial imaginary part which has absolute value at least $\delta_d\sqrt{d}$. By this observation, we see that the $y$ in \eqref{equ:lobd} must be an integer. 
We now prove \eqref{equ:lobd}.  If $|x|\geq \delta_d\sqrt{d}$, then we have for any $y\in \cO_K$ with $|y|<\delta_d\sqrt{d}$, $|x+i\lambda_dy|\geq |x|-\lambda_d|y|\geq \delta_d\sqrt{d}(1-\lambda_d)=\lambda_d$ as desired. Next, if $|x|< \delta_d\sqrt{d}$, then again by the observation we also have $x\in \Z$. Thus for any $(x,y)\in \cO_K^2\smallsetminus\{\bm{0}\}$ with $\max\{|x|,|y|\}<\delta_d\sqrt{d}$, $(x,y)$ is a nonzero integer vector, and thus
\begin{align*}
|x+i\lambda_dy|=\sqrt{x^2+\lambda_d^2y^2} \geq \lambda_d.
\end{align*}
This proves \eqref{equ:lobd}. Now we show that $\|\mathrm{g}_1\bm{v}\|\geq \sqrt{\lambda_d\delta_d}d^{1/4}$ for any nonzero $\bm{v}\in \cO_K^2$. Take any nonzero $\bm{v}=\begin{pmatrix} x \\ y\end{pmatrix}\in \cO_K^2$.
Note that 
\begin{align*}
\mathrm{g}_1\begin{pmatrix} x \\ y\end{pmatrix}=\begin{pmatrix}
\sqrt{\delta_d/\lambda_d}d^{1/4}(x+i\lambda_d y) \\ \sqrt{\lambda_d/\delta_d}d^{-1/4}y\end{pmatrix}. 
\end{align*}
If $|y|\geq \delta_d\sqrt{d}$, then we have $\|\mathrm{g}_1\begin{pmatrix} x \\ y\end{pmatrix}\|\geq \sqrt{\lambda_d/\delta_d}d^{-1/4}|y|\geq \sqrt{\lambda_d\delta_d} d^{1/4}$. If $|y|<\delta_d\sqrt{d}$, then by \eqref{equ:lobd} we have $\|\mathrm{g}_1\begin{pmatrix} x \\ y\end{pmatrix}\|\geq \sqrt{\delta_d/\lambda_d}d^{1/4}|x+i\lambda_dy|\geq \sqrt{\lambda_d\delta_d}d^{1/4}$. This finishes the proof of the lower bound $r_K\geq \sqrt{\lambda_d\delta_d}d^{1/4}$. 

Next, we prove the second lower bound $r_K\geq \sqrt{\kappa_d \min\{|b_d|,1\}}$, where $b_d$ and  $\kappa_d$  are as in this proposition. Let $a_d=\sqrt{\frac{\kappa_d}{\min\{|b_d|,1\}}}$. We take $\Lambda_2:=\mathrm{g}_2\cO_K^2$ with
$$
\mathrm{g}_2=\begin{pmatrix} 
a_d & a_db_d\\
0 & a_d^{-1}\end{pmatrix}.
$$
It suffices to show that $\sys(\Lambda_2)\geq \sqrt{\kappa_d \min\{|b_d|,1\}}$. 
We first claim that 
\begin{align*}
\inf_{x\in \cO_K}|x+b_d|=\inf_{x\in \cO_K}|x-b_d|= |b_d|=\left\{\begin{array}{ll} 
	\frac{\sqrt{d+1}}{2}  & d\equiv 1,2\Mod{4},\\
\frac{d+1}{4\sqrt{d}} &  d\equiv 3\Mod{4}.
\end{array}\right.
\end{align*}
The first equality is clear since $-\cO_K=\cO_K$. For the second equality, when $d\equiv 1, 2\Mod{4}$ so that $\cO_K=\Z[\sqrt{d}i]$ and $b_d=\frac12+\frac{\sqrt{d}}{2}i$,  the infimum is attained by taking $x\in \{0,1, \sqrt{d}i, 1+\sqrt{d}i\}$, and for every other choice of $x\in \cO_K$, one verifies that $|x-b_d|>|b_d|$. When $d\equiv 3\Mod{4}$ so that $\cO_K=\Z[\frac{1+\sqrt{d}i}{2}]$ and $b_d=\frac12+\frac{d-1}{4\sqrt{d}}i$, the infimum is attained by taking $x\in \{0, 1, \frac{1+\sqrt{d}i}{2}\}$ and for every other choice of $x\in \cO_K$, we have $|x-b_d|>|b_d|$. 

Now take any nonzero $\bm{v}=\begin{pmatrix} x \\ y\end{pmatrix}\in \cO_K^2$.
Note that 
\begin{align*}
\mathrm{g}_2\begin{pmatrix} x \\ y\end{pmatrix}&=\begin{pmatrix}
a_d(x+b_d y) \\ a_d^{-1}y\end{pmatrix}. 
\end{align*}
If $y=0$, we have $\|\mathrm{g}_2\bm{v}\|\geq a_d\geq a_d\min\{|b_d|,1\}=\sqrt{\kappa_d \min\{|b_d|,1\}}$. If $y\in \cO_K^{\times}=\{\pm 1\}$ we have
$\|\mathrm{g}_2\bm{v}\|\geq a_d|x\pm b_d|\geq a_d|b_d|\geq a_d\min\{|b_d|,1\}=\sqrt{\kappa_d \min\{|b_d|,1\}}$. If $y\in \cO_K\smallsetminus\{0,\pm1\}$, then we have $|y|\geq \kappa_d$ and thus $\|\mathrm{g}_2\bm{v}\|\geq a_d^{-1}\kappa_d=\sqrt{\kappa_d \min\{|b_d|,1\}}$. This finishes the proof of the lower bound $\sys(\Lambda_2)\geq \sqrt{\kappa_d \min\{|b_d|,1\}}$, and hence also $r_K\geq \sqrt{\kappa_d \min\{|b_d|,1\}}$.
\end{proof}

\subsection{Proof of Theorem \ref*{thm:dirichlet}}\label{sec:thm11}
In this section we prove Theorem \ref{thm:dirichlet}, the Dirichlet-type theorem in our setting. We first introduce some notation. Let $K$ be an imaginary quadratic field. Let $X=G/\G_K=\SL_2(\C)/\SL_2(\cO_K)$ be the homogeneous space as before. For any $s\in \R$ and $z\in \C$ recall $g_s$ and $u(z)$ are given in \eqref{equ:gsdef} and \eqref{equ:uzpm} respectively. For any $z\in \C$, let 
\begin{align}\label{equ:lambdazlattice}
     \Lambda_z:=u(z)\cO_K^2=\begin{pmatrix} 1 & z\\ 0 & 1\end{pmatrix}\cO_K^2\in X. 
\end{align}
Denote also by $\mu_K$ the probability $G$-invariant measure on $X$. Locally,  $\mu_K$ agrees with a Haar measure of $G$. Moreover, it is well-known that the diagonal flow $\{g_s\}_{s\in \R}$ acts ergodically on the probability space $(X,\mu_K)$.
\begin{proof}[Proof of Theorem \ref{thm:dirichlet}]
Recall that the critical radius is defined in \eqref{def:rklambda}.
We will prove Theorem \ref{thm:dirichlet} by showing that 
\begin{align}\label{equ:costrel}
c_K=r_K^2. 
\end{align}

We first show for any $z\in \C$, the system of inequalities 
\begin{align}\label{equ:dithm}
|qz-p|\leq \frac{r_K^2}{t}, \qquad 0<|q|\leq t, 
\end{align}
is solvable in $(p,q)\in \cO_K^2$ for all $t>r^2_K $. Note that in view of the definition of $r_K$, for any $\Lambda\in X$ and for any $r>r_K$,  we have $\Lambda\cap \B_{\C^2}(r)\neq \{\bm{0}\}$. Taking $r\to r_K^+$ and using the discreteness of $\Lambda$ we see that $\Lambda\cap \overline{\B_{\C^2}(r_K)}\neq \{\bm{0}\}$. 
Now take any $z\in \C$, and for any $t> r^2_K$, we let $s=\log(t/r_K)$. Then $g_s\Lambda_z\cap \overline{\B_{\C^2}(r_K)}\neq \{\bm{0}\}$, where $\Lambda_z$ is as in \eqref{equ:lambdazlattice}. This implies that there exists $(p,q)\in \cO_K^2\smallsetminus \{\bm{0}\}$ such that
$$
\max\{e^{s}|qz-p|, e^{-s}|q|\}\leq r_K.
$$
Plugging in $s=\log(t/r_K)$ and using the fact that 
$
\inf\{|p|: p\in \cO_K\smallsetminus\{0\}\}\geq 1>\frac{r_K^2}{t},
$ 
we see that the above system of inequalities is equivalent to \eqref{equ:dithm}. {Finally, we show that for any $z\in \C$, the solvability of \eqref{equ:dithm} in $(p,q)\in \cO_K^2$ for all $t>r_K^2$ implies the  solvability of \eqref{equ:dirichlet} in $(p,q)\in \cO_K^2$ for all $t>r_K^2$. Indeed, take any $z\in \C$ and $t>r_K^2$. Take $t'>t$  so that the set $\{z\in \cO_K: t<|z|\leq t'\}$ is empty. Note that such $t'$ always exists since the set $\{|z|: z\in \cO_K\}$ is a discrete subset of $[0,\infty)$. Applying \eqref{equ:dithm} with $t'$, we can find a pair $(p,q)\in \cO_K^2$ satisfying $|qz-p|\leq r_K^2/t'$ and $0<|q|\leq t'$. We then have $|qz-p|\leq r_K^2/t'<r_K^2/t$. Moreover, since the set $\{z\in \cO_K: t<|z|\leq t'\}$ is empty, for any $q\in \cO_K$,  $0<|q|\leq t'$ is equivalent to $0<|q|\leq t$. We have thus shown that for any $t>r_K^2$, one can find a pair $(p,q)\in \cO_K^2$ satisfying \eqref{equ:dirichlet}.  }

Next, {we prove the second half of Theorem \ref{thm:dirichlet}.  Let $0<c<c_K$ be as in \eqref{eq: dirichlet impro}. Take any $c<c'<c_K$ and let $\psi(t)=\frac{c'}{t}$. To prove the second half of Theorem \ref{thm:dirichlet}}, it is easy to see that it suffices to show $\DI_K(\psi)$ is of zero Lebesgue measure. 
By definition \eqref{def:deltaalt} of $\Delta$, we have 
$$
z\in \mathbf{DI}_K(\psi)\quad \Leftrightarrow\quad \Delta(g_s\Lambda_z)>r_0:=\frac12\log\left(\frac{r_K^2}{c'}\right),\quad  \text{for all  sufficiently large $s$}.
$$
We now prove $\mathbf{DI}_K(\psi)$ is of zero Lebesgue measure.  Suppose not, then $\mathrm{Leb}(\mathbf{DI}_K(\psi))>0$. 
Note that we can find sufficiently small open neighborhoods  $U_0$ of $0$ and $U_1$ of $1$ in $\C$ satisfying the following property.
\begin{align}\label{equ:condiffcom}
\forall\, z_0\in U_0, z_1\in U_1, s>0, \Lambda\in X\,: \, \left|\Delta\left(g_s\left(\begin{smallmatrix} 1 & 0\\ z_0 & 1\end{smallmatrix}\right)\left(\begin{smallmatrix} z_1 & 0\\ 0 & z_1^{-1}\end{smallmatrix}\right)\Lambda\right)-\Delta(g_s\Lambda)\right|< \frac12 r_0. 
\end{align}

Choose a fundamental domain $\cF\subset \C$ for $\C/\cO_K$ and identify $\mathbf{DI}_K(\psi)$ with its set of representatives in $\cF$. Consider the following sets
\begin{align*}
C:=\left\{\begin{pmatrix} 1 & 0\\ z_0 & 1\end{pmatrix}\begin{pmatrix} z_1 & 0\\ 0 & z_1^{-1}\end{pmatrix}u(z): z_0\in U_0, z_1\in U_1, z\in \mathbf{DI}_K(\psi)\right\}\subset G, \quad \text{and}\quad \cC:=\pi(C)\subset X.
\end{align*} 
Here $\pi: G\to X$ is the natural projection map from $G$ to $X$.
Now note that under the local coordinates 
$$
g=\begin{pmatrix}
    1 & 0\\ z_0 & 1
\end{pmatrix} \begin{pmatrix}
    z_1 & 0\\ 0 & z_1^{-1}
\end{pmatrix}\begin{pmatrix}
    1 & z_2\\ 0 & 1
\end{pmatrix}\in G,\qquad  (z_0, z_2\in \C, z_1\in \C \smallsetminus \{0\})
$$
the Haar measure (up to scalars) of $G$ is given by 
$$
\dd m_G(g)=|z_1|^2\dd z_0\dd z_1\dd z_2,
$$
where $\dd z$ is the Lebesgue measure on $\C$. In view of this Haar measure description and Fubini's theorem, since $\mathrm{Leb}(\mathbf{DI}_K(\psi))>0$ and {$U_0, U_1$ are nonempty open subsets of $\C$}, we see that $m_G(C)>0$. 
Since $\Gamma_K$ is discrete, $G$ admits a countable open
cover by sets on which $\pi$ is injective. As $m_G(C)>0$,
there is one such set $V$ with $m_G(C\cap V)>0$.
Since $\mu_K$ locally agrees with a positive multiple of $m_G$, we have
$\mu_K(\cC)\geq\mu_K(\pi(C\cap V))>0$.
On the other hand, note that for any $\Lambda\in \cC$, we have $\Delta(g_s\Lambda)>\frac12 r_0$ for all  sufficiently large $s$. That is, for any $\Lambda\in \cC$, $g_s\Lambda \notin \Delta^{-1}[0, \frac12 r_0]$ for all  sufficiently large $s$. Since $\cC, \Delta^{-1}[0, \frac12 r_0]\subset X$ are both of positive $\mu_K$-measure, this violates the ergodicity of the flow $\{g_s\}$ on $(X,\mu_K)$.  We have thus obtained a contradiction, proving that $\mathrm{Leb}(\mathbf{DI}_K(\psi))=0$.

Finally, in view of the relation \eqref{equ:costrel} we immediately get the following bound on $c_K$ using \eqref{equ:uppbdrk}: 
\begin{align}\label{equ:bdckmp}
\max\left\{\tfrac{|D_K|}{4+2|D_K|^{1/2}}, \kappa_d\min\{|b_d|,1\}\right\}\leq c_K\leq \frac{2}{\pi}|D_K|^{1/2}.
\end{align}
The bound \eqref{equ:bdck} is then an immediate consequence of \eqref{equ:bdckmp}. 
\end{proof}

\subsection{DI vs BA}
Let $\psi_1(t)=\frac{1}{t}$. Recall 
$$\DI_K=\bigcup_{c\in(0,c_K)} \DI_K(c\psi_1)
$$
is the set of Dirichlet improvable complex numbers with respect to $K$. In view of Theorem \ref{thm:dirichlet}, $\DI_K$ is of zero Lebesgue measure. Moreover, as shown in the above proof, we have the following dynamical interpretation of Dirichlet improvable numbers:
\begin{align}\label{equ:didyint}
    z\in \DI_K\quad \Leftrightarrow\quad \exists\, r>0\, \text{ such that }\, g_s\Lambda_z \in \Delta^{-1}(r, \infty),\quad \forall\, s\gg 1.
\end{align}
In other words, in view of Mahler's compactness criterion, $z\in \DI_K$ if and only if the positive orbit $\{g_s\Lambda_z: s>0\}$ eventually avoids a compact neighborhood of $\mathfrak{C}_K$.
On the other hand,  the set of badly approximable complex numbers with respect to $K$, denoted by $\mathbf{BA}_K$, is defined to be the set of complex numbers $z\in \C/\cO_K$ for which
\begin{align}\label{equ:badapp}
    \liminf_{\substack{|q| \to \infty\\ q \in \mathcal{O}_K}} |q| \min_{p \in \mathcal{O}_K} |qz - p| > 0.
\end{align}
Using similar analysis as in the  above proof, we see that 
\begin{align}\label{eq: BA chararacterization}
    z\in \mathbf{BA}_K\quad \Leftrightarrow\quad \exists\, R>0\ \text{such that}\ g_s\Lambda_z\in \Delta^{-1}[0, R],\quad \forall\, s>0. 
\end{align}
It then follows from a similar ergodicity argument  that the set $\mathbf{BA}_K$ is of zero Lebesgue measure. 
For the two special cases under our consideration when $K = \mathbb{Q}(\sqrt{-1})$ or $K= \mathbb{Q}(\sqrt{-3})$, this set is known to have full Hausdorff dimension; see \cite[Theorem 1]{EK10}. In fact, the authors there show this set to be winning in the sense of Schmidt. As a byproduct of our analysis, especially the description of the  critical loci in Theorem \ref{thm:crilocus}, we have the following markedly different relations between the two sets $\DI_K$ and $\mathbf{BA}_K$ in these two cases. The contrasting situations should be viewed as a justification for studying the two notions of bad approximability and Dirichlet improvability independently.
\begin{Prop}
    If $K = \mathbb{Q}(\sqrt{-3})$, then $\mathbf{BA}_K \subset \mathbf{DI}_K$. In particular, the latter set is of full Hausdorff dimension. When $K = \mathbb{Q}(\sqrt{-1})$,  $\mathbf{BA}_K\smallsetminus \mathbf{DI}_K$ has full Hausdorff dimension.
\end{Prop}
\begin{proof}
    We see from \eqref{eq: BA chararacterization} and  Mahler's compactness criterion that $z$ is badly approximable if and only if $\{g_s \Lambda_z : s>0\}$ is precompact in $X$.
    When $K = \mathbb{Q}(\sqrt{-3})$, every lattice  $\Lambda \in \mathfrak{C}_K$ has the property that the trajectory $\{g_s \Lambda: s\in \R\}$ is unbounded in either the positive or negative direction. It then follows from \cite[Lemma 1.1]{Einsiedler-Kleinbock-Rao} that $\mathbf{BA}_K$ is contained in $\mathbf{DI}_K$. For completeness, we give a direct proof here.
    Let $z \in \mathbf{BA}_K$. 
    Suppose that $(t_k)_{k\in\N}$ is a
sequence of positive times tending to $\infty$ and that
$\Lambda\in X$ satisfies
    \begin{align}\label{eq: limit lattice}
        \lim_{k \to \infty} g_{t_k}\Lambda_z = \Lambda.
    \end{align}
    The continuity of the action implies that, for any $s \in \mathbb{R}$,
    \begin{align*}
        \lim_{k \to \infty} g_{t_k +s}\Lambda_z = g_s \Lambda.
    \end{align*}
    Since the sequence $(t_k+s)_{k \in \N}$ is eventually positive, we see that the orbit $\{g_s\Lambda : s \in \R\}$ is also precompact. In particular, $\Lambda$ cannot lie in $\mathfrak{C}_K$. The compactness of $\mathfrak{C}_K$ then implies that there exists some $s_0 > 0$ such that
    \begin{align*}
        \overline{\{g_s\Lambda_z : s > s_0\}} \cap \mathfrak{C}_K = \emptyset. 
    \end{align*}
Thus $z \in \mathbf{DI}_K$.

When $K = \mathbb{Q}(\sqrt{-1})$, the situation changes in light of  Lemma \ref{rmk:cloges}. {The critical locus $\mathfrak{C}_K$ lies in the union of the $\cR$-translates of two periodic $g_s$-orbits.} 
    The recent result \cite[Theorem 2.5]{Einsiedler-Kleinbock-Rao} studies precompact orbits and their limit points and is relevant to this setting. For convenience, let 
    $$
    C:=\left\{g_sr: s \in \R, r \in \mathcal{R}\right\} \subset G
    $$ denote the centralizer of the flow $\{g_s\}_{s\in \R}$. The cited theorem states that, given any $\Lambda \in X$ with precompact orbit $\{g_s\Lambda \in X: s \in \mathbb{R}\}$, there is a full Hausdorff dimension set of points $z \in \mathbf{BA}_K$ with the property that 
    there exists a sequence of positive times $s_k\to \infty$ such that
    \begin{align}\label{equ:consthm}
        \lim_{k \to \infty} g_{s_k} \Lambda_z \in C\Lambda.
    \end{align}
    We now specialize to $\Lambda \in \mathfrak{C}_K$ and observe that $C$ is commutative. Thus there is a full Hausdorff dimension set of $z \in \mathbf{BA}_K$ with the property that there exists an unbounded sequence of positive times $(t_k)_{k \in \N}$ with 
    \begin{align*}
        \lim_{k \to \infty} g_{t_k}\Lambda_z \in \mathcal{R}\mathfrak{C}_K = \mathfrak{C}_K.
    \end{align*}
    Thus, each such $z \in \mathbf{BA}_K$ cannot lie in $\mathbf{DI}_K$, and \cite[Theorem 2.5]{Einsiedler-Kleinbock-Rao} constructs a  set of full Hausdorff dimension contained in $\mathbf{BA}_K \smallsetminus \mathbf{DI}_K$.
\end{proof}

\section{Measure computation for Gaussian and Eisenstein integers}\label{sec: measure estimates}

Let $X=G/\G_K$ with $G=\SL_2(\C)$ and $\G_K=\SL_2(\cO_K)$ as before. Let $\mu_K$ be the $G$-invariant probability measure on $X$. 
Note that under the coordinates
\begin{align*}
g=\begin{pmatrix}
g_{11} & g_{12}\\
g_{21} & g_{22}
\end{pmatrix}\in G,
\end{align*}
the following measure
\begin{align}\label{equ:Haar}
\mathrm{d}m_G(g)=\frac{\mathrm{d}g_{11}\mathrm{d}g_{12}\mathrm{d} g_{21}}{|g_{11}|^2}
\end{align}
gives a Haar measure of $G$, where $\mathrm{d}z$ denotes the usual Lebesgue measure on $\C (\cong \R^2)$. Thus we have
\begin{align}\label{equ:nhamea}
    \mathrm{d}\mu_K(g)=\frac{1}{m_G(\cF_X)}\mathrm{d}m_G(g),
\end{align}
where $\cF_X\subset G$ is a fundamental domain for $X=G/\G_K$.

In this section, we bound  $\mu_K(\Delta^{-1}[0,r])$ for all  sufficiently small $r>0$ in both the Gaussian and Eisenstein cases, hence finishing the proof of Theorem \ref{thm: measure estimates}. This estimate will be based on the explicit description of the critical locus for these two cases as given in Theorem \ref{thm:crilocus}.

We first collect some general results which will be useful for our analysis. 
Throughout, for any $z_1,z_2\in \C\smallsetminus \{0\}$, we denote by $\ang(z_1,z_2)\in [0,\pi] $ the angle between them. If one of $z_1$ or $z_2$ is zero, then we regard $\ang(z_1,z_2)$ to be any value in $[0,\pi]$, and any constraint on $\ang(z_1,z_2)$ is regarded as vacuous. First we prove an elementary inequality which will be a key ingredient for our analysis.

\begin{Lem}\label{lemma: basic inequality}
For any $a>b\geq 0$ and $\theta\in [0,2\pi]$, we have 
\begin{align}\label{equ:basineq}
 |a+be^{i\theta}| \leq a+\frac{2ab\cos\theta+b^2}{2a}.
 \end{align}
In particular, for any $\theta_0\in ( \frac{\pi}{2}, \pi)$ and $z_1\in \C\smallsetminus\{0\},z_2\in \C$ satisfying
$$\ang(z_1,z_2)\in [\theta_0, \pi] \, \text{ and } \, \frac{|z_2|}{|z_1|}<\frac{|\cos(\theta_0)|}{2},$$
we have
\begin{align}\label{equ:normineq}
|z_1+z_2|\leq |z_1|+\frac{\cos(\theta_0)}{2}|z_2|.
\end{align}

\end{Lem}

\begin{proof}
Note that
$$ |a+be^{i\theta}|=\sqrt{ a^2+2ab\cos\theta+b^2 }\leq a+\frac{2ab\cos\theta+b^2}{2a}, $$
where the last inequality follows from the inequality 
$$  \sqrt{A+B}\leq \sqrt{A}+\frac{B}{2\sqrt{A}}, $$
for any $A,B\in \R$ satisfying $A>\max\{0, -B\}$. 
For the second assertion of the lemma, fix $\theta_0\in (\frac{\pi}{2}, \pi)$.  For any  $z_1\in \C\smallsetminus\{0\}$ and $z_2\in \C$, set $\theta:=\ang(z_1,z_2)$, $a:=|z_1|$ and $b:=|z_2|$, then we have 
\begin{align*}
|z_1+z_2|=|a+be^{i\theta}|.
\end{align*}
Thus if $\theta\in [\theta_0,\pi]$ (so that $\cos\theta\leq\cos\theta_0<0$) and $\frac{b}{a}<|\cos \theta_0|$, then applying \eqref{equ:basineq} we get
\begin{align*}
|z_1+z_2|\leq a+b\cos\theta+\frac{b^2}{2a}\leq a+\frac{\cos\theta_0}{2}b.
\end{align*}
This finishes the proof of \eqref{equ:normineq}. 
\end{proof}

Next, we show that the compact sets $\Delta^{-1}[0,r]$ are all contained in some $\e$-neighborhood of the critical locus, cf. \cite[Lemma 4.5]{KleinbockStrombergssonYu2022}.
\begin{Lem}\label{lem:comarg}
For any $\e>0$, there exists some $r_{\e}>0$ such that 
\begin{align}\label{equ:eapprox}
\Delta^{-1}[0,r]\subset \B_{G}(\e)\Delta^{-1}\{0\},\qquad \forall\, 0<r<r_{\e},
\end{align}
where 
$$
\B_{G}(\e):=\{g\in G: \|g-\Id\|<\e\},
$$
with $\|\cdot\|$ the supremum norm on $G$ (viewed as a subset of $\mathrm{M}_{2}(\C)\cong \C^4$). 
\end{Lem}
\begin{proof}
We prove the lemma by contradiction. Suppose not, then there exist some $\e>0$ and a decreasing sequence $\{\eta_n\}_{n\in \N}$ with $\eta_n\to 0$ such that, for each $n\in \N$, there exists some  $\Lambda_n\in \Delta^{-1}[0, \eta_n]$ satisfying $\Lambda_n\not\in \B_{G}(\e)\Delta^{-1}\{0\}$. Since $\Lambda_n\in \Delta^{-1}[0, \eta_n]\subset \Delta^{-1}[0, \eta_1]$ for all $n\in \N$ and since $\Delta^{-1}[0, \eta_1]$ is compact by Lemma \ref{lem:criticallocus}, after passing to a subsequence, we may assume $\Lambda_n\to \Lambda_0$ for some $\Lambda_0=g_0\cO_K^2\in X$. We may then choose $\{g_n\}\subset G$ such that $\Lambda_n=g_n\cO_K^2$ and $g_n\to g_0$ as $n\to\infty$. We now claim that $\Lambda_0\in \Delta^{-1}\{0\}$. Note that we can then get a contradiction since this would imply $\Lambda_n=g_n\cO^2_K\in \B_{G}(\e)\Lambda_0\subset \B_{G}(\e)\Delta^{-1}\{0\}$ for all sufficiently large $n$, contradicting the assumption that $\Lambda_n\not\in \B_{G}(\e)\Delta^{-1}\{0\}$ for all $n\in \N$. 

We now prove the claim. Suppose not, that is, $\Lambda_0\not\in \Delta^{-1}\{0\}$. This means that there exists some nonzero $\bm{v}=g_0\bm{m}$ with $\bm{m}\in \cO_K^2$ such that $\|\bm{v}\|<r_K$. Set $\delta:=\frac12(r_K-\|\bm{v}\|)>0$. Since $g_n\to g_0$, setting $\bm{v}_n:=g_n\bm{m}\in \Lambda_n\smallsetminus\{\bm{0}\}$, we have $\bm{v}_n\to \bm{v}$. In particular, $\|\bm{v}_n\|\to \|\bm{v}\|$. Thus we have $\|\bm{v}_n\|<r_K-\delta$ for all $n$ sufficiently large. This implies that $\Delta(\Lambda_n)\geq \log(\frac{r_K}{r_K-\delta})$ for all $n$ sufficiently large, contradicting the fact that $\Delta(\Lambda_n)\leq \eta_n\to 0$ as $n\to\infty$.
\end{proof}

\subsection{Measure estimates in the Gaussian case}

Recall that the critical locus of this case is given in \eqref{equ:-1} with $h_0\in G$ as in \eqref{eq:h_0}. {Note also that in this case $\cO_K=\Z[i]$ and $\Z[i]^{\times}=\{\pm 1, \pm i\}$.} The main goal of this section is to prove the following upper and lower bounds on $\Delta^{-1}[0,r]$ for all sufficiently small $r>0$. 

\begin{Prop}\label{prop:ulbd-1}
Assume $K=\Q(\sqrt{-1})$. There exists $\e_0>0$ such that 
\begin{align}\label{equ:incluupplow}
\underline{K}_r\subset \Delta^{-1}[0,r]\subset \overline{K}^+_r\cup \overline{K}^-_r,\qquad \forall\, 0<r<r_{\e_0}.
\end{align}
Here $r_{\e_0}$ is as in Lemma \ref{lem:comarg} and 
\begin{align}\label{equ:lowbdkr}
\underline{K}_r:=\left\{ gh_0\mathcal{O}_K^2: g\in \B_{G}(\e_0),\, |g_{11}| \in (1-r/2, 1+r/2), \, |g_{12}|, |g_{21}|\in (\tfrac{r}{2(2+\sqrt{3})},\tfrac{r}{2})\right\},
\end{align}
and
\begin{align}\label{equ:uppbdkr}
\overline{K}^{+}_r:=\overline{\cK}_rh_0\mathcal{O}_K^2\quad \text{and}\quad \overline{K}^{-}_r:=w_0\overline{\cK}_rw_0^{-1}h_1\mathcal{O}_K^2,
\end{align}
with
\begin{align}\label{eq: bar cKr}
\overline{\cK}_r:=\left\{g\in \B_{G}(\e_0)\cR: |g_{11}| \in (1-2r, 1+2r),\, \max\{|g_{12}|, |g_{21}|\}<24r\right\}.
\end{align}
\end{Prop}

\begin{remark}
Indeed, the relation $\underline{K}_r\subset \Delta^{-1}[0,r]$ still holds even if we drop the condition 
$$
\min\{|g_{12}|, |g_{21}|\}>\frac{r}{2(2+\sqrt{3})}
$$ 
in the definition of  $\underline{K}_r$; see the proof of Proposition \ref{prop:ulbd-1}. These further restrictions are needed for a later key disjointness statement; see Proposition \ref{prop:disjoint} below.
\end{remark}

 As a direct corollary of Proposition \ref{prop:ulbd-1}, we can prove the Gaussian case of Theorem \ref{thm: measure estimates}. 
 
\begin{proof}[Proof of Theorem \ref{thm: measure estimates}, Gaussian case]
Let $\e_0$ and $r_{\e_0}$ be as in Proposition \ref{prop:ulbd-1}. Then by \eqref{equ:incluupplow} we have 
\begin{align*}
\mu_K(\underline{K}_r)\leq \mu_K(\Delta^{-1}[0,r])\leq \mu_K(\overline{K}^+_r\cup \overline{K}^-_r),\qquad \forall\, 0<r<r_{\e_0}. 
\end{align*}
Now using the Haar measure descriptions \eqref{equ:Haar} and \eqref{equ:nhamea}, a direct computation gives
\begin{align*}
\mu_K(\underline{K}_r)\asymp_{\e_0,K} r^5\asymp_{\e_0,K}\mu_K(\overline{K}^{\pm}_r),\qquad \forall\, 0<r<r_{\e_0}.
\end{align*}
In view of the above two estimates, the Gaussian case estimate in \eqref{equ:meestboth} follows immediately.

\end{proof}

The remainder of this subsection is devoted to proving Proposition \ref{prop:ulbd-1}. First we describe the set of shortest vectors in lattices from the critical locus. Since $\mathfrak{C}_K=\tilde{\cR} h_0\mathcal{O}_K^2$ by Theorem \ref{thm:crilocus}, it suffices to understand the set of shortest vectors in $h_0\mathcal{O}_K^2$.
\begin{Lem}\label{lem:shvec-1}
Assume $K=\Q(\sqrt{-1})$. Let 
\begin{align}\label{equ:sk-1}
\cS_K:=
\left\{\begin{pmatrix} 1 \\ 0\end{pmatrix}, \begin{pmatrix} 0\\  1 \end{pmatrix},  \begin{pmatrix}  1 \\ -1 \end{pmatrix}, \begin{pmatrix} 1 \\ i\end{pmatrix}, \begin{pmatrix}
    1 \\ -1+i
\end{pmatrix}, \begin{pmatrix}
    1+i\\
    -1
\end{pmatrix} \right\}.
\end{align}
The set of all nonzero shortest vectors (with respect to the supremum norm) in $h_0 \mathcal{O}_K^2$ is $h_0\Z[i]^{\times}\cS_K$, which contains $24$ vectors of length $r_K$. 

\end{Lem}

\begin{remark}\label{rmk:shvecdes}
Note that $\omega=|\omega|e^{\frac{\pi i}{12}}$ with $|\omega|=\frac{\sqrt{2}(\sqrt{3}-1)}{2}= 0.517638\cdots$. Moreover, we have by direct computation that $\omega-i=e^{-\frac{\pi i}{3}}$ and $\omega-(1+i)=e^{-\frac{2\pi i}{3}}$.  From these identities, we have the following more explicit descriptions of these shortest vectors using polar coordinates:  
$$
h_0\begin{pmatrix} 1 \\ 0\end{pmatrix}=r_K\begin{pmatrix} 1 \\ |\omega|e^{-\frac{\pi i}{2}}\end{pmatrix},\quad h_0\begin{pmatrix} 1 \\ i\end{pmatrix}=r_Ke^{\frac{\pi i}{6}}\begin{pmatrix} 1 \\ |\omega|e^{\frac{\pi i}{6}} \end{pmatrix}, \quad  h_0\begin{pmatrix} 1+i \\ -1\end{pmatrix}=r_Ke^{\frac{\pi i}{3}}\begin{pmatrix} 1 \\  |\omega|e^{\frac{5\pi i}{6}} \end{pmatrix},
$$ 
$$
 h_0\begin{pmatrix} 1 \\ -1+i\end{pmatrix}=r_Ke^{\frac{3\pi i}{4}}\begin{pmatrix} |\omega|e^{-\frac{\pi i}{2}} \\ 1\end{pmatrix}, h_0\begin{pmatrix} 0 \\ 1\end{pmatrix}=r_Ke^{-\frac{\pi i}{12}}\begin{pmatrix} |\omega|e^{\frac{\pi i}{6}} \\ 1\end{pmatrix}, \quad h_0\begin{pmatrix} 1 \\ -1\end{pmatrix}=r_Ke^{\frac{13\pi i}{12}}\begin{pmatrix} |\omega|e^{\frac{5\pi i}{6}} \\ 1\end{pmatrix}.
$$ 
\end{remark}

\begin{proof}[Proof of Lemma \ref{lem:shvec-1}]
Let $\tilde{\cS}\subset \mathcal{O}_K^2$ be such that $h_0\tilde{\cS}$ consists of all shortest nonzero vectors in $h_0\mathcal{O}_K^2$.  We would like to show $\tilde{\cS}=\Z[i]^{\times}\cS_K$. Clearly, 
$$
\tilde{\cS}\subset (\mathcal{O}_K^2)_{\rm pr}:=\{\bm{v}\in \mathcal{O}_K^2: (\bm{v})=\mathcal{O}_K\}.
$$ 
Here $(\bm{v})$ denotes the ideal generated by the two entries of $\bm{v}$. 
We also have $\|h_0\bm{v}\|\leq r_K$ for any $\bm{v}\in \tilde{\cS}$ since $h_0\Z[i]^2$ contains a vector of length $r_K$.
Recall that we have used $\|\cdot\|$ for the supremum norm both on $G\subset \mathrm{M}_{2}(\C)$ and on $\C^2$. Note that $\|g\|=\|g^{-1}\|$ for any $g\in G$ and $\|h^{-1}_0\bm{v}\|\leq r_K(1+|\omega|)\|\bm{v}\|$ for any $\bm{v}\in \C^2$. Hence we have for any $\bm{v}\in \tilde{\cS}$,
\begin{align*}
\|\bm{v}\|=\|h_0^{-1}h_0\bm{v}\|\leq r_K(1+|\omega|)\|h_0\bm{v}\|\leq (1+|\omega|)r^2_K<1.7.
\end{align*}
This shows that 
\begin{align*}
\tilde{\cS}&\subset  \{\bm{v}\in \Z[i]^2_{\rm pr}: \|\bm{v}\|< 1.7\}\\
&=\Z[i]^{\times}\left\{\begin{pmatrix} 1 \\ 0\end{pmatrix}, \begin{pmatrix} 0\\  1 \end{pmatrix},  \begin{pmatrix}  1 \\ \pm 1 \end{pmatrix}, \begin{pmatrix} 1 \\ \pm i\end{pmatrix}, \begin{pmatrix}
    1 \\ \pm (1-i)
\end{pmatrix}, \begin{pmatrix}
    1\\
    \pm(1+i)
\end{pmatrix},\begin{pmatrix}
    1+i\\
    \pm 1,\pm i
\end{pmatrix}
 \right\}.
\end{align*}
Finally, a direct computation of the norms of the above vectors shows that $\tilde\cS=\Z[i]^{\times}\cS_K$ as desired.
\end{proof}
Combining Lemma \ref{lem:comarg} and Lemma \ref{lem:shvec-1} we have the following more explicit description of the sets $\Delta^{-1}[0,r]$ for all sufficiently small $r>0$.

\begin{Lem}\label{lem:moremandes-1}
Let $K=\Q(\sqrt{-1})$. There exists some $\e_0>0$ such that for any $0<r<r_{\e_0}$ with $r_{\e_0}$ as in Lemma \ref{lem:comarg}, we have
\begin{align}\label{equ:nebexpdes-1}
\Delta^{-1}[0,r]&=\cK_r h_0\mathcal{O}_K^2\cup w_0\cK_rw_0^{-1} h_1\mathcal{O}_K^2,
\end{align}
where 
\begin{align}\label{equ:calk_r}
\cK_r:=\left\{g\in \B_{G}(\e_0)\cR: \|gh_0\bm{v}\|\geq r_Ke^{-r},\ \forall\, \bm{v}\in \cS_K\right\},
\end{align}
with $\cS_K$ as in \eqref{equ:sk-1}. 
\end{Lem}

\begin{proof}
Fix $\e_0>0$ to be determined. Let $r_{\e_0}>0$ be as in Lemma \ref{lem:comarg} and fix $0<r<r_{\e_0}$. Take any $\Lambda\in \Delta^{-1}[0,r]$. By Lemma \ref{lem:comarg} and \eqref{equ:ckr-1}, we may assume $\Lambda=gh_i\mathcal{O}_K^2$ for some $g\in \B_{G}(\e_0)\cR$ and $i=0,1$. 

If $\Lambda=gh_0\Z[i]^2$, then by definition $\Lambda\in \Delta^{-1}[0,r]$ if and only if 
\begin{align}\label{equ:necondeq}
\|gh_0\bm{v}\|\geq r_Ke^{-r},\qquad \forall\, \bm{v}\in \Z[i]^2\smallsetminus\{\bm{0}\}.
\end{align}
We now claim that there exists some $\e_0>0$ such that for $g\in \B_{G}(\e_0)\cR$, \eqref{equ:necondeq} is equivalent to 
\begin{align}\label{equ:necondeqfin}
\|gh_0\bm{v}\|\geq r_Ke^{-r},\qquad \forall\, \bm{v}\in \cS_K. 
\end{align}
This claim clearly implies $\Lambda\in \cK_r h_0\Z[i]^2$. To prove the claim, we only need to show \eqref{equ:necondeqfin} implies \eqref{equ:necondeq}. Note that by Theorem \ref{thm:crilocus} and Lemma \ref{lem:shvec-1}, $\|h_0\bm{v}\|=r_K$ if and only if $\bm{v}\in \Z[i]^{\times}\cS_K$. The discreteness of $\Z[i]^2\subset \C^2$ then implies that there exists some absolute $r_K'>r_K$ such that $\|h_0\bm{v}\|>r_K'$ for any nonzero $\bm{v}\in \Z[i]^2\smallsetminus (\Z[i]^{\times}\cS_K)$. Now we can take $\e_0>0$ sufficiently small such that $\|g'\bm{w}\|>\frac{r_K+r_K'}{2}>r_K$ for all $g'\in \B_{G}(2\e_0)\cR$ and $\bm{w}\in \C^2$ with $\|\bm{w}\|>r_K'$. Fix such $\e_0$. Then we have for $g\in \B_{G}(\e_0)\cR$,  
\begin{align*}
\|gh_0\bm{v}\|>r_K>r_Ke^{-r},\quad\text{for any  nonzero $\bm{v}\in \Z[i]^2\smallsetminus (\Z[i]^{\times}\cS_K)$}.
\end{align*}
This shows that \eqref{equ:necondeqfin} implies \eqref{equ:necondeq}, whence the claim.

If $\Lambda=gh_1\Z[i]^2$, then we have $w^{-1}_0\Lambda=g_{w_0} h_0\Z[i]^2\in \Delta^{-1}[0,r]$ with $$
g_{w_0}:=w^{-1}_0gw_0\in w_0^{-1}\B_{G}(\e_0)\cR w_0= \B_{G}(\e_0)\cR.
$$ 
Thus, by the previous discussion we have that $w_0^{-1}\Lambda \in \Delta^{-1}[0,r]$ if and only if $g_{w_0}\in \cK_r$, i.e. $g\in w_0\cK_rw_0^{-1}$. This shows that $\Lambda\in w_0\cK_rw_0^{-1} h_1\Z[i]^2$, which finishes the proof.  
\end{proof}

\begin{remark}\label{rmk:expinequ}
Let $\e_0>0$ be as in Lemma \ref{lem:moremandes-1}. We see from the above proof that 
for any $g=\begin{pmatrix}
g_{11} & g_{12}\\
g_{21} & g_{22}
\end{pmatrix}\in \B_{G}(2\e_0)\cR$ and for any $0<r<r_{\e_0}$ 
\begin{align*}
 gh_0\Z[i]^2\in \Delta^{-1}[0,r]\quad \Leftrightarrow\quad \|gh_0\bm{v}\|\geq r_Ke^{-r},\quad \forall\, \bm{v}\in \cS_K.
\end{align*}
Using Remark \ref{rmk:shvecdes} we see that the above inequalities are equivalent to the following $6$ inequalities:
\begin{align}\label{equ:ineqgauss1}
\left|g_{11}+|\omega|e^{(-\frac{1}{2}+\frac{2k_1}{3})\pi i}g_{12}\right|\geq e^{-r}, \qquad \forall\, k_1=0,1,2,
\end{align}
and
\begin{align}\label{equ:ineqgauss2}
 \left|g_{22}+|\omega|e^{(-\frac{1}{2}+\frac{2k_2}{3})\pi i}g_{21}\right|\geq e^{-r},\qquad \forall\,  k_2=0,1,2.
\end{align}
\end{remark}

With all these preparations, we now give the proof of Proposition \ref{prop:ulbd-1}. 
\begin{proof}[Proof of Proposition \ref{prop:ulbd-1}.]
Fix some $\epsilon_0$ as given by Lemma \ref{lem:moremandes-1} and let $0<r<r_{\e_0}$ where $r_{\epsilon_0}$ is as in Lemma \ref{lem:comarg}. For any $g\in \B_{G}(\e_0)\cR$, by the above analysis we see that $gh_0\Z[i]^2\in \Delta^{-1}[0,r]$ if  and only if $g$ satisfies \eqref{equ:ineqgauss1} and \eqref{equ:ineqgauss2}. 

We first prove the inclusion $\underline{K}_r\subset \Delta^{-1}[0,r]$ in \eqref{equ:incluupplow}. In view of the definition of $\underline{K}_r$, it suffices to show  $\Lambda\in \Delta^{-1}[0,r]$ for any $\Lambda=gh_0\Z[i]^2$ with $g\in \B_{G}(\e_0)$ satisfying 
$$
1-r/2<|g_{11}|<1+r/2\quad  \text{and} \quad \max\{|g_{12}|, |g_{21}|\}<r/2.
$$ 
For this, since $\B_{G}(\e_0)\subset \B_{G}(\e_0)\cR$, in view of Lemma \ref{lem:moremandes-1} and Remark \ref{rmk:expinequ}, it suffices  to show $g$ satisfies \eqref{equ:ineqgauss1} and \eqref{equ:ineqgauss2}.  For \eqref{equ:ineqgauss1}, note that for any $k_1\in \{0,1,2\}$ we have
\begin{align*}
\left|g_{11}+|\omega|e^{(-\frac{1}{2}+\frac{2k_1}{3})\pi i}g_{12}\right|\geq |g_{11}|-|\omega||g_{12}|>1-\frac{r}{2}-0.3r>e^{-r},
\end{align*}
provided that $r>0$ is sufficiently small. For \eqref{equ:ineqgauss2}, note that 
$$
1-0.6r<\frac{1-(r/2)^2}{1+r/2}<|g_{22}|=\frac{|1+g_{12}g_{21}|}{|g_{11}|}< \frac{1+(r/2)^2}{1-r/2}<1+0.6r.
$$
Hence we can similarly estimate that for any $k_2\in \{0,1,2\}$, 
$$
\left|g_{22}+|\omega|e^{(-\frac{1}{2}+\frac{2k_2}{3})\pi i}g_{21}\right|\geq |g_{22}|-|\omega||g_{21}|>1-0.6r-0.3r>e^{-r},
$$
provided that $r>0$ is sufficiently small. This finishes the verification of \eqref{equ:ineqgauss1} and \eqref{equ:ineqgauss2}, whence $\underline{K}_r\subset \Delta^{-1}[0,r]$.

Next, we show $ \Delta^{-1}[0,r]\subset \overline{K}^+_r\cup \overline{K}^-_r$ as in \eqref{equ:incluupplow}. In view of Lemma \ref{lem:moremandes-1}, it suffices to show that 
\begin{align*}
\cK_r h_0\Z[i]^2\subset \overline{K}_r^+=\overline{\cK}_rh_0\Z[i]^2\quad \text{and}\quad w_0\cK_rw_0^{-1} h_1\Z[i]^2\subset \overline{K}_r^-=w_0\overline{\cK}_rw_0^{-1}h_1\Z[i]^2.
\end{align*}
Hence, it suffices to show 
$
\cK_r\subset \overline{\cK}_r.
$
For this,  first note that
\begin{align*}
\B_{G}(\e_0)\cR\subset \tilde{\B}_{G}(\e_0):=\{g\in G: |g_{11}|, |g_{22}|\in (1-\e_0, 1+\e_0),\, \max\{|g_{12}|, |g_{21}|\}<\e_0\}. 
\end{align*}
Then by the definition of $\cK_r$ (see \eqref{equ:calk_r}) and Remark \ref{rmk:expinequ}, we only need to prove that for any  $g\in \tilde{\B}_G(\e_0)$  satisfying \eqref{equ:ineqgauss1} and \eqref{equ:ineqgauss2}, $g$ also satisfies
\begin{align}\label{equ:desestg'}
|g_{11}|\in (1-2r, 1+2r)\quad\text{and}\quad \max\{|g_{12}|, |g_{21}|\}< 24 r.
\end{align}
First note that we can find $k_2\in \{0,1,2\}$ such that 
$$
\ang(g_{22}, |\omega|e^{(-\frac{1}{2}+\frac{2k_2}{3})\pi i}g_{21})\in [\tfrac{2\pi}{3}, \pi]. 
$$
Then by \eqref{equ:ineqgauss2} and \eqref{equ:normineq} (noting that $|g_{21}|<\e_0$ and $|g_{22}|>1-\e_0$) we have for such $k_2$,
\begin{align*}
e^{-r}\leq \left|g_{22}+|\omega|e^{(-\frac{1}{2}+\frac{2k_2}{3})\pi i}{g_{21}}\right|\leq |g_{22}|-\frac{|\omega|}{4}|g_{21}|\leq |g_{22}|-\frac{1}{8}|g_{21}|.
\end{align*}
This further implies that 
\begin{align}\label{equ:useine1}
|g_{22}|\geq e^{-r}\quad \text{and}\quad |g_{21}|\leq 8(|g_{22}|-e^{-r}). 
\end{align}
Next, we use the inequalities \eqref{equ:ineqgauss1}. 
Rewriting $g_{11}=\frac{1+g_{12}g_{21}}{g_{22}}$ we see that the inequalities in \eqref{equ:ineqgauss1} become
\begin{align}\label{equ:ineqgauss1'}
\left|\tfrac{1}{g_{22}}+\left(\tfrac{g_{21}}{g_{22}}+|\omega|e^{(-\frac{1}{2}+\frac{2k_1}{3})\pi i}\right)g_{12}\right|\geq e^{-r},\quad k_1=0,1,2.
\end{align}
Note that $\frac{|g_{21}|}{|g_{22}|}<\frac{\e_0}{1-\e_0}$. By reducing $\e_0$ if necessary, we have 
\begin{align}\label{equ:inequinte1}
\ang(z, \tfrac{g_{21}}{g_{22}}+z)<\tfrac{\pi}{360},\qquad \forall\, z\in \C \ \text{with}\ |z|=|\omega|.
\end{align}
Now as before we can find $k_1\in \{0,1,2\}$ such that 
$$
\ang(\tfrac{1}{g_{22}}, |\omega|e^{(-\frac{1}{2}+\frac{2k_1}{3})\pi i}g_{12}) \in [\tfrac{2\pi}{3}, \pi].
$$ 
This, together with \eqref{equ:inequinte1}, implies that for such $k_1$, 
\begin{align}\label{equ:keyinequ2}
\ang(\tfrac{1}{g_{22}}, \left(\tfrac{g_{21}}{g_{22}}+|\omega|e^{(-\frac{1}{2}+\frac{2k_1}{3})\pi i}\right)g_{12}) \in [\tfrac{239\pi}{360}, \pi].
\end{align}
Now applying \eqref{equ:ineqgauss1'} and \eqref{equ:normineq} for $z_1=\frac{1}{g_{22}}$ and $z_2=\left(\tfrac{g_{21}}{g_{22}}+|\omega|e^{(-\frac{1}{2}+\frac{2k_1}{3})\pi i}\right)g_{12}$ with $k_1\in \{0,1,2\}$ such that \eqref{equ:keyinequ2} holds, we get
\begin{align*}
e^{-r}\leq \left|\tfrac{1}{g_{22}}+\left(\tfrac{g_{21}}{g_{22}}+|\omega|e^{(-\frac{1}{2}+\frac{2k_1}{3})\pi i}\right)g_{12}\right|\leq \frac{1}{|g_{22}|}-\frac{1}{8}|g_{12}|.
\end{align*}
Here for the last inequality we used the estimate that 
$$
\frac{|\cos(\frac{239\pi}{360})|}{2}\left|\tfrac{g_{21}}{g_{22}}+|\omega|e^{(-\frac{1}{2}+\frac{2k_1}{3})\pi i}\right|>0.245(|\omega|-2\e_0)>0.245\times 0.517>\frac{1}{8},
$$
provided that $\e_0<0.0003$. (Recall that $|\omega|=\frac{\sqrt{2}(\sqrt{3}-1)}{2}=0.517638\cdots$.) We thus get
\begin{align}\label{equ:useine2}
\frac{1}{|g_{22}|}\geq e^{-r}\quad \text{and}\quad |{g_{12}}|\leq 8\left(\tfrac{1}{|g_{22}|}-e^{-r}\right). 
\end{align} 
Combining \eqref{equ:useine1} and \eqref{equ:useine2} we get
\begin{align*}
1-2r<e^{-r}\leq |g_{22}|\leq e^r<1+2r,
\end{align*}
and
\begin{align}\label{equ:gijest}
\max\{|g_{12}|, |g_{21}|\}< 8(e^r-e^{-r})<24 r. 
\end{align}
Finally, using the above two estimates we get
\begin{align}\label{equ:g11est}
1-2r<(1-(24r)^2)e^{-r}<|g_{11}|=\frac{|1+g_{12}g_{21}|}{|g_{22}|}<(1+(24r)^2)e^r<1+2r. 
\end{align}
We have thus proved \eqref{equ:desestg'}, hence \eqref{equ:incluupplow} and the proposition. 
\end{proof}

\subsection{Measure estimates in the Eisenstein case}
Recall the critical locus in this case is given in \eqref{equ:-3}. {Note also that in this case $\cO_K=\Z[j]$ with $j=e^{\frac{2\pi}{3}i}$ as before and $\Z[j]^{\times}=\{\pm j^k: k=0,1,2\}$.} The main goal here is to give the following upper and lower bounds for $\Delta^{-1}[0,r]$. 
\begin{Prop}\label{prop:ulbd-3}
Assume $K=\Q(\sqrt{-3})$. There exists some $\e_0>0$ such that 
\begin{align}\label{equ:incluupplow-3}
\underline{K}_r\subset \Delta^{-1}[0,r]\subset \overline{K}^{+}_r\cup \overline{K}^{-}_r,\qquad \forall\, 0<r<r_{\e_0},
\end{align}
where for any $0<r<r_{\e_0}$,
\begin{align}\label{equ:lowbdkr-3}
\underline{K}_r:=\left\{ g\Z[j]^2: \begin{array}{l} g\in \B_{G}(\e_0),\, |g_{11}|\in (1-\tfrac{r}{4}, 1+\tfrac{r}{4}),\, 0<|g_{12}g_{21}|< \tfrac{r}{4},\\ r<|g_{21}|<\sqrt{r},\, {\mathrm{Arg}(-g_{21}),\mathrm{Arg}(\overline{g_{12}})\in (\tfrac{25\pi}{72}, \tfrac{35\pi}{72})}\end{array}\right\},
\end{align}
and
\begin{align}\label{equ:uppbdkr-3}
\overline{K}^{+}_r:=\overline{\cK}_r\Z[j]^2\quad \text{and}\quad \overline{K}^{-}_r:=w_0\overline{\cK}_rw_0^{-1}\Z[j]^2,
\end{align}
with
\begin{align}
\overline{\cK}_r:=\left\{ g\in \tilde{\B}_G(\e_0): |g_{11}| \in (1-2r, 1+6002r),  |g_{12}g_{21}|< 6002 r\right\}.
\end{align}
Here
\begin{align}\label{equ:tildebgedef}
\tilde{\B}_G(\e_0):=\left\{g
\in G: |g_{11}|, |g_{22}|\in (1-2\e_0, 1+2\e_0), |g_{21}|<2\e_0, |g_{12}|<\tfrac{\sqrt{3}}{3}+2\e_0\right\},
\end{align}
{and for any nonzero $z\in \C$, $\mathrm{Arg}(z)\in (-\pi, \pi]$ is the \textit{principal argument} of $z$. }
\end{Prop}
\begin{remark}
 {Note that for $z=re^{i\theta}\neq 0$, $-z=re^{i(\theta+\pi)}$ and $\overline{z}=re^{-i\theta}$. From this one verifies that the  angular conditions  
    $$
    \mathrm{Arg}(-g_{21}), \mathrm{Arg}(\overline{g_{12}})\in (\tfrac{25\pi}{72}, \tfrac{35\pi}{72})
    $$ are equivalent, in terms of the principal arguments of $g_{21}$ and $g_{12}$, to \begin{align}\label{equ:anglerest}
        \mathrm{Arg}(g_{21})\in
\left(-\tfrac{47\pi}{72},-\tfrac{37\pi}{72}\right)
\quad\text{and}\quad 
\mathrm{Arg}(g_{12})\in
\left(-\tfrac{35\pi}{72},-\tfrac{25\pi}{72}\right).
    \end{align}}
\end{remark}
\begin{remark}
As in the Gaussian case, the relation $\underline{K}_r\subset \Delta^{-1}[0,r]$ still holds if we drop the condition $r<|g_{21}|<\sqrt{r}$. This condition is needed for the disjointness statement which is the key new ingredient to treat the divergent case; see Proposition \ref{prop:disjoint} below. We also note that the set $\tilde{\B}_G(\e_0)$ is defined so that 
\begin{align}\label{equ:tildebgedefrela}
\B_{G}(\e_0)\cR\{u(z): z\in \cF\}\subset \tilde{\B}_G(\e_0)\quad \text{and}\quad \B_{G}(\e_0)\cR\{u^-(z): z\in \cF\}\subset w_0\tilde{\B}_G(\e_0)w_0^{-1}.
\end{align}
Here the second relation follows from the first relation together with the fact that $w_0\B_{G}(\e_0)\cR w_0^{-1}=\B_{G}(\e_0)\cR$. 
\end{remark}
As before, we can use Proposition \ref{prop:ulbd-3} to prove the Eisenstein case of Theorem \ref{thm: measure estimates}.

\begin{proof}[Proof of Theorem \ref{thm: measure estimates}, Eisenstein case]
As in the Gaussian case, in view of the relation \eqref{equ:incluupplow-3} and the definitions of $\underline{K}_r$ and $\overline{K}_r^{\pm}$, the Eisenstein measure estimates in \eqref{equ:meestboth} follow from the following measure computations: 
\begin{align*}
\mu_K(\underline{K}_r)\asymp_{\e_0,K} r^3\log\left(\frac{1}{r}\right)\asymp_{\e_0, K} \mu_K(\overline{K}^{\pm}_r),\qquad \forall\, 0<r<r_{\e_0}.
\end{align*}
{Indeed, let $\pi: G\to X$ be the natural projection map. Up to further reducing $\e_0$, we may assume that $\pi|_{\underline{\cK}_r}$ is injective for each $0<r<r_{\e_0}$, where
\begin{align*}
\underline{\cK}_r:=\left\{ g\in \B_{G}(\e_0): \begin{array}{l}  |g_{11}|\in (1-\tfrac{r}{4}, 1+\tfrac{r}{4}),\, 0<|g_{12}g_{21}|< \tfrac{r}{4},\\ r<|g_{21}|<\sqrt{r},\, \mathrm{Arg}(-g_{21}),\mathrm{Arg}(\overline{g_{12}})\in (\tfrac{25\pi}{72}, \tfrac{35\pi}{72})\end{array}\right\}.
\end{align*}
Thus for all $0<r\ll_{\e_0} 1$, using \eqref{equ:Haar} and \eqref{equ:nhamea} (and noting that $|g_{11}|\in (\frac12, \frac32)$ for all $g\in \underline{\cK}_r$) we have
\begin{align*}
\mu_K(\underline{K}_r)&\asymp_K m_G(\underline{\cK}_r)\asymp_{\e_0} \vol_{\C}(B_{1})\vol_{\C^2}(B_2),
\end{align*}
where 
$$
B_1:=\left\{z\in \C: 1-\tfrac{r}{4}<|z|<1+\tfrac{r}{4}\right\}
$$ 
and 
\begin{align*}
B_2:=\left\{(z_1, z_2)\in \B_{\C^2}(\e_0): 0<|z_1z_2|<\tfrac{r}{4},\, r<|z_1|<\sqrt{r},\, \mathrm{Arg}(-z_1),\mathrm{Arg}(\overline{z_2})\in (\tfrac{25\pi}{72}, \tfrac{35\pi}{72}) \right\}.
    \end{align*}
   Clearly we have $\vol_{\C}(B_1)\asymp r$. For $\vol_{\C^2}(B_2)$, using polar coordinates $z_1=r_1e^{i\theta_1}$ and $z_2=r_2e^{i\theta_2}$ we have for all $0<r\ll_{\e_0}\ll 1$,
    \begin{align*}
\vol_{\C^2}(B_2)&\asymp \int_{\{(r_1,r_2)\in (0,\e_0)^2\,:\, r_1r_2<\tfrac{r}{4},\, r<r_1<\sqrt{r} \}} r_1 r_2\, \dd r_1\dd r_2
=\int_r^{\sqrt{r}}\left(\int_0^{\min\{\e_0, r/(4r_1)\}}r_2\,\dd r_2\right) r_1\,\dd r_1\\
&\asymp_{\e_0}\int_r^{\sqrt{r}}\min\{1, r/r_1\}^2 r_1\, \dd r_1
=r^2\int_r^{\sqrt{r}}\,\frac{\dd r_1}{r_1}\asymp r^2\log\left(\frac{1}{r}\right).
        \end{align*}
        Combining the above two estimates we get 
        \begin{align*}
\mu_K(\underline{K}_r)\asymp_{\e_0, K} \vol_{\C}(B_{1})\vol_{\C^2}(B_2)\asymp_{\e_0} r^3\log\left(\frac{1}{r}\right),\qquad (\forall\, 0<r\ll_{\e_0} 1),
            \end{align*}
            as desired. 
        The other claimed estimates that $\mu_K(\overline{K}^{\pm}_r)\asymp_{\e_0} r^3\log(\frac{1}{r})$ follow from a similar but simpler computation.}
\end{proof}

\subsubsection{Proof of Proposition \ref{prop:ulbd-3}}
The remainder of this subsection is devoted to proving Proposition \ref{prop:ulbd-3}. The proof strategy is similar to that of Proposition \ref{prop:ulbd-1}, {but the analysis needed here is different and much more delicate.} As before, we first list all the shortest vectors in lattices from the critical locus. 
\begin{Lem}\label{lem:shvec-3}
Assume $K=\Q(\sqrt{-3})$. Let 
\begin{align}\label{equ:sk-3}
\cS_K:=
\left\{\begin{pmatrix} 1 \\ 0\end{pmatrix}, \begin{pmatrix} 0\\  1 \end{pmatrix},   \begin{pmatrix}
    1\\
    z
\end{pmatrix}: z\in \Z[j]^{\times} \right\}.
\end{align}
Let $\cF$ be the fixed fundamental domain for $\C/\Z[j]$ as in \eqref{equ:fdeis}. For any $\Lambda=h\Z[j]^2$ with $h=\begin{pmatrix}
        1 & z\\
         0 & 1
    \end{pmatrix}$ or $h=\begin{pmatrix}
        1 & 0\\
         z & 1
    \end{pmatrix}$ for some $z\in \cF$, the set of all nonzero shortest vectors (with respect to the supremum norm) in $\Lambda$ is  contained in $h\Z[j]^{\times}\cS_K$.
\end{Lem}
\begin{proof}
First note that $|z|\leq \frac{\sqrt{3}}{3}$ for any  $z\in \cF$ (see \eqref{equ:keyprop}). Moreover, all the nonzero  non-unit elements of $\Z[j]$ have absolute value at least $\sqrt{3}$. From this we see that
\begin{align}\label{equ:keobse}
\forall\, z\in \cF,\ \forall\, z'\in \Z[j]\ \ :\  \ |z+z'|\leq 1 \Rightarrow  z'\in \Z[j]^{\times}\cup\{0\}.
\end{align}
This is true since otherwise we would have $|z+z'|\geq |z'|-|z|\geq \sqrt{3}-\frac{\sqrt{3}}{3}=\frac{2\sqrt{3}}{3}>1$. 

We now prove this lemma. We only prove the case when $h=\begin{pmatrix}
        1 & z\\
         0 & 1
    \end{pmatrix}$; the other case can be treated similarly. By Theorem \ref{thm:crilocus} we know that the shortest nonzero vectors in $\Lambda$ are all of length $1$. We thus need to solve the inequalities
    \begin{align*}
    \left\|\begin{pmatrix}
    1 &z\\
    0 & 1\end{pmatrix}\begin{pmatrix}
    a \\ b\end{pmatrix}\right\|=\left\|\begin{pmatrix}
    a+bz \\ b\end{pmatrix}\right\|\leq 1\quad \Leftrightarrow\quad  \max\{|a+bz|, |b|\}\leq 1,\quad \forall\, (a,b)\in \Z[j]^2\smallsetminus\{\bm{0}\}. 
    \end{align*}
   The inequality $|b|\leq 1$ shows that  $b$ is either $0$ or a unit in $\Z[j]$. If $b=0$, then $a\neq 0$ and the inequality $|a|=|a+bz|\leq 1$ forces $a\in \Z[j]^{\times}$. If $b\in \Z[j]^{\times}$, then we have  $|ab^{-1}+z|=|a+bz|\leq 1$. Since $z\in \cF$, by \eqref{equ:keobse} we have $ab^{-1}\in \Z[j]^{\times}\cup \{0\}$, or equivalently, $a\in \Z[j]^{\times}\cup \{0\}$. This shows that all the shortest nonzero vectors of $\Lambda$ are contained in the set $h\Z[j]^{\times}\cS_K$, finishing the proof.
\end{proof}

Using similar arguments as in the proof of Lemma \ref{lem:moremandes-1}, but with Lemma \ref{lem:shvec-3} in place of Lemma \ref{lem:shvec-1}, we have the following more explicit description of the set $\Delta^{-1}[0,r]$ in this setting. 
\begin{Lem}\label{lem:moremandes-3}
Let $K=\Q(\sqrt{-3})$. There exists some $\e_0>0$ such that for any $0<r<r_{\e_0}$ with $r_{\e_0}$ as in Lemma \ref{lem:comarg}, we have
\begin{align}\label{equ:nebexpdes-3}
\Delta^{-1}[0,r]&=\left\{gu(z)\Z[j]^2: g\in \B_{G}(\e_0)\cR, z\in \cF, \|gu(z)\bm{v}\|\geq e^{-r},\ \forall\, \bm{v}\in \cS_K\right\}\\ \nonumber
&\bigcup \left\{gu^-(z)\Z[j]^2: g\in \B_{G}(\e_0)\cR, z\in \cF, \|gu^-(z)\bm{v}\|\geq e^{-r},\ \forall\, \bm{v}\in \cS_K\right\},
\end{align}
where $\cS_K$ is as in \eqref{equ:sk-3}.
\end{Lem}

{Recall $\tilde{\B}_G(\e_0)$ was defined in \eqref{equ:tildebgedef} and it satisfies the relations in \eqref{equ:tildebgedefrela}. }
Note that for any $g=\begin{pmatrix}
g_{11} & g_{12}\\
g_{21} & g_{22}
\end{pmatrix}\in {\tilde{\B}_G(\e_0)\cup w_0\tilde{\B}_G(\e_0)w_0^{-1}}$, the condition $\|g\bm{v}\|\geq e^{-r}$ for all $\bm{v}\in \cS_K$ is equivalent to the following inequalities 
\begin{align}\label{equ:ineq1-3}
\min\{|g_{11}|, |g_{22}|\}\geq e^{-r}, 
\end{align}
and
\begin{align}\label{equ:ineq2-3}
\max\left\{|g_{11}+g_{12}e^{\frac{k\pi i}{3}}|, |g_{22}+g_{21}e^{-\frac{k\pi i}{3}}| \right\}\geq e^{-r},\qquad  \forall\, 0\leq k\leq 5.
\end{align}
We thus have 
\begin{align}\label{equ:uppbdest-3}
\Delta^{-1}[0,r]\subset \left\{g\Z[j]^2: g\in \tilde{\B}_G(\e_0)\cup w_0\tilde{\B}_G(\e_0)w_0^{-1}, \text{$g$ satisfies \eqref{equ:ineq1-3} and \eqref{equ:ineq2-3}}\right\}. 
\end{align}

In order to prove the second inclusion relation in \eqref{equ:incluupplow-3} we need the following technical lemma. Let $T: \R/2\pi\Z\to \R/2\pi \Z$ be defined such that $T(x)=x+\frac{\pi}{3}\Mod{2\pi}$ and identify $\R/2\pi \Z$ with $(-\pi, \pi]$. 
\begin{Lem}\label{lem:ropro}
For any $\theta_1,\theta_2\in \R/2\pi \Z$ such that $\ang(e^{i(\theta_1+\theta_2)}, 1)\leq \frac{37\pi}{72}$ and for any $\alpha_1,\alpha_2\in \R/2\pi\Z$ such that $\ang (e^{i(\alpha_1+\alpha_2)}, 1)<\frac{\pi}{72}$,  there exists $0\leq k\leq 5$ satisfying 
\begin{align*}
\cos(T^k(\theta_1)-\alpha_1)<-0.3\quad \text{and}\quad \cos(T^{-k}(\theta_2)-\alpha_2)<-0.6\cos(\theta_1+\theta_2)-0.03. 
\end{align*}

\end{Lem}

\begin{proof}
Let $S:=\mathrm{Arg}(e^{i(\theta_1+\theta_2)}), s:=\mathrm{Arg}(e^{i(\alpha_1+\alpha_2)})\in (-\pi, \pi]$. By assumption we have $|S|\leq \frac{37\pi}{72}$ and $|s|< \frac{\pi}{72}$. Define 
$$
I_x:=\{\beta\in \R/2\pi\Z: \cos(\beta)<x\},\qquad \forall\, x\in (-1,1).
$$
Note that $I_x\subset \R/2\pi\Z$ is the arc $(-\pi, -\arccos(x))\cup (\arccos(x), \pi]$ of length $2\pi-2\arccos(x)$. 
Since 
$$T^k(\theta_1)+T^{-k}(\theta_2)=\theta_1+\theta_2, \qquad \forall \, 0\leq k\leq 5,$$
the lemma is equivalent to the existence of some $0\leq k\leq 5$ such that 
\begin{align*}
T^{-k}(\theta_2)-\alpha_2\in I_{-0.6\cos(S)-0.03}\cap (-I_{-0.3}+S-s).
\end{align*}
This can be guaranteed if we can show 
$$
J_{S,s}:=(I_{-0.6\cos(S)-0.03}-S)\cap (-I_{-0.3}-s)
$$ 
contains an arc of length larger than $\frac{\pi}{3}$ provided that  $|S|\leq \frac{37\pi}{72}$ and $|s|< \frac{\pi}{72}$. 
Note that since $|s|<\frac{\pi}{72}$, $-I_{-0.3}-s=I_{-0.3}-s$ always contains the arc 
$$
J_0:=(-\pi, -\tfrac{\pi}{72}-\arccos(-0.3))\cup (\arccos(-0.3)+\tfrac{\pi}{72}, \pi]\supseteq (-\pi, -\tfrac{11\pi}{18}]\cup [\tfrac{11\pi}{18}, \pi], 
$$ 
which is an arc centered at $-\pi$ and of length at least $\frac{7\pi}{9}$. 
On the other hand, since $|S|\leq\frac{37\pi}{72}$, $I_{-0.6\cos(S)-0.03}-S$ is the arc centered at $-\pi -S$ of length $2\pi-2\arccos(-0.6\cos(S)-0.03)$. Note that the function $S\mapsto 2\pi-2\arccos(-0.6\cos(S)-0.03)$ is even on $S\in [-\frac{37\pi}{72}, \frac{37\pi}{72}]$ and is increasing on $[0, \frac{37\pi}{72}]$. For simplicity, we assume $S\geq 0$. If $S< \frac{\pi}{3}$, then since 
$$
2\pi-2\arccos(-0.6\cos(S)-0.03)\geq 2\pi-2\arccos(-0.63)> \frac{101\pi}{180},
$$
using the above description of the arc $J_0$ we see that in this case $(I_{-0.6\cos(S)-0.03}-S)\cap J_0$ always contains an arc of length larger than $\frac{\pi}{3}$. 

Now we assume $\frac{\pi}{3}\leq S\leq \frac{37\pi}{72}$. Then for any such $S$, we have
$$
2\pi-2\arccos(-0.6\cos(S)-0.03)\geq 2\pi-2\arccos(-0.6\cos(\pi/3)-0.03)>\frac{7\pi}{9}.
$$
Note that $I_{-0.6\cos(S)-0.03}-S$ is an arc centered at $-\pi-S$ with length larger than $\frac{7\pi}{9}$. If $\tfrac{\pi}{3}\leq S\leq \frac{7\pi}{18}$, we have $-\pi-S\in J_0$ and it follows that $(I_{-0.6\cos(S)-0.03}-S)\cap J_0$ contains an arc of length at least $\frac{7\pi}{18}>\frac{\pi}{3}$. If $\tfrac{7\pi}{18}\leq S\leq \tfrac{37\pi}{72}$, then $(I_{-0.6\cos(S)-0.03}-S)\cap J_0$ contains an arc of length at least $f(S)$, where
$$f(S)=\frac{25}{18}\pi-\arccos(-0.6\cos(S)-0.03)-S. $$
A direct computation verifies that $f(S)$ is a decreasing function when $\tfrac{7\pi}{18}\leq S\leq \tfrac{37\pi}{72}$. Hence, $(I_{-0.6\cos(S)-0.03}-S)\cap J_0$ contains an arc of length at least $f(37\pi/72)>\frac{67\pi}{180}>\pi/3$. This finishes the proof of this lemma. 
\end{proof}

\begin{proof}[Proof of Proposition \ref{prop:ulbd-3}]
Fix $0<r<r_{\e_0}$. We first prove the inclusion $\underline{K}_r\subset \Delta^{-1}[0,r]$ in \eqref{equ:incluupplow-3}. Take any $g\Z[j]^2\in \underline{K}_r$ with $g\in \B_{G}(\e_0)$ satisfying 
$$
|g_{11}|\in (1-\tfrac{r}{4}, 1+\tfrac{r}{4}),\, 0<|g_{12}g_{21}|< \tfrac{r}{4},\, \mathrm{Arg}(-g_{21}), \mathrm{Arg}(\overline{g_{12}})\in (\tfrac{25\pi}{72}, \tfrac{35\pi}{72}).
$$
We would like to show $g\Z[j]^2\in \Delta^{-1}[0,r]$, or equivalently, to show $g$ satisfies \eqref{equ:ineq1-3} and \eqref{equ:ineq2-3}.  Clearly, $|g_{11}|> 1-\frac{r}{4}>e^{-r}$. Moreover, we have 
$$
|g_{22}|=\frac{|1+g_{12}g_{21}|}{|g_{11}|}\geq \frac{1-r/4}{1+r/4}\geq e^{-r}.
$$
Hence $g$ satisfies \eqref{equ:ineq1-3}. Next, we show $g$ satisfies \eqref{equ:ineq2-3}. For this, note that since $\max\{|g_{11}-1|, |g_{22}-1|\}<\e_0$, up to further assuming  $\e_0<0.04<\sin(\frac{\pi}{72})$, we may assume 
$$
\max\{\ang(g_{11}, 1), \ang(g_{22},1)\}<\frac{\pi}{72}.
$$ 
Together with the assumption $\mathrm{Arg}(-g_{21}),\mathrm{Arg}(\overline{g_{12}})\in (\tfrac{25\pi}{72}, \tfrac{35\pi}{72})$, we have
\begin{align}\label{equ:anglest}
\ang(g_{11}, g_{12}e^{\frac{k\pi i}{3}} ) <\frac{\pi}{2} \ \text{for $k=0,1,2$},
\quad
\text{and} 
\quad 
\ang(g_{22}, g_{21}e^{-\frac{k\pi i}{3}} ) <\frac{\pi}{2} \ \text{for $k=3,4,5$}.
\end{align}
This implies
$$
|g_{11}+g_{12}e^{\frac{k\pi i}{3}}|\geq |g_{11}|\geq e^{-r}, \forall \, k\in \{0,1,2\} \text{ and}\,\, |g_{22}+g_{21}e^{-\frac{k\pi i}{3}}|\geq |g_{22}|\geq e^{-r}, \,\, \forall \, k\in \{3,4,5\}.
$$
The above two estimates then clearly imply \eqref{equ:ineq2-3}, which finishes the proof of the first inclusion relation in \eqref{equ:incluupplow-3}.

Next, we prove  $\Delta^{-1}[0,r]\subset \overline{K}^{+}_r\cup \overline{K}^{-}_r$ in \eqref{equ:incluupplow-3}. In view of the inclusion relation \eqref{equ:uppbdest-3}, for any $g\in \tilde{\B}_G(\e_0)\cup w_0\tilde{\B}_G(\e_0)w_0^{-1}$ satisfying \eqref{equ:ineq1-3} and \eqref{equ:ineq2-3}, it suffices to show that $g$ satisfies 
\begin{align}\label{equ:uppbdgoal-3}
|g_{11}|\in (1-2r, 1+6002r)\quad \text{and}\quad |g_{12}g_{21}|<6002r. 
\end{align}
If $g_{12}g_{21}=0$, then the second estimate in \eqref{equ:uppbdgoal-3} clearly holds. Moreover, we have by \eqref{equ:ineq1-3} that $|g_{11}|\geq e^{-r}>1-2r$ and $|g_{11}|=\frac{1}{|g_{22}|}\leq e^r<1+2r$. Hence the first estimate in \eqref{equ:uppbdgoal-3} also holds. Thus we may assume $g_{12}g_{21}\neq 0$. 
{Moreover, up to replacing $g$ by $w_0^{-1}gw_0$, we may assume $g\in \tilde{\B}_G(\e_0)$. } Now write 
$$
g_{11}=a_1 e^{i\alpha_1},\quad g_{22}=a_2 e^{i\alpha_2},\quad g_{12}=b_1e^{i\theta_1}\quad \text{and} \quad g_{21}=b_2 e^{i\theta_2}.
$$ 
By assumption, 
$$
a_1, a_2\in (1-2\e_0, 1+2\e_0),\quad 0<b_1<\tfrac{\sqrt{3}}{3}+2\e_0\quad \text{and} \quad 0<b_2<2\e_0.
$$ 
Moreover, note that 
$$
|g_{11}g_{22}-1|=|g_{12}g_{21}|<2\e_0.
$$
From this we see that 
$$
\cos(\alpha_1+\alpha_2)>\frac{a_1^2a_2^2+1-4\e_0^2}{2a_1a_2}>\frac{(1-2\e_0)^2+1-4\e_0^2}{2(1+2\e_0)^2}=\frac{1-2\e_0}{(1+2\e_0)^2}>1-7\e_0.
$$
Thus up to reducing $\e_0$ further, we may assume
\begin{align}\label{equ:alphaest}
\ang(e^{i(\alpha_1+\alpha_2)}, 1)<\frac{\pi}{72}.
\end{align}
We now start to prove \eqref{equ:uppbdgoal-3}. We divide the discussion into two cases. Let 
$$
\theta_0:=\frac{\pi}{2}+\frac{\pi}{72}=\frac{37\pi}{72}.
$$ 

\underline{\textbf{Case 1:}} $\ang(g_{12}g_{21},1)
\in (\theta_0,\pi]$.

Since $|\cos(\theta_0)|>0.04$ and  $|g_{12}g_{21}|\leq 2\epsilon_0<0.02$, we can apply  \eqref{equ:normineq} to get 
$$ \abs{g_{11}g_{22}}=\abs{1+g_{12}g_{21} }\leq 1-0.02|g_{12}g_{21}|\leq 1. $$
Since $\min\{|g_{11}|, |g_{22}|\}\geq e^{-r}$ (by \eqref{equ:ineq1-3}), we have
$$ 
 \max\{|g_{11}|, |g_{22}|\}\leq \frac{1}{\min\{|g_{11}|, |g_{22}|\}}\leq e^{r}<1+2r. $$
Moreover, we have
$
1-2r\leq e^{-2r}\leq |g_{11}g_{22}|\leq 1-0.02|g_{12}g_{21}|,
$ 
which implies that 
$$
|g_{12}g_{21}|\leq 100 r.
$$
Thus in this case $g$ satisfies \eqref{equ:uppbdgoal-3}. 

\underline{\textbf{Case 2:}} $\ang(g_{12}g_{21},1)\in [0, \theta_0]$. 

Note that the assumption $\ang(g_{12}g_{21},1)\in [0, \theta_0]$ is equivalent to $\ang(e^{i(\theta_1+\theta_2)},1)\leq \theta_0=\frac{37\pi}{72}$. Moreover, by \eqref{equ:alphaest} we have $\ang( e^{i(\alpha_1+\alpha_2)},1)<\frac{\pi}{72}$. Hence we can apply Lemma \ref{lem:ropro} to find some $0\leq k\leq 5$ such that 
\begin{align}\label{equ:keyinteinef}
\cos(T^k(\theta_1)-\alpha_1)<-0.3\quad\text{and}\quad \cos(T^{-k}(\theta_2)-\alpha_2)<-0.6\cos(\theta_1+\theta_2)-0.03.
\end{align}
Applying \eqref{equ:ineq2-3} for such $k$, we get
\begin{align}\label{eq: Eisenstein condition 2}
    \max\left\{ |a_1+b_1e^{i(T^k(\theta_1)-\alpha_1)}|,|a_2+b_2e^{i(T^{-k}(\theta_2)-\alpha_2)}| \right\}\geq e^{-r}.
\end{align}
Below, let $r_1=|g_{12}g_{21}|$ and set $\theta_{k,1}':=T^k(\theta_1)-\alpha_1$ and $\theta_{k,2}':=T^{-k}(\theta_2)-\alpha_2$. Note that by \eqref{equ:basineq} we have
\begin{align*}
a_1a_2=|g_{11}g_{22}|=|1+g_{12}g_{21}|=|1+r_1e^{i(\theta_1+\theta_2)}|\leq 1+r_1\cos(\theta_1+\theta_2)+\frac{r_1^2}{2}.
\end{align*}
Using also $\min\{a_1,a_2\}\geq e^{-r}>(1+2r)^{-1}$, this shows that 
\begin{align}\label{equ:a12estf}
\max\{a_1,a_2\}=\frac{a_1a_2}{\min\{a_1,a_2\}}\leq (1+r_1\cos(\theta_1+\theta_2)+r_1^2/2)(1+2r).
\end{align}
Moreover, by \eqref{equ:a12estf}, \eqref{equ:basineq} and the first inequality in \eqref{equ:keyinteinef}, 
we have for the above $k$,
\begin{align*}
|a_1+b_1e^{i\theta_{k,1}'}|&\leq a_1+b_1\cos(\theta_{k,1}')+\frac{b_1^2}{2a_1}\\
&\leq (1+r_1\cos(\theta_1+\theta_2)+r_1^2/2)(1+2r)+\left(\cos(\theta_{k,1}')+\frac{b_1}{2a_1}\right)\frac{r_1}{b_2}\\
&<1+2r+r_1\left((1+2r)(\cos(\theta_1+\theta_2)+\e_0)+\frac{-0.3+\sqrt{3}/6+4\e_0}{2\e_0}\right)\\
&<1+2r-10 r_1.
\end{align*}

Here for the second last inequality we used that for any $\e_0>0$ sufficiently small,
$$\cos(\theta_{k,1}')+\frac{b_1}{2a_1}<-0.3+\frac{\sqrt{3}/3+2\e_0}{2(1-2\e_0)}<-0.3+\frac{\sqrt{3}}{6}+4\e_0<0,$$
and $0<b_2<2\e_0$, and for the last inequality we used that 
$$
\frac{-0.3+\frac{\sqrt{3}}{6}+4\e_0}{2\e_0}<\frac{-0.01}{2\e_0}< -20,\qquad \text{provided that $0<\e_0<0.00025$}. 
$$
Similarly, by \eqref{equ:a12estf}, Lemma \ref{lemma: basic inequality} and the second inequality in \eqref{equ:keyinteinef} we have for such $k$, 
\begin{align*}
|a_2+b_2e^{i\theta_{k,2}'}|&\leq a_2+b_2\cos(\theta_{k,2}')+\frac{b_2^2}{2a_2}\\
&\leq (1+r_1\cos(\theta_1+\theta_2)+r_1^2/2)(1+2r)+\left(\frac{\cos(\theta_{k,2}')}{b_1}+\frac{b_2}{2a_2b_1}\right)r_1\\
&<1+2r+r_1\left((1+2r)\cos(\theta_1+\theta_2)+(1+2r)\e_0+\frac{-0.6\cos(\theta_1+\theta_2)-0.03+3\e_0}{\sqrt{3}/3+2\e_0}\right)\\
&<1+2r-0.0005r_1.
\end{align*}
Here the second to the last inequality follows from the fact that
$$ -0.6\cos(\theta)-0.03 +3\e_0<0, \,\, \forall\, \theta\in [0,\tfrac{37\pi}{72}], $$
 provided that $\e_0>0$ is sufficiently small. The last inequality holds since  $0.6>\sqrt{3}/3$ and 
$$ \cos(\theta)\left(1-\frac{0.6}{\sqrt{3}/3}\right)-\frac{0.03}{\sqrt{3}/3}<-0.0005, \,\, \forall \, \theta\in [0,\tfrac{37\pi}{72}]. $$
Using these two estimates together with \eqref{equ:ineq2-3} we get for the above $k$,
\begin{align*}
1-r<e^{-r}\leq \max\{|a_1+b_1e^{i\theta_{k,1}'}|, |a_2+b_2e^{i\theta_{k,2}'}|\}\leq  1+2r-0.0005r_1\quad \Rightarrow \quad r_1<6000 r.
\end{align*}
Moreover, we have 
$$
|g_{11}g_{22}|=|1+g_{12}g_{21}|\leq 1+r_1\leq  1+6000r,
$$
which together with \eqref{equ:ineq1-3} implies that
\begin{align*}
\max\{|g_{11}|, |g_{22}|\}\leq (1+6000r)e^r<1+6002 r.
\end{align*}
Hence we see that \eqref{equ:uppbdgoal-3} also holds in this case. 
\end{proof}

\subsection{Measure estimates for the thickenings}
The shrinking targets we will work with are actually certain thickenings of these level sets along the direction of the diagonal flow $\{g_{s}\}$. Thus we prove the following measure estimates for these thickenings. 
\begin{Prop} \label{prop: measure estimates the thickenings}
Let $K=\Q(\sqrt{-1})$ or $K=\Q(\sqrt{-3})$.
Let  $r_0, c\in (0, 1]$ and $s_0>0$ be any three parameters. Suppose $\{\Delta'_r\}_{0<r<r_0}$ is a family of Borel sets satisfying
\begin{align*}
\Delta'_r\subset \Delta^{-1}[0,r]\quad \text{and}\quad \mu_K(\Delta'_r)\geq c \mu_K(\Delta^{-1}[0,r]),\qquad \forall\, 0<r<r_0. 
\end{align*}
Then for any $\tau>0$ we have 
$$
\mu_K(\widetilde{\Delta}'_{\tau,{s_0}}(r))\asymp_{\tau, s_0,c} r^{-1}\mu_K(\Delta^{-1}[0,r]),\qquad \text{as $r\to 0^+$},
$$
where
$$
\widetilde{\Delta}'_{\tau, {s_0}}(r):=\bigcup_{0\leq s<s_0}g_{-\tau s}\Delta'_r.
$$
\end{Prop}

\begin{proof}
The proof is similar to that in \cite[Section 5]{KleinbockStrombergssonYu2022}. {First note that for any $0<r<r_0$,
$$
\widetilde{\Delta}'_{\tau,s_0}(r)=\bigcup_{0\leq s<\tau s_0}g_{-s}\Delta'_r=\widetilde{\Delta}'_{1,\tau s_0}(r).
$$
Thus after replacing $\tau s_0$ by $s_0$, we may assume $\tau=1$. We abbreviate $\widetilde{\Delta}'_{1, s_0}(r)$ by $\widetilde{\Delta}'(r)$ for simplicity.}
We first prove the desired upper bound. By definition of the function $\Delta$, we have 
$$ g_{-s}\Delta^{-1}[0,r]\subset \Delta^{-1}[0,2r],\qquad \forall\, r>0, s\in [0, r]. $$
Let $q=\lceil s_0/r \rceil$. Then
\begin{align*}
   \widetilde{\Delta}'(r)\subset \bigcup_{0\leq s<s_0}g_{- s}\Delta'_r\subset \bigcup_{k=0}^{q-1}g_{-ks_0/q} \left(\bigcup_{0\leq s< s_0/q} \, g_{-s}\Delta^{-1}[0,r]\right)\subset \bigcup_{k=0}^{q-1}g_{-ks_0/q} \Delta^{-1}[0,2r].
\end{align*}
This implies that 
$$ \mu_K(\widetilde{\Delta}'(r))\ll_{ s_0} r^{-1}\mu_K(\Delta^{-1}[0,2r])\asymp r^{-1}\mu_K(\Delta^{-1}[0,r]), $$
which is the desired upper bound. Here the last estimate follows from Theorem \ref{thm: measure estimates}.

To prove the desired lower bound, note that it suffices to show that {for all  sufficiently small $r>0$,}
\begin{align}\label{equ:disj1}
g_{-s}\Delta^{-1}[0,r]\cap \Delta^{-1}[0,r]=\emptyset,\qquad \forall\, C_1r\leq |s|\leq \e_0.
\end{align}
Here $\epsilon_0$ is a fixed sufficiently small constant such that Lemma \ref{lem:moremandes-1} and Lemma \ref{lem:moremandes-3} hold, and $C_1>0$ is an absolute large constant to be specified later. Indeed, assuming this, we obtain the desired lower bound. Note that 
$$g_{-s}{\Delta}'_r\cap {\Delta}'_r \subset g_{-s}\Delta^{-1}[0,r]\cap \Delta^{-1}[0,r], \qquad \forall\, 0<r<r_0,\, s\in \R. $$
Let $\epsilon_1=\min\{\epsilon_0,s_0\}$. Assuming \eqref{equ:disj1} we have for any sufficiently small $r>0$,
\begin{align*}
    \mu_K(\widetilde{\Delta}'(r))\geq \mu_K\left( \bigcup_{0\leq k< \lfloor \tfrac{\epsilon_1}{C_1r} \rfloor} g_{-kC_1 r} {\Delta}_r' \right) = \sum_{0\leq k< \lfloor \tfrac{\epsilon_1}{C_1r} \rfloor} \mu_K({\Delta}'_r)\gg_{s_0, c} r^{-1}\mu_K(\Delta^{-1}[0,r]),
\end{align*}
which is the desired lower bound.

Now we prove the disjointness statement \eqref{equ:disj1} for some absolute constant $C_1$. When $K=\Q(\sqrt{-1})$, we take $C_1=53$. In view of Proposition \ref{prop:ulbd-1} we only need to show that for all  sufficiently small $r>0$,
$$ g_{s}\Lambda \notin \Delta^{-1}[0,r], \,\, \forall \, \Lambda\in \overline{K}^{+}_r\cup \overline{K}^{-}_r, \qquad \forall\, 53 r\leq |s|\leq \e_0.$$
First assume that $\Lambda\in \overline{K}^{+}_r$. Then $\Lambda=gh_0\Z[i]^2$ where $g\in \overline{\cK}_r$ (see \eqref{eq: bar cKr}). If $s\in [53r,\epsilon_0]$, we have for any $k\in \{0,1,2\}$ that
$$ e^{-s}\left|g_{22}+|\omega|e^{(-\frac{1}{2}+\frac{2k}{3})\pi i}g_{21}\right|\leq e^{-s}(1+2r+24r)<e^{-r}.  $$
 If $s\in [-\epsilon_0,-53r]$, then for any $k\in \{0,1,2\}$ we have
$$e^{s}\left|g_{11}+|\omega|e^{(-\frac{1}{2}+\frac{2k}{3})\pi i}g_{12}\right|< e^{-r}. $$
On the other hand, as $|s|\leq \epsilon_0$, we have $g_s g\in \B_{G}(2\epsilon_0)\cR$. Hence, by Lemma \ref{lem:moremandes-1} and Remark \ref{rmk:expinequ}, $g_{s} \Lambda\notin \Delta^{-1}[0,r]$. If $\Lambda\in \overline{K}_r^-$, then $w_0^{-1}\Lambda\in \overline{K}_r^{+}$. By the above argument,  $g_{s}w_0^{-1}\Lambda\notin \Delta^{-1}[0,r]$ for any $53r\leq  |s|\leq \epsilon_0$. Hence, we have
$$ g_{-s}\Lambda=w_0 g_{s}w_0^{-1}  \Lambda\notin w_0\Delta^{-1}[0,r]=\Delta^{-1}[0,r].$$
This concludes the proof of the disjointness statement \eqref{equ:disj1} when $K=\Q(\sqrt{-1})$.

When $K=\Q(\sqrt{-3})$, we take $C_1=10^5$. In view of Proposition \ref{prop:ulbd-3}, we only need to show that for all sufficiently small $r>0$,
$$g_{s}\Lambda \notin \Delta^{-1}[0,r], \,\, \forall \, \Lambda\in \overline{K}^{+}_r\cup \overline{K}^{-}_r, \qquad \forall\,  10^5 r\leq |s|\leq \epsilon_0.$$
We first assume that $\Lambda\in \overline{K}_r^+$. Then $\Lambda=g\Z[j]^2$ for $g\in \overline{\cK}_r$. If $ 10^5 r\leq |s|\leq \epsilon_0$, we have 
$$ \min\{e^s |g_{11}|,e^{-s}|g_{22}|\}<(1+10^4 r)e^{-|s|}< e^{-r}. $$
This implies $g_s\Lambda\notin \Delta^{-1}[0,r]$ in view of \eqref{equ:nebexpdes-3} and \eqref{equ:ineq1-3}. Now if $\Lambda\in \overline{K}_r^-$, then $w_0^{-1}\Lambda\in \overline{K}_r^+$. By the above argument, we have for all $ 10^5 r\leq |s|\leq \epsilon_0$, $g_s w_0^{-1}\Lambda\notin \Delta^{-1}[0,r]$. Hence, 
$$ g_{-s}\Lambda=w_0 g_s w_0^{-1}\Lambda\notin w_0 \Delta^{-1}[0,r]=\Delta^{-1}[0,r], $$
which proves \eqref{equ:disj1} and completes the proof.
\end{proof}

\section{Effective single and double equidistribution}
\label{sec: effective equidistribution}

Let $X=G/\G_K=\SL_2(\C)/\SL_2(\cO_K)$ be as before. Let $\mu_K$ be the probability $G$-invariant measure on $X$. Recall that for $s\in \R$ and $z\in \C$, we have defined 
\begin{align}\label{equ:gpeleno}
g_s=\begin{pmatrix}
    e^s & 0\\
    0 &  e^{-s}
\end{pmatrix}, \quad u(z)=\begin{pmatrix}
    1 & z\\
    0 & 1
\end{pmatrix}\quad \text{and}\quad \Lambda_z=u(z)\cO_K^2. 
\end{align}
Set 
$$
\cY:=\{\Lambda_z: z\in \C\}\subset X.
$$
Note that for any $z,z'\in \C$, $\Lambda_{z}=\Lambda_{z'}$ if and only if $z-z'\in \cO_K$. Hence  $\cY$ is naturally identified with the torus $\C/\cO_K$. We denote by $\dd z$ the probability measure on $\cY$  induced, under this identification, by the normalized Lebesgue measure, $\Leb$, on $\C/\cO_K$. 

The key dynamical ingredients for our proof of the zero-one laws are the following effective single and double equidistribution results for the torus $\cY$ under the expanding translates by the flow $\{g_s\}$ as $s\to\infty$. To state these results, let us first introduce some notation. 

 Let $C_c^{\infty}(X)$ denote the space of smooth compactly supported functions on $X$. For any $m\in \N$ and for any $f\in C_c^{\infty}(X)$ define
\begin{align}\label{eq: Cm norm}
 \norm{f}_{C^m}:=\sum_{\deg (\mathcal{D})\le m} \sup_{x\in X}|\mathcal{D}f(x)|,
\end{align}
where the sum is over all monomials in a fixed basis of the Lie algebra of $G$, viewed as a real Lie algebra. {Recall the \textit{injectivity radius} of a point $x\in X$ is given by
$$ \inj(x):=\sup \{r>0: \text{the map } \B_{G}(r)\to X \text{ defined by } g\mapsto gx \text{ is injective}  \}, $$
where $\B_{G}(r)\subset G$ is the ball of radius $r$ centered at $\Id$.} Using the natural identification between $\C$ and $\R^2$, for any $\delta>0$ we use the notation $[-\delta,\delta]^2$ for the set $\{z\in \C: \max\{|\Re(z)|,|\Im(z)|\}\leq \delta\}$.

We first state the following effective equidistribution theorem due to Edwards \cite[Theorem $1'$]{Edwards17}. We tailor it here for our purpose:

\begin{Thm}\label{thm: single equidistribution}
Keep the assumptions as above. There exists a constant $c_1>0$ depending only on $K$ such that the following holds: Let $I$ be a parallelogram fundamental domain for $\C/\cO_K$, or $I={[-\delta,\delta]^2}$ for some $\delta\in (0,1]$. Then for any $x\in X$, 
$f\in C_c^{\infty}(X)$ and $s\geq 0$,
\begin{align}\label{equ:singleequ}
\frac{1}{\Leb(I)}\int_I f(g_s u(z)x) \dd z= \mu_K(f)  +O(e^{-c_1s} \inj(x)^{-1} E_I \norm{f}_{C^7}), 
\end{align}
where $E_I$ is a positive constant when $I$ is the chosen parallelogram fundamental domain, and $E_I=\delta^{-1}$ when $I={[-\delta,\delta]^2}$.
\end{Thm}

\begin{proof}
After identifying $\G_K\bk G$ with $G/\G_K$ by inversion,
it is straightforward to verify that the set $I$ satisfies all the assumptions in \cite[Theorem $1'$]{Edwards17}. Then the theorem here follows directly from that theorem. 
Here we take $T=2s$, where $T$ is the time parameter in the cited theorem.
Under inversion, right differentiation becomes left differentiation up to signs, so the $L^2$-Sobolev norm of order $7$ used there is bounded by a constant multiple of the $L^\infty$-Sobolev norm $\|f\|_{C^7}$ used here.
By a standard comparison in cusp coordinates, the height function
defined in \cite[Section~4.1]{Edwards17} satisfies
$1+\mathcal Y_{\G_K}(\G_Kg^{-1})\ll_K\inj(g\G_K)^{-1}$.
Thus the height factors in Edwards' error term are $\ll\inj(x)^{-1}$ since their exponents lie in $(0,1]$.
Moreover, the domain-dependent factor in the cited theorem
is $\ll E_I$. Indeed, for $I$ the
chosen parallelogram, all geometric parameters in
\cite[Theorem~1']{Edwards17} are fixed. For $I=[-\delta,\delta]^2$, using
the normalized arclength parametrization of its boundary,
we may take
$$
L=8\delta,\qquad R=2\sqrt{2}\,\delta,\qquad
c=4\delta,\qquad z_0=0,\qquad m(I)=4\delta^2
$$
in Edwards' notation. His explicit bound for the
domain-dependent constant is therefore $\ll_K\delta^{-1}$
for $0<\delta\leq1$, giving the stated factor $E_I$. 
Finally, with $T=2s$, we choose $c_1>0$ sufficiently small to absorb the polynomial factors in Edwards’ error term into the exponential decay.
\end{proof}

We also need the following effective double equidistribution theorem, whose proof follows the strategy of the proof of \cite[Theorem 1.2]{KleinbockShiWeiss2017}. For completeness, we include a self-contained proof.

\begin{Thm}\label{thm:effdoubequi}
Keep the assumptions as in Theorem \ref{thm: single equidistribution}. Let $c=\min\{1,\tfrac{c_1\gamma}{2(2+\gamma)}\}$ with $c_1>0$ as in Theorem \ref{thm: single equidistribution} and $\gamma$ as in Proposition \ref{prop: nondivergence} below.  Then the following holds: For any $s_1\geq s_2>0$ and any $f_1,f_2\in C_c^{\infty}(X)$, we have
\begin{align}\label{equ:doubleequ}
     \int_{\C/\cO_K} f_1(g_{s_1}\Lambda_z) f_2(g_{s_2}\Lambda_z)\, \dd z=& \mu_{K}(f_1) \mu_{K}(f_2)+O(e^{-c s_2}|\mu_{K}(f_1)|\|f_2\|_{C^7})\\ \nonumber
     &+O(e^{-c(s_1-s_2)} \|f_1\|_{C^7}\|f_2\|_{C^7} ).
\end{align}

\end{Thm}

As in \cite{KleinbockShiWeiss2017}, we will use the effective single equidistribution result (Theorem \ref{thm: single equidistribution}) to prove the above double equidistribution theorem. A key ingredient is a  \textit{quantitative non-divergence} result in this setting. Such non-divergence results were obtained by Dani, Margulis and Kleinbock in various settings \cite{DM91,KM98}; the particular version we need here  follows, for instance, from  \cite[Proposition 3.1]{LM23} or \cite[Lemma 3.3]{Shi21}.

\begin{Prop}\label{prop: nondivergence}
   There exist $C,\gamma>0$ depending only on $X$ such that given any $0<\epsilon\leq 1$ and any $s>0$,
    $$ \Leb\left(\{z\in \C/\cO_K: \inj(g_s\Lambda_z)\leq \epsilon \}\right)\leq C\epsilon^{\gamma} \Leb(\C/\cO_K).$$
\end{Prop}

\begin{proof}
We deduce the proposition from the general quantitative non-divergence estimate in \cite[Lemma 3.3]{Shi21}. 
Let $B\subset\C$ be a ball centered at $0$ containing a fixed bounded fundamental domain for $\C/\cO_K$, chosen so that $\Leb(B)\ll \Leb(\C/\cO_K)$, where the implied constant depends only on $X$.
Applying \cite[Lemma 3.3]{Shi21} for our setting, we obtain some $C,\gamma>0$ depending only on $X$ such that for any $s,\epsilon>0$,
$$ \Leb\left(\{z\in \C/\cO_K: \inj(g_s\Lambda_z)\leq \epsilon \}\right) \leq\Leb(\{z\in B: \inj(g_s\Lambda_z)\leq \epsilon\})\leq C\epsilon^{\gamma}\Leb(\C/\cO_K), $$
whence the proposition.
\end{proof}

Using Proposition \ref{prop: nondivergence} and Theorem \ref{thm: single equidistribution}, we prove Theorem \ref{thm:effdoubequi}.
\begin{proof}[Proof of Theorem \ref{thm:effdoubequi}]
Let $I$ be a parallelogram fundamental domain for $\C/\cO_K$ and recall that we have normalized the Lebesgue measure $\Leb$ on $\C$ such that $\Leb(I)=1$. Let  $\varphi=\chi_I$. Let $\delta=e^{-\tfrac{c_1(s_1-s_2)}{2(2+\gamma)}}$, where $c_1>0$ is the constant in Theorem \ref{thm: single equidistribution}, and $\gamma>0$ is the exponent in Proposition \ref{prop: nondivergence}. 
Let 
$$
h=\frac{\chi_{[-\delta,\delta]^2}}{\Leb([-\delta,\delta]^2)}
$$ and  $w=\frac{s_1+s_2}{2}$. By the change of variables $z\mapsto z+e^{-2w}\xi$, we have
\begin{align*}
    \int_I f_1(g_{s_1}\Lambda_z)f_2(g_{s_2}\Lambda_z)\, \dd z&=\int \varphi(z) f_1(g_{s_1}\Lambda_z)f_2(g_{s_2}\Lambda_z)\dd z\int h(\xi) \dd \xi\\
    &=\int \int h(\xi) f_1(g_{w-s_2}u(\xi)g_w \Lambda_z)\varphi(z+e^{-2w}\xi)f_2(u(e^{s_2-s_1}\xi)g_{s_2}\Lambda_z)\dd \xi \dd z.
\end{align*}
Note that for any $z\in I$ and $\xi\in \supp h$, we have
\begin{align}\label{eq: error by f2}
    |f_2(u(e^{s_2-s_1}\xi)g_{s_2}\Lambda_z)-f_2(g_{s_2}\Lambda_z)|\leq e^{-(s_1-s_2)} \|f_2\|_{C^1}.
\end{align}

{As $\Leb(I\Delta (I+e^{-2w}\xi))\ll_{I} e^{-2w}$ for any $\xi\in {[-1,1]^2}$, \eqref{eq: error by f2} implies }
\begin{align}\label{eq:double equid 1}
  \int_I f_1(g_{s_1}\Lambda_z)f_2(g_{s_2}\Lambda_z)\dd z=\int \Psi(z)\varphi(z)f_2(g_{s_2}\Lambda_z) \dd z+O(e^{-(s_1-s_2)}\norm{f_1}_{\infty}\|f_2\|_{C^1}),
\end{align}
where
$$ \Psi(z):= \int h(\xi) f_1(g_{w-s_2}u(\xi)g_w \Lambda_z) \dd \xi. $$
Let $\epsilon=e^{-\tfrac{c_1 (s_1-s_2)}{2(2+\gamma)}}=\delta$. Consider the set 
$$I_{\epsilon}:=\{z\in I: \inj(g_w \Lambda_z)\leq \epsilon\}.$$
Then by Proposition \ref{prop: nondivergence} and the assumption of the theorem, we have 
\begin{align}\label{eq: bound for I epsilon}
    |I_\epsilon|\ll \epsilon^{\gamma}{=e^{-\tfrac{c_1\gamma (s_1-s_2)}{2(2+\gamma)}}}.
\end{align}
For any $z\in I\smallsetminus I_{\epsilon}$, applying Theorem \ref{thm: single equidistribution} and by the choice of $\delta$, 
we obtain
\begin{align}\label{eq: approx of Psi}
\Psi(z)=\mu_K(f_1)+O(e^{-\tfrac{c_1\gamma (s_1-s_2)}{2(2+\gamma)}}\|f_1\|_{C^7}).
\end{align}
Therefore, \eqref{eq: bound for I epsilon} and \eqref{eq: approx of Psi} together imply that
\begin{align}\label{eq: double equid 2}
    \int \Psi(z)\varphi(z)f_2(g_{s_2}\Lambda_z) \dd z&=\mu_K(f_1) \int \varphi(z)f_2(g_{s_2}\Lambda_z) \dd z +O(e^{-\tfrac{c_1\gamma (s_1-s_2)}{2(2+\gamma)}} \norm{f_1}_{C^7}\norm{f_2}_{\infty}).
\end{align}
Finally, we apply Theorem \ref{thm: single equidistribution} again to obtain
\begin{align}\label{eq: double equid 3}
\int \varphi(z)f_2(g_{s_2}\Lambda_z) \dd z= \mu_K(f_2)+ O(e^{-c_1s_2}\norm{f_2}_{C^7}).
\end{align}
Taking $c=\min\{1,\tfrac{c_1\gamma}{2(2+\gamma)}\}\leq \min\{1, c_1\}$, the theorem follows by combining \eqref{eq:double equid 1}, \eqref{eq: double equid 2}, \eqref{eq: double equid 3}, and noting that for any $f\in C_c^{\infty}(X)$, 
we have
\begin{displaymath}
\max\{|\mu_K(f)|,\norm{f}_{\infty}\}\leq \|f\|_{C^7}.\qedhere 
\end{displaymath}
\end{proof}

\begin{remark}\label{rmk:uniformc}
Since $c=\min\{1,\tfrac{c_1\gamma}{2(2+\gamma)}\}<c_1$, clearly \eqref{equ:singleequ} still holds with $c_1$ replaced by $c$.
\end{remark}

\section{Proof of the zero-one laws}
\label{sec: proof of zero-one laws}
\subsection{Reduction to dynamics}

In this section we prove the zero-one laws in Theorem \ref{thm: zero-one laws}.  We follow closely the treatment in \cite[Section 4]{StrombergssonYu2026}. Here the main ingredients are the two effective equidistribution results in Section \ref{sec: effective equidistribution}, together with a suitable disjointness statement. A minor complication in our setting is that the lower targets we construct are not nested in general. We therefore require a somewhat stronger disjointness statement than \cite[Proposition 4.4]{StrombergssonYu2026}; see Proposition \ref{prop:disjoint} below. 

We first give a simple reduction. 
\begin{Lem}\label{lem:redlem1}
When proving Theorem \ref{thm: zero-one laws}, 
there is no loss of generality to also assume
\begin{align}\label{equ:lowerbdpsi}
\psi(t)\geq \frac{c_K}{2t}\qquad \forall\, t\geq 1.
\end{align}  
\end{Lem}
The proof follows the same argument as in \cite[Lemma 3.4]{StrombergssonYu2026} with $1/t$ and $1/(2t)$ replaced by $c_K/t$ and $c_K/(2t)$, respectively. In particular, one considers 
\begin{align*}
    \psi_1(t):=\max\left\{\psi(t), \frac{c_K}{2t}\right\},
\end{align*}
for which $F_{\psi_1}(t)=\min\{F_{\psi}(t), \frac12\}$. We omit the details.

Recall that for any $s\in \R$, $z\in \C$, we defined $g_s$, $u(z)$ and $\Lambda_z$ in \eqref{equ:gpeleno}. 
The following version of Dani's correspondence in our setting gives a dynamical interpretation of $(\psi,K)$-Dirichlet numbers. 
\begin{Prop}\label{prop:dani}
Let $\psi: [1,\infty)\to (0,1)$ be a continuous decreasing function satisfying $\psi(t)<c_K/t$ for all $t\geq 1$, \eqref{equ:lorecondpsi} and \eqref{equ:lowerbdpsi}. Then there exists a continuous function $\ttr=\ttr_{\psi}:  [s_0,\infty)\to (0,\frac{\log 2}{2}]$ {with $s_0=\frac12\log(\frac{1}{\psi(1)})$} such that 
\begin{align}\label{equ:lorecondr}
\forall\, \tau>0,\, \exists\, C_{\ttr,\tau}\geq 1\ \text{such that }\ \ttr(s_2)\leq C_{\ttr,\tau}\ttr(s_1)\qquad \text{whenever $s_0\leq s_1\leq s_2\leq  s_1+\tau$},
\end{align}
and
\begin{align}\label{equ:daniliminf}
z\in \DI_K(\psi)\quad \Leftrightarrow\quad \Delta(g_s\Lambda_z)> \ttr(s)\qquad \text{for all sufficiently large $s$}. 
\end{align}
Moreover, for any $\alpha, \tau>0$ and $\beta \geq 0$  we have
\begin{align}\label{equ:dieq}
\sum_k k^{-1}F_{\psi}(k)^{\alpha}\log^{\beta}\left(1+\frac{1}{F_{\psi}(k)} \right)=\infty\quad \Leftrightarrow \quad \sum_k \ttr(\tau k)^{\alpha}\log^{\beta}\left( 1+\frac{1}{\ttr(\tau k)} \right )=\infty.
\end{align}
Here and below, sums involving $\ttr(\tau k)$ are understood to start from integers $k$ for which $\ttr(\tau k)$ is defined, that is, $\tau k\geq s_0$.
\end{Prop}

\begin{proof}
The function $\ttr=\ttr_{\psi}$ is uniquely determined by $\psi$ via the relations
\begin{align}\label{equ:deredani}
e^s\psi(t)=e^{-s}t=:r_Ke^{-\ttr(s)}.
\end{align}
Using the relation that $c_K=r_K^2$ (see \eqref{equ:costrel}), this is equivalent to
\begin{align*}
    s=\frac12\log\left(\frac{t}{\psi(t)}\right)\quad \text{and}\quad e^{-2\ttr(s)}=\frac{t\psi(t)}{c_K}.
\end{align*}
Thus, by \eqref{equ:lowerbdpsi} and the assumption $\psi(t)<c_K/t$, the parameter $t\in [1,\infty)$ corresponds to $s\in [s_0,\infty)$ with $s_0=\frac12\log(\frac{1}{\psi(1)})$, and 
\begin{align*}
    0<\ttr(s)\leq \frac{\log 2}{2}.
\end{align*}
Then \eqref{equ:daniliminf} follows from the definition of $\DI_K(\psi)$ and $\Delta$, together with $c_K=r_K^2$. 

The proofs of \eqref{equ:lorecondr} and \eqref{equ:dieq} for $\tau=1$ follow the same argument as in \cite[Proposition 3.6]{StrombergssonYu2026} with $1-t\psi(t)$ there replaced by $F_{\psi}(t)=1-\frac{t\psi(t)}{c_K}$. Moreover, since $0<F_{\psi}(t)\leq \frac12$ and $0<\ttr(s)\leq \frac{\log 2}{2}$, the use of $\log(1+\frac{1}{F_{\psi}})$ and $\log(1+\frac{1}{\ttr})$ in \eqref{equ:dieq} in place of $\log(\frac{1}{F_{\psi}})$ and $\log(\frac{1}{\ttr})$ (which were used in \cite[(3.9)]{StrombergssonYu2026}) does not affect the corresponding convergence or divergence criterion. 

More generally, note that \eqref{equ:lorecondr} for general $\tau>0$ follows from the special case of \eqref{equ:lorecondr} for $\tau=1$ and a simple iteration argument. 
On the other hand, in view of the special case of \eqref{equ:dieq} for $\tau=1$, to show \eqref{equ:dieq} for  general $\tau>0$, we need to show the following equivalence: For any $\tau>0$,
\begin{align}\label{equ:itmequiva}
\sum_k \ttr(k)^{\alpha}\log^{\beta}\left( 1+\frac{1}{\ttr(k)} \right )=\infty\quad \Leftrightarrow \quad \sum_k \ttr(\tau k)^{\alpha}\log^{\beta}\left( 1+\frac{1}{\ttr(\tau k)} \right )=\infty.
\end{align}
For this, note that by \eqref{equ:lorecondr} one easily verifies that for any $\tau>0$, 
\begin{align}
\sum_k \ttr(\tau k)^{\alpha}\log^{\beta}\left( 1+\frac{1}{\ttr(\tau k)} \right )=\infty\quad \Leftrightarrow \quad \int_{s_0}^{\infty} \ttr(x)^{\alpha}\log^{\beta}\left( 1+\frac{1}{\ttr(x)} \right )\, \mathrm{d}x=\infty.
\end{align}
Then \eqref{equ:itmequiva} for general $\tau>0$ follows immediately from the above equivalence. 
\end{proof}

In view of Dani's correspondence in Proposition \ref{prop:dani}, the proof of the desired zero-one laws in Theorem \ref{thm: zero-one laws} can be reduced to a shrinking target problem on the homogeneous space $X$. For any $\tau, r>0$, define 
\begin{align}\label{eq: upper bound set for smooth approx}
    \widetilde{\Delta}_\tau(r):=\bigcup_{0\leq s<1}g_{-\tau s}\Delta^{-1}[0,r]. 
\end{align}
As a direct consequence of Dani's correspondence, we have the following relation bounding $\mathbf{DI}_K(\psi)^c$ from above and below by two limsup sets using the target sets $\widetilde{\Delta}_\tau(r)$.

\begin{Lem}\label{lemma: limsupbd}
Let $\psi$, $\ttr=\ttr_{\psi}$ be as in Proposition \ref{prop:dani}. Then we have for any $\tau>0$,
\begin{align}\label{equ:limsupbd}
\limsup_{k\to\infty}(g_{-\tau k}\widetilde{\Delta}_\tau({C_{\ttr, \tau}^{-1}\ttr({\tau}(k+1))})\cap \cY)\subset \mathbf{DI}_K(\psi)^c\subset \limsup_{k\to\infty}(g_{-\tau k}\widetilde{\Delta}_\tau({C_{\ttr, \tau}\ttr({\tau} k)})\cap \cY).
\end{align}
Here $C_{\ttr, \tau}\geq 1$ is as in \eqref{equ:lorecondr}.
\end{Lem}

\begin{proof}
Note that \eqref{equ:daniliminf} is equivalent to 
\begin{align*}
z\notin \mathbf{DI}_K(\psi)\quad \Leftrightarrow\quad \Delta(g_s\Lambda_z)\leq \ttr(s)\quad \text{for an unbounded set of $s>s_0$}.
\end{align*}
Then one easily sees that the above right hand side is equivalent to 
\begin{align*}
g_{\tau k}\Lambda_z\in \bigcup_{0\leq s<1}g_{-\tau s}\Delta^{-1}[0, \ttr(\tau(k+s))]\quad \text{for infinitely many positive integers $k$.}
\end{align*}
Using the natural identification between $\cY$ and $\C/\cO_K$ 
this is equivalent to saying that 
\begin{align}\label{equ:nodipsiint}
    \mathbf{DI}_K(\psi)^c=\limsup_{k\to\infty}(g_{-\tau k}\bigcup_{0\leq s<1}g_{-\tau s}\Delta^{-1}[0, \ttr(\tau(k+s))]\cap \cY).
\end{align}
By \eqref{equ:lorecondr} we have 
\begin{align*}
\bigcup_{0\leq s<1}g_{-\tau s}\Delta^{-1}[0, C^{-1}_{\ttr,\tau}\ttr(\tau(k+1))]\subset \bigcup_{0\leq s<1}g_{-\tau s}\Delta^{-1}[0, \ttr(\tau(k+s))]\subset \bigcup_{0\leq s<1}g_{-\tau s}\Delta^{-1}[0, C_{\ttr, \tau}\ttr(\tau k)].
\end{align*}
Note that the above left and right end sets are $\widetilde{\Delta}_\tau({C_{\ttr, \tau}^{-1}\ttr({\tau}(k+1))})$ and $\widetilde{\Delta}_\tau({C_{\ttr, \tau}\ttr({\tau} k)})$. 
Then \eqref{equ:limsupbd} follows immediately from \eqref{equ:nodipsiint} and the above two inclusion relations.
\end{proof}

In view of the relation \eqref{equ:limsupbd},  to determine whether $\mathbf{DI}_K(\psi)^c$ is of full or null Lebesgue measure, it suffices to establish zero-one laws for the two  limsup sets in \eqref{equ:limsupbd}. The key dynamical input here is the effective single and double equidistribution theorems in Theorem \ref{thm: single equidistribution} and Theorem \ref{thm:effdoubequi}. 
This is exactly our approach when showing $\mathbf{DI}_K(\psi)^c$ is of zero Lebesgue measure. To show $\mathbf{DI}_K(\psi)^c$ is of full Lebesgue measure, as mentioned in the introduction, we need a certain disjointness statement to control the short-range mixing contribution. To establish this disjointness statement, we will replace the target sets above by suitably chosen subsets which we will describe in the next subsection.

\subsection{A disjointness statement}
We let $K=\Q(\sqrt{-1})$ or $K=\Q(\sqrt{-3})$ for the remainder of this section. Fix a parameter
\begin{align}\label{equ:taudef}
\tau=\tau_K:=
\begin{cases}
\log(2+\sqrt 3) & K=\mathbb Q(\sqrt{-1}),\\
1 & K=\mathbb Q(\sqrt{-3}).
\end{cases}
\end{align}
Define also 
\begin{align}\label{equ:fkdef}
f_K(x):=
\begin{cases}
x^4 & K=\mathbb Q(\sqrt{-1}),\\
x^2\log(1+\frac{1}{x}) & K=\mathbb Q(\sqrt{-3}),
\end{cases}\qquad \forall\, x>0.
\end{align}
For later use, note that $f_K$ is strictly increasing on $(0,\infty)$ and satisfies the following two bounds:
\begin{align}\label{equ:fkbd}
f_K(Cx)\asymp_C f_K(x),\qquad \forall\, C>0, x>0. 
\end{align}
and
\begin{align}\label{equ:fkest2}
    \frac{f_K(x)}{f_K(y)}\leq \frac{x}{y},\qquad \forall\, 0<x<y.
\end{align}

Fix a parameter $0<s_*<1$ to be determined. Let $r_{\e_0}$ be as in Proposition \ref{prop:ulbd-1} if $K=\Q(\sqrt{-1})$ and be as in Proposition \ref{prop:ulbd-3} if $K=\Q(\sqrt{-3})$.
For any $0<r<r_{\e_0}$ define
\begin{align}\label{eq: tilde delta tau}
\widetilde{\Delta}'_{\tau,s_*}(r):=\bigcup_{0\leq s<s_*}g_{-\tau s}\underline{K}_r,
\end{align}
where $\underline{K}_r$ is as in \eqref{equ:lowbdkr} if $K=\Q(\sqrt{-1})$ and is as in \eqref{equ:lowbdkr-3} if $K=\Q(\sqrt{-3})$.
In view of the relations \eqref{equ:incluupplow} and \eqref{equ:incluupplow-3}, we clearly have that 
\begin{align}\label{equ:deltaincl}
\widetilde{\Delta}'_{\tau,s_*}(r)\subset \widetilde{\Delta}_\tau(r),\qquad (\forall\, 0<r<r_{\e_0}).
\end{align}
{Moreover, from the proofs of Propositions \ref{prop:ulbd-1} and \ref{prop:ulbd-3} we see that there exists some $r_0>0$ such that 
\begin{align*}
\underline{K}_r\subset \Delta^{-1}[0,r]\quad \text{and}\quad \mu_K(\underline{K}_r)\gg \mu_K(\Delta^{-1}[0,r]),\qquad \forall\, 0<r<r_0.
    \end{align*}
From these two relations, together with Proposition \ref{prop: measure estimates the thickenings} we have
\begin{align}\label{equ:thickeningmeest}
    \mu_K(\widetilde{\Delta}'_{\tau,s_*}(r))\asymp \mu_K(\widetilde\Delta_{\tau}(r))
\asymp f_K(r),
\qquad \text{as } r\to 0^+.
\end{align}
}

We will use the sets $\widetilde{\Delta}'_{\tau,s_*}(r)$ to study the divergence case of the shrinking target problem. 
A key property that these sets satisfy is the following disjointness statement. This is an analogue of the disjointness statement in \cite[Proposition 4.4]{StrombergssonYu2026}. However, since the sets considered here are not nested, we need to prove the following slightly stronger version.
\begin{Prop}\label{prop:disjoint}
Let $\tau=\tau_K$ be as in \eqref{equ:taudef}. Then 
there exist sufficiently small $0<s_*, r_*<0.01$ such that for any $0<r<r_*$ and for any $0<r_1,r_2\leq r$, we have 
\begin{align*}
  g_{\tau k}\widetilde{\Delta}'_{\tau,s_*}(r_1)\cap \widetilde{\Delta}'_{\tau,s_*}(r_2)=\emptyset,\qquad \forall\, 1\leq k\leq J_r:=\left \lfloor{\frac{1}{4\tau}\log\left(\frac{1}{r}\right)}\right \rfloor.
\end{align*}
Here the range $1\leq k\leq J_r$ is understood to mean $k\in\Z$ with $1\leq k\leq J_r$.
\end{Prop}

\begin{proof}
Let $0<s_*,r_*<0.01$ be parameters to be determined. Let $0<r<r_*$ and take $0<r_1,r_2\leq r$.  We prove by contradiction. Suppose not, then there exists some $1\leq k\leq J_r$ such that $g_{\tau k}\widetilde{\Delta}'_{\tau,s_*}(r_1)\cap \widetilde{\Delta}'_{\tau,s_*}(r_2)\neq\emptyset$. That is, there exist $\Lambda\in X$, $0\leq s, s'<s_*$ such that 
$
g_{\tau(s-k)}\Lambda\in \underline{K}_{r_1}$ and $ g_{\tau s'}\Lambda\in \underline{K}_{r_2}$. Writing $\Lambda':= g_{\tau s'}\Lambda$ and $\delta:=s-s'\in (-s_*,s_*)$, these two conditions are equivalent to 
\begin{align}\label{equ:intersecond}
\Lambda'\in \underline{K}_{r_2}\quad \text{and}\quad g_{\tau(\delta-k)}\Lambda'\in \underline{K}_{r_1}.
\end{align}
We first consider the case $K=\Q(\sqrt{-1})$, so that $\underline{K}_r$ is given as in \eqref{equ:lowbdkr} and $\tau=\log(2+\sqrt{3})$. Let $\e_0$ and $r_{\epsilon_0}$ be as in Proposition  \ref{prop:ulbd-1}. 
Let $x_0:=h_0\Z[i]^2\in X$ and let $\inj(x_0)$ be its injectivity radius. Fix $0<\eta<\inj(x_0)$. After possibly reducing $\e_0$, we can take $s_*>0$ sufficiently small so that  
$$
\bigcup_{|t|<s_*}g_{-\tau t}\B_{G}(\e_0)\subset \B_{G}(\eta).
$$  
Then the map 
\begin{align}\label{equ:injmaps}
\bigcup_{|t|<s_*}g_{-\tau t}\B_{G}(\e_0) \to X,\qquad \mathrm{g}\mapsto   \mathrm{g}h_0\Z[i]^2
\end{align} 
is injective.
Let $r_*:=\min\{r_{\e_0}, \e_0^2\}$. 
By definition, there exist $g=\begin{pmatrix}
g_{11} & g_{12}\\
g_{21} & g_{22}\end{pmatrix}, g'=\begin{pmatrix}
g'_{11} & g'_{12}\\
g'_{21} & g'_{22}\end{pmatrix}\in \B_{G}(\e_0)$ such that 
\begin{align}\label{equ:resr2}
1-\frac{r_2}{2}<|g_{11}|<1+\frac{r_2}{2},\quad   \frac{r_2}{2(2+\sqrt{3})}<|g_{21}|, |g_{12}|<\frac{r_2}{2},
\end{align}
and
\begin{align}\label{equ:resr1}
1-\frac{r_1}{2}<|g'_{11}|<1+\frac{r_1}{2}, \quad  \frac{r_1}{2(2+\sqrt{3})}<|g_{21}'|, |g'_{12}|<\frac{r_1}{2},
\end{align}
such that $\Lambda'=gh_0\Z[i]^2$ and $g_{\tau(\delta-k)}\Lambda'=g'h_0\Z[i]^2$. Using the identity $g_{\tau}h_0\Z[i]^2=h_0\Z[i]^2$ we get
\begin{align}\label{equ:matrixiden}
gh_0\Z[i]^2=\Lambda'=g_{\tau(k-\delta)}g'h_0\Z[i]^2=g_{-\tau \delta}g_{\tau k}g'g_{-\tau k}h_0\Z[i]^2=g_{-\tau \delta}\begin{pmatrix} g'_{11} & e^{2k\tau}g_{12}'\\ e^{-2k\tau}g_{21}' &g_{22}'\end{pmatrix}h_0\Z[i]^2.
\end{align}
From this identity, the injectivity of the map in \eqref{equ:injmaps} and the facts that $g, g'\in \B_G(\e_0)$ and $|\delta|<s_*$, we see that 
\begin{align}\label{equ:contrad-1}
g=g_{-\tau \delta}\begin{pmatrix} g'_{11} & e^{2k\tau}g_{12}'\\ e^{-2k\tau}g_{21}' &g_{22}'\end{pmatrix}=\begin{pmatrix} e^{-\delta\tau}g'_{11} & e^{(2k-\delta)\tau}g_{12}'\\ e^{-(2k-\delta)\tau}g_{21}' & e^{\delta\tau}g_{22}'\end{pmatrix}\qquad \text{as long as $e^{2k\tau}|g_{12}'|<\e_0$.}
\end{align}  
In view of \eqref{equ:resr1} and the assumptions $1\leq k\leq J_r$ and $r_1\leq r$, the condition $e^{2k\tau}|g_{12}'|<\e_0$  can be guaranteed as $e^{2 J_r\tau}r/2<\e_0$, i.e. $J_r<\frac{1}{2\tau}\log(\frac{2\e_0}{r})$. The latter condition is satisfied by our choice of  $J_r=\lfloor\frac{1}{4\tau}\log(\frac{1}{r})\rfloor$ provided that $0<r< \e_0^2$. Thus the identity in \eqref{equ:contrad-1} holds. Now note that by \eqref{equ:resr2} and \eqref{equ:resr1} we have 
\begin{align}\label{equ:conbds}
 e^{-\tau}=\frac{1}{2+\sqrt{3}} <  \frac{|g_{12}|}{|g_{21}|}, \frac{|g'_{12}|}{|g'_{21}|}< 2+\sqrt{3}=e^{\tau}.
\end{align}

These two estimates, together with \eqref{equ:contrad-1}, imply
\begin{align*}
    \frac{|g_{12}|}{|g_{21}|}=e^{(4k-2\delta)\tau}\frac{|g_{12}'|}{|g_{21}'|}>e^{2\tau}\times e^{-\tau}=e^{\tau}.
\end{align*}
We have thus obtained a contradiction since $\frac{|g_{12}|}{|g_{21}|}<e^{\tau}$ by \eqref{equ:conbds}. This finishes the proof of the desired disjointness statement when $K=\Q(\sqrt{-1})$.

Next, we assume $K=\Q(\sqrt{-3})$, so that $\underline{K}_r$ is given as in \eqref{equ:lowbdkr-3} and $\tau=1$. By the definition of $\underline{K}_r$, any lattice in $\underline{K}_r$ contains a nonzero vector in the region
$$
\cR_r:=\left\{(z_1,z_2)\in \C^2: 1-\frac{r}{4}<|z_1|<1+\frac{r}{4},\, r<|z_2|<\sqrt{r} \right\}.
$$ 
Let $\e_0$ and $r_{\e_0}$ be as in Proposition \ref{prop:ulbd-3}. Let $r_*:=\min\{r_{\e_0},0.005\}$ and let $s_*=0.005$. Hence $|\delta|< 0.005$.
We apply the conditions in \eqref{equ:intersecond} to get a contradiction.
First, note that $\Lambda'\in \underline{K}_{r_2}$ implies that 
there exists a nonzero $\bm{v}=(v_1,v_2)\in \Lambda'$ such that $\bm{v}\in \cR_{r_2}$. On the other hand, $g_{\delta-k}\Lambda'\in \underline{K}_{r_1}\subset \Delta^{-1}[0,r_1]$ implies that $\|g_{\delta-k}\bm{v}\|\geq e^{-r_1}$ (see \eqref{equ:svec}). 
Since $\max\{|\delta|, r_1,r_2\}<0.005$, we have 
\begin{align*}
    e^{\delta-k}|v_1|<e^{-0.99}\left(1+\frac{r_2}{4}\right)<e^{-r_1}.
\end{align*}
This forces $e^{k-\delta}|v_2|\geq e^{-r_1}$, but we then get a contradiction since 
$$
e^{k-\delta}|v_2|<e^{J_r+1}\sqrt{r_2}\leq er^{1/4}<e^{-r_1}.
$$
Here the last inequality holds since $0<r_1,r_2\leq r\leq 0.005$. This finishes the proof of this proposition  when $K=\Q(\sqrt{-3})$.
\end{proof}

\subsection{Smooth approximation and estimates on norms}

We also need the following result on smooth approximation to apply Theorems \ref{thm: single equidistribution} and \ref{thm:effdoubequi}. 
\begin{Prop}\label{prop: smooth approximation}
Up to further reducing the two parameters $0<r_*, s_*<0.01$ as in Proposition \ref{prop:disjoint}, there exists some $L>0$ such that for any $0<r<r_*$, there exist smooth functions $\underline{\phi}_r, \overline{\phi}_r\in C^{\infty}_c(X)$ satisfying
\begin{align*}
\chi_{\widetilde{\Delta}_{\tau,s_*}''(r)}\leq \underline{\phi}_r\leq \chi_{\widetilde{\Delta}_{\tau,s_*}'(2r)},
\quad \chi_{\widetilde{\Delta}_{\tau}(r)}\leq \overline{\phi}_r\leq \chi_{\widetilde{\Delta}_{\tau}({2r})}\quad \text{and}\quad \max\{\|\underline{\phi}_r\|_{C^7}, \|\overline{\phi}_r\|_{C^7}\}\ll r^{-L},
\end{align*}
where $\widetilde{\Delta}_{\tau}(r)$ and $\widetilde{\Delta}_{\tau,s_*}'(r)$ are as in \eqref{eq: upper bound set for smooth approx} and \eqref{eq: tilde delta tau} respectively, and 
\begin{align}\label{eq: low bound of smooth approx}
\widetilde{\Delta}_{\tau, s_*}''(r):=\bigcup_{0\leq s<s_*} g_{-\tau s}\underline{K}_r'.
\end{align}
Here in the case $K=\Q(\sqrt{-1})$,
$$ \underline{K}_r':=\left\{gh_0\Z[i]^2:g\in \B_{G}(\epsilon_0/4),|g_{11}|\in (1-r/4,1+r/4),\,   r/(1+\sqrt{3})<|g_{12}|,|g_{21}|<r/2\right\}, $$
and in the case $K=\Q(\sqrt{-3})$,
$$ \underline{K}_r':=\left\{ g\Z[j]^2: \begin{array}{l} g\in \B_{G}(\e_0/4),\, |g_{11}|\in (1-r/4, 1+r/4),\, 0<|g_{12}g_{21}|< r/4,\\ 4r<|g_{21}|<\sqrt{r}/2,\, r<|g_{12}|,\, \mathrm{Arg}(-g_{21}),\mathrm{Arg}(\overline{g_{12}})\in (\tfrac{27\pi}{72}, \tfrac{33\pi}{72})\end{array}\right\}. $$
Moreover, in either case, we have
\begin{align}\label{equ:douprest}
   \mu_K(\widetilde{\Delta}_{\tau,s_*}''(r))\asymp \mu_K(\widetilde{\Delta}_{\tau,s_*}'(r)),\qquad \text{ as $r\to 0^+$}. 
\end{align}
\end{Prop}

\begin{proof}
We only construct $\underline{\phi}_r$ and show the desired properties for it. The construction of $\overline{\phi}_r$ is similar and simpler (see \cite[Lemma 6.5]{KleinbockStrombergssonYu2022}).

For any $r>0$ small, we define a neighborhood of $\Id$ in $G$ by
$$ \mathcal{U}_r=\{g\in G:\norm{g-\Id}\leq 10^{-10}r, \, \norm{g^{-1}-\Id}\leq 10^{-10}r\}. $$
The definitions of $\underline{\phi}_r$ are different for $K=\Q(\sqrt{-1})$ and $K=\Q(\sqrt{-3})$.

For the case $K=\Q(\sqrt{-1})$, we define the set
$$ S_r=\left\{gh_0\Z[i]^2:g\in \B_{G}(\epsilon_0/2),\, |g_{11}|\in (1-r/2,1+r/2),  \, r/3<|g_{12}|,|g_{21}|<3r/4\right\}. $$

For the case $K=\Q(\sqrt{-3})$, on the other hand, we define 
$$ 
S_r=\left\{ g\Z[j]^2: \begin{array}{l} g\in \B_{G}(\e_0/2),\, |g_{11}|\in (1-r/3, 1+r/3),0<|g_{12}g_{21}|< r/3,\\ 3r<|g_{21}|<\sqrt{r},\, r/2<|g_{12}|,\, \mathrm{Arg}(-g_{21}),\mathrm{Arg}(\overline{g_{12}})\in (\tfrac{26\pi}{72}, \tfrac{34\pi}{72})\end{array}\right\}. 
$$
We claim that in both cases,  we have
\begin{align}\label{eq: inclusion for S_r}
    h\underline{K}_r'\subset S_r, \,\, hS_r\subset \underline{K}_{2r}, \qquad \forall \, h\in \mathcal{U}_r.
\end{align}
Let us prove \eqref{eq: inclusion for S_r} separately for these two cases.

\underline{Case: $K=\Q(\sqrt{-1})$}. Let $\Lambda\in \underline{K}_r'$. Then $\Lambda=gh_0\Z[i]^2$ for $g$ satisfying the corresponding conditions. Write $hg=g'=(g'_{ij})$. As $g'_{ij}=\sum_{k=1}^2h_{ik}g_{kj}$, by the choice of $h$ and $g$, we immediately obtain $h\Lambda\in S_r$. The second inclusion of \eqref{eq: inclusion for S_r} is proved similarly.

\underline{Case: $K=\Q(\sqrt{-3})$}. Let $\Lambda\in \underline{K}_r'$. Then $\Lambda=g\Z[j]^2$ for $g$ satisfying the conditions in the definition of $\underline{K}_r'$. Again we write $hg=g'=(g'_{ij})$. By the choice of $h$ and $g$, it is straightforward to verify that
$$ |g_{11}'|\in (1-\tfrac{r}{3},1+\tfrac{r}{3}), \, 0<|g_{12}'g_{21}'|<\tfrac{r}{3},\, 3r<|g_{21}'|<\sqrt{r},\, \tfrac{r}{2}<|g_{12}'|. $$
We now verify the desired angular restrictions for $g_{12}'$ and $g_{21}'$.
By comparing the angular conditions in the definitions of $\underline{K}_r'$ and $S_r$, it suffices to show
\begin{align}
   \max\{ |\mathrm{Arg}(g'_{12}/g_{12})|, |\mathrm{Arg}(g'_{21}/g_{21})|\}<\frac{\pi}{72}.
\end{align}
By direct computation we have 
\begin{align*}
    \mathrm{Arg}(g_{12}'/g_{12})=\mathrm{Arg}(h_{11}+h_{12}g_{22}g_{12}^{-1})=\mathrm{Arg}(h_{11})+\mathrm{Arg}(1+h_{12}g_{22}g_{12}^{-1}h_{11}^{-1}).
\end{align*}
Moreover, by the assumptions that  $h\in \mathcal{U}_r$ and $g$ satisfies the defining conditions for $ \underline{K}_r'$, we have 
$$
|\mathrm{Arg}(h_{11})|<10^{-5}r\quad \text{and}\quad |h_{12}g_{22}g_{12}^{-1}h_{11}^{-1}|< 10^{-10}r\times \frac{1+r/4}{1-r/4}\times \frac{1}{r}\times \frac{1}{1-10^{-10}r}<10^{-5}.
$$
Now applying the following simple angular inequality:
\begin{align}
    |\mathrm{Arg}(1+z)|\leq 2|z|,\qquad \forall\, 0<|z|<0.1,
\end{align}
we get 
\begin{align*}
    \left|\mathrm{Arg}(g_{12}'/g_{12})\right|<|\mathrm{Arg}(h_{11})|+|\mathrm{Arg}(1+h_{12}g_{22}g_{12}^{-1}h_{11}^{-1})|<10^{-5}r+2\cdot 10^{-5}<\frac{\pi}{72}
\end{align*}
as desired. The estimate for $\mathrm{Arg}(g'_{21}/g_{21})$ is similar. Indeed, using again that $h\in\mathcal U_r$ and that $g$ satisfies the defining conditions for $\underline{K}_r'$, we have
\begin{align*}
\left|\mathrm{Arg}\left(g'_{21}/g_{21}\right)\right|&\leq |\mathrm{Arg}(h_{22})|+\left|\mathrm{Arg}\left(1+h_{21}g_{11}g_{21}^{-1}h_{22}^{-1}\right)\right|\\
&< 10^{-5}r+2\times 10^{-10}r\times (1+r/4)\times \frac{1}{4r}\times \frac{1}{1-10^{-10}r}<\frac{\pi}{72}.
\end{align*}
The second inclusion of \eqref{eq: inclusion for S_r} can be proved similarly.

Now let $\theta_r$ be a smooth non-negative bump function supported on $\mathcal{U}_{r/100}$ such that
\begin{align*}
    \int_G \theta_r\dd m_K=1\quad \text{and}\quad 
    \norm{\theta_r}_{C^7}\ll r^{-L}
\end{align*}
for some $L>0$ depending only on $G$, where $\norm{\theta_r}_{C^7}$ is defined  as in \eqref{eq: Cm norm}. Here $m_{K}$ is the Haar measure of $G$ that locally agrees with $\mu_{K}$. Let 
\begin{align*} 
\tilde{S}_{\tau}(r)=\bigcup_{0\leq s<s_*} g_{-\tau s}S_r,
\end{align*}
and
\begin{align*}
\underline{\phi}_r(\Lambda)=\theta_r*\chi_{\tilde{S}_{\tau}(r)}(\Lambda)=\int_G\theta_r(h)\chi_{\tilde{S}_{\tau}(r)}(h^{-1}\Lambda)\,\dd m_K(h),\qquad (\Lambda \in X).
\end{align*}
For any $\Lambda\in \widetilde{\Delta}_{\tau, s_*}''(r)$, there exists $s\in [0,s_*)$ such that $g_{\tau s}\Lambda\in \underline{K}_r'$. Given any $h\in \mathcal{U}_{r/100}$, {since $0\leq s< s_*<0.01$}, we have by direct computation that $g_{\tau s}h^{-1} g_{-\tau s}\in \mathcal{U}_{r}$. Hence, by the first inclusion of \eqref{eq: inclusion for S_r}, we have
$$ g_{\tau s} {h^{-1}}\Lambda=g_{\tau s}h^{-1} g_{-\tau s} g_{\tau s}\Lambda\in \mathcal{U}_r\underline{K}_r'\subset S_r,  $$
which implies $\chi_{\widetilde{\Delta}_{\tau, s_*}''(r)}\leq \underline{\phi}_r$. 

For any $\Lambda$ such that $\underline{\phi}_r(\Lambda)>0$, there exist $h\in \mathcal{U}_{r/100}$ and $s\in [0,s_*)$ such that $g_{\tau s}{h^{-1}}\Lambda\in S_r$. Then 
$$g_{\tau s}\Lambda=g_{\tau s}hg_{-\tau s} g_{\tau s}h^{-1}\Lambda\in g_{\tau s}hg_{-\tau s}S_r\subset \mathcal{U}_r S_r\subset \underline{K}_{2r}, $$
where the last inclusion follows by the second inclusion of \eqref{eq: inclusion for S_r}. This, {together with the fact that $0\leq \underline{\phi}_r\leq 1$}, implies that $\underline{\phi}_r\leq \chi_{\widetilde{\Delta}_{\tau, s_*}'(2r)}$. 

The final assertion of the proposition follows from the fact that $\mu_K(\underline{K}_r')\asymp \mu_K(\underline{K}_r)$ due to the construction of $\underline{K}_r'$ and Proposition \ref{prop: measure estimates the thickenings}.
\end{proof}

We are now ready to prove the zero-one laws in Theorem \ref{thm: zero-one laws}.  
Let $\psi:[1,\infty)\to (0,1)$ be a continuous decreasing function as in Theorem \ref{thm: zero-one laws}. We fix the function $\ttr=\ttr_{\psi}$ given by Proposition \ref{prop:dani} in the remainder of Section \ref{sec: proof of zero-one laws}. Let $\tau=\tau_K$ be the parameter as in \eqref{equ:taudef}. {Let $c>0$ be the constant as in Theorem \ref{thm:effdoubequi} so that the effective double equidistribution estimate \eqref{equ:doubleequ} holds. Moreover, the single equidistribution estimate \eqref{equ:singleequ} also holds with $c$ in place of $c_1$; see Remark \ref{rmk:uniformc}.}
We also fix the two parameters $0<r_*,s_*<0.01$ as in Proposition \ref{prop: smooth approximation}, and we abbreviate
\begin{align*}
    \widetilde{\Delta}_{\tau}'(r)=\widetilde{\Delta}_{\tau,s_*}'(r)\quad \text{and}\quad \widetilde{\Delta}_{\tau}''(r)=\widetilde{\Delta}_{\tau,s_*}''(r),\qquad \forall\, 0<r<r_*.
\end{align*}
 
\subsection{Convergence case} \label{sec: convergence case}
In this section we establish the convergence part of Theorem \ref{thm: zero-one laws}.
We assume the following series is convergent:
\begin{align}\label{equ:seriesdk}
\left\lbrace\begin{array}{cc}
\sum_kk^{-1}F_{\psi}(k)^4 & K=\Q(\sqrt{-1}),\\
\sum_kk^{-1}F_{\psi}(k)^2\log\left(1+\tfrac{1}{F_{\psi}(k)}\right)  &  K=\Q(\sqrt{-3}).
\end{array}\right. 
\end{align}
 We would like to show $\Leb(\DI_K(\psi))=1$. In this case,  \eqref{equ:dieq} implies the following series
\begin{align}\label{eq: convergence for r}
\sum_{k}f_K(\ttr(\tau k))
\end{align}
is convergent. 
For any $0<r<1$, recall that we have defined $\widetilde\Delta_\tau(r)$ as in \eqref{eq: upper bound set for smooth approx}. 
{Note that  \eqref{eq: convergence for r}, together with \eqref{equ:thickeningmeest} and \eqref{equ:fkbd}, implies that 
\begin{align}\label{equ:conseries}
   \sum_{k}\mu_K(\widetilde\Delta_\tau(C_{\ttr, \tau}\ttr(\tau k)))<\infty.  
\end{align}
}
Moreover, by \eqref{equ:limsupbd}, we have
$$
\DI_K(\psi)^c
\subset
\limsup_{k\to\infty}
\left(g_{-\tau k}\widetilde\Delta_\tau(C_{\ttr, \tau}\ttr(\tau k))\cap \mathcal Y\right).
$$
Hence, by the convergent case of the Borel-Cantelli lemma, we have
$$
\sum_{k \geq 1} \mathrm{Leb} \left ( E_k\right ) <\infty \quad \Rightarrow \quad \Leb(\DI_K(\psi)^c)=0.
$$
Here 
$$E_k:= \{z\in\mathbb C/\mathcal O_K:
g_{\tau k}\Lambda_z\in\widetilde\Delta_\tau(C_{\ttr, \tau}\ttr(\tau k))\}.$$

By Proposition~\ref{prop: smooth approximation}, there is a constant $L>0$ such that for all $r>0$ sufficiently small, there exists $\overline{\phi}_r\in C_c^\infty(X)$ so that
\begin{align}\label{equ:smbd}
\chi_{\widetilde\Delta_\tau(r)}
\le \overline{\phi}_r
\le \chi_{\widetilde\Delta_\tau(2r)}\quad \text{and }\quad  \|\overline{\phi}_r\|_{C^7}\ll r^{-L}.
\end{align}
As in \cite[Section~7.2]{KleinbockStrombergssonYu2022}, one first notes that the convergence of the series in \eqref{eq: convergence for r} forces $\ttr(\tau k)\to 0$ as $k\to \infty$. Fix a  small constant $\eta=\frac{c\tau}{2L}$ and set $u_k:=C_{\ttr, \tau}\ttr(\tau k)$.  
If $u_k>e^{-\eta k}$,
Theorem~\ref{thm: single equidistribution} and \eqref{equ:smbd} give
\begin{align}\label{eq: bound 1 for convergence}
    \mathrm{Leb}(E_k) &\le \int_{\mathcal Y} \overline{\phi}_{u_k}(g_{\tau k}\Lambda_z) \dd z 
    = \mu_K(\overline{\phi}_{u_k}) + O\!\left(e^{-c\tau k}u_k^{-L}\right) \nonumber\\
    & \ll \mu_K(\widetilde\Delta_\tau(2u_k))+ e^{-\frac{c\tau}{2} k} 
    \asymp \mu_K(\widetilde\Delta_\tau(u_k))+ e^{-\frac{c\tau}{2} k}.
\end{align}
If $u_k\le e^{-\eta k}$, since the family $\{\widetilde{\Delta}_{\tau}(r)\}$ is nested in $r$, we have 
$$
E_k\subset\left\{z\in \C/\cO_K: g_{\tau k}\Lambda_z\in \widetilde{\Delta}_{\tau}(e^{-\eta k})\right\}.
$$ 
Hence 
\begin{align}\label{eq: bound 2 for convergence}
    \Leb(E_k)&
    \leq \int_{\mathcal Y} \overline{\phi}_{e^{-\eta k}}(g_{\tau k}\Lambda_z)\dd z 
    = \mu_K(\overline{\phi}_{e^{-\eta k}}) + O\!\left(e^{-c\tau k}(e^{-\eta k})^{-L}\right)\\ \nonumber
    &\ll e^{-\eta k}+e^{-\frac{c\tau}{2}k}.
\end{align}
Combining \eqref{equ:conseries}, \eqref{eq: bound 1 for convergence} and \eqref{eq: bound 2 for convergence} yields
$$
\sum_{k}\mathrm{Leb}(E_k)
\ll
\sum_{k}\mu_K(\widetilde\Delta_\tau(u_k)) + \sum_{k}\left(e^{-\eta k}+e^{-\frac{c\tau}{2}k}\right) < \infty.
$$
Therefore, we conclude that $\Leb(\DI_K(\psi)^c)=0$, and hence $\mathrm{Leb}(\DI_K(\psi))=1$.

\subsection{Divergence case}
\label{sec:dive}
We now prove the divergence case in Theorem \ref{thm: zero-one laws}. We assume the series in \eqref{equ:seriesdk} is divergent. Thus, in view of \eqref{equ:dieq}, the series in \eqref{eq: convergence for r} is divergent, that is,
\begin{align}\label{eq: convergencer}
\sum_{k}f_K(\ttr(\tau k))=\infty.
\end{align}
Recall that by Lemma~\ref{lemma: limsupbd}, we have 
\begin{align}\label{equ:ndidieq}
\limsup_{k\to\infty}(g_{-\tau k}\widetilde{\Delta}_\tau({C_{\ttr, \tau}^{-1}\ttr({\tau}(k+1))})\cap \cY)\subset \DI_K(\psi)^c.
\end{align}
Set $\alpha_k=f_K(\ttr(\tau (k+1)))$. Then $\{\alpha_k\}$ is a sequence of positive numbers with divergent sum. By \cite[Lemma 4.3]{StrombergssonYu2026}, there exists a decreasing sequence of positive numbers $\{\beta_k\}$ such that 
\begin{align}\label{equ:deseqbk}
\lim_{k\to\infty}\beta_k=0\quad \text{and}\quad \sum_{k}\min\{ \alpha_k, \beta_k\}=\infty.
\end{align}

Recall that $f_K$ is positive, continuous and strictly increasing on $(0,\infty)$; thus its inverse function $f_K^{-1}$ is also positive, continuous and strictly increasing on $(0,\infty)$. Define 
\begin{align*}
v_k:=f_K^{-1}(\min\{\alpha_k, \beta_k\}). 
\end{align*}
Then $\{v_k\}$ is a sequence of positive numbers satisfying $f_K(v_k)\leq \beta_k$. Since $f_K$ is positive, strictly increasing and satisfies $\lim_{x\to 0^+}f_K(x)=0$, we have $\lim_{y\to 0^+}f_K^{-1}(y)=0$. Then by the first condition in  \eqref{equ:deseqbk} we have 
\begin{align*}
0< v_k=f_K^{-1}(f_K(v_k))\leq f_K^{-1}(\beta_k)\to 0.
 \end{align*}
 Thus
 \begin{align}\label{equ:vk0}
     \lim_{k\to\infty} v_k=0.
 \end{align}

 On the other hand, by the second condition in \eqref{equ:deseqbk} we have 
 \begin{align}\label{equ:fvk}
\sum_k f_K(v_k)=\sum_k \min\{\alpha_k, \beta_k\}=\infty.
\end{align}
Define 
\begin{align}
    \rho_k:=\frac{\max\{v_k, k^{-5}\}}{2C_{\ttr, \tau}}.
\end{align}
By \eqref{equ:vk0} and \eqref{equ:fvk} we clearly have 
\begin{align}\label{equ:divefk}
    \lim_{k\to\infty}\rho_k=0\quad \text{and}\quad \sum_kf_K(\rho_k)=\infty.
\end{align}

Let $0<r_*, s_*<0.01$ be the two parameters as in Propositions \ref{prop:disjoint} and \ref{prop: smooth approximation}. 
Fix $k_0\in \N$ a sufficiently large integer such that 
\begin{align}\label{equ:sqebd}
    2\rho_k^{1/2}< \min\{r_*, e^{-4\tau}\},\qquad \forall\, k\geq k_0.
\end{align} 
Note that such an integer exists since $\rho_k\to 0$. In particular, since $\rho_k<1$, this also implies $\rho_k<r_*$ for all $k\geq k_0$. 
Then by Proposition \ref{prop: smooth approximation} we have 
\begin{align}\label{equ:smapplphi}
\chi_{\widetilde\Delta''_\tau({\rho_k})}
\le \underline{\phi}_{\rho_k}
\le \chi_{\widetilde\Delta'_\tau(2{\rho_k})},
\qquad
\|\underline{\phi}_{\rho_k}\|_{C^7}\ll \rho_k^{-L},\qquad \forall\, k\geq k_0.
\end{align}
Moreover, by \eqref{equ:thickeningmeest}, \eqref{equ:douprest} and the above first estimate we have 
\begin{align}\label{equ:mukphiest}
    \mu_K(\underline{\phi}_{\rho_k})\asymp f_K(\rho_k),\qquad \forall\, k\geq k_0. 
\end{align}
Define for all $k\geq k_0$,
$$
h_k(z):=\underline{\phi}_{\rho_k}(g_{\tau k}\Lambda_z),
\qquad
b_k:=\int_{\C/\cO_K} h_k(z)\, \mathrm{d} z .
$$
We claim that in order to prove the divergence case of Theorem \ref{thm: zero-one laws}, it suffices to show the following statement:
\begin{align}\label{equ:equistate}
    \text{For $\mathrm{Leb}$-a.e. $z\in \C/\cO_K$, $h_k(z)>0$ for infinitely many $k\geq k_0$. }
\end{align}
Indeed, suppose \eqref{equ:equistate} holds. Then in view of the definition of $h_k$ and by  \eqref{equ:smapplphi}  we have for $\mathrm{Leb}$-a.e. $z\in \C/\cO_K$, 
$$
g_{\tau k}\Lambda_z\in \widetilde\Delta'_\tau(2\rho_k)\subset \widetilde\Delta_\tau(2\rho_k),\quad \text{for infinitely many $k\geq k_0$.}
$$
That is, 
\begin{align}\label{equ:interleb1}
\mathrm{Leb}(\limsup_{k\to\infty}(g_{-\tau k}\widetilde{\Delta}_\tau(2\rho_k)\cap \cY))=1.
\end{align} 
In view of the definition of $\rho_k$, we have 
\begin{align*}
    \limsup_{k\to\infty}(g_{-\tau k}\widetilde{\Delta}_\tau(2\rho_k)\cap \cY)\subset \limsup_{k\to\infty}(g_{-\tau k}\widetilde{\Delta}_\tau(C_{\ttr, \tau}^{-1}v_k)\cap \cY)\cup \limsup_{k\to\infty}(g_{-\tau k}\widetilde{\Delta}_\tau(C_{\ttr, \tau}^{-1}k^{-5})\cap \cY).
\end{align*}

Note that $\sum_k f_K(C_{\ttr, \tau}^{-1}k^{-5})\asymp_{C_{\ttr, \tau}}\sum_k f_K(k^{-5})<\infty$. Thus by the convergence case proved in Section \ref{sec: convergence case}, we have 
$$
\mathrm{Leb}(\limsup_{k\to\infty}(g_{-\tau k}\widetilde{\Delta}_\tau(C_{\ttr, \tau}^{-1}k^{-5})\cap \cY))=0.
$$ 
This, together with the above inclusion relation and \eqref{equ:interleb1}, implies that 
\begin{align*}
  \mathrm{Leb}(\limsup_{k\to\infty}(g_{-\tau k}\widetilde{\Delta}_\tau(C_{\ttr, \tau}^{-1}v_k)\cap \cY))  =1.
\end{align*}
Finally, note that $v_k\leq f_K^{-1}(\alpha_k)\leq \ttr(\tau(k+1))$. This, together with \eqref{equ:ndidieq} implies that
$$
\limsup_{k\to\infty}(g_{-\tau k}\widetilde{\Delta}_\tau(C_{\ttr, \tau}^{-1}v_k)\cap \cY)\subset \limsup_{k\to\infty}(g_{-\tau k}\widetilde{\Delta}_\tau(C_{\ttr, \tau}^{-1}\ttr(\tau(k+1)))\cap \cY) 
\subset \DI_K(\psi)^c.
$$
This then shows that $\mathrm{Leb}(\DI_K(\psi)^c)=1$ as desired.

\subsubsection{Applications of single and double equidistribution}
For the remainder of this section we prove \eqref{equ:equistate}. 
Recall that, by the definition of $\rho_k$, we have 
\begin{align}\label{eq: lb and up for r(k)}
    \rho_k\geq \frac{k^{-5}}{2C_{\ttr, \tau}},\qquad \forall\, k\geq k_0.
\end{align}
Then by Theorem \ref{thm: single equidistribution}, \eqref{eq: lb and up for r(k)} and the estimate that $\|\underline{\phi}_{\rho_k}\|_{C^7}\ll \rho_k^{-L}$, see \eqref{equ:smapplphi}, we have
\begin{align}\label{equ:bkest}
b_k = \mu_K(\underline{\phi}_{\rho_k}) + O(e^{-c\tau k}\rho_k^{-L}) = \mu_K(\underline{\phi}_{\rho_k})+O(e^{-\frac{c\tau}{2} k}f_K(\rho_k)),\qquad \forall\, k\geq k_0.
\end{align}
In particular, combining with \eqref{equ:mukphiest} we get 
\begin{align}\label{eq: asymptotic for b_k}
    b_k\asymp \mu_K(\underline{\phi}_{\rho_k})
\asymp f_K(\rho_k)\gg \rho_k^4,\qquad \forall\, k\geq k_0.
\end{align}
Hence by the second condition in \eqref{equ:divefk} we have  
\begin{align}\label{eq: bk divergent}
    \sum_{k\ge k_0} b_k=\infty .
\end{align}
 For $j>i\ge k_0$, put
$$
b_{i,j}:=
\int_{\mathbb C/\cO_K}
\bigl(h_i(z)h_j(z)-b_i b_j\bigr)\,\mathrm{d} z .
$$

{We first give various estimates on $b_{i,j}$ for all $j> i\geq k_0$.}  By Theorem \ref{thm:effdoubequi}, \eqref{equ:bkest}, \eqref{eq: asymptotic for b_k} and \eqref{eq: lb and up for r(k)}, we have for all $j> i\geq k_0$,
\begin{align}\label{eq: double mixing}
    \int_{\C/\cO_K} h_i(z)h_j(z)\,\mathrm{d}z
=
b_ib_j
+
O\!\left(e^{-c\tau (j-i)}\rho_i^{-L}\rho_j^{-L}
+ e^{-c\tau j}\rho_j^{-L}b_i+
e^{-c\tau i}\rho_i^{-L}b_j\right).
\end{align}
Note also that 
\begin{align}\label{equ:bijtbd}
    |b_{i,j}|\leq \max\left\{\int_{\mathbb C/\cO_K}
h_i(z)h_j(z)\,\mathrm{d} z , b_ib_j\right\}\leq \min\{b_i, b_j\}.
\end{align}
Let 
\begin{align*}
    D_{i,j}:=\min\{i,j-i\},\qquad \forall\, j>i\geq k_0.
\end{align*} 
Set 
\begin{align}\label{equ:thetadef}
    \Theta:=\frac{8(L+4)}{c\tau}\quad \text{and}\quad \vartheta:=\frac{c\tau}{4(L+4)}.
\end{align}
We split our discussion below depending on the comparison between $D_{i,j}$ and $\Theta\log(1/\rho_j)$.

First assume 
\begin{align*}
    D_{i,j}\geq \Theta \log \left(\frac{1}{\rho_j}\right) \quad \Leftrightarrow\quad \rho_j\geq e^{-\Theta^{-1}D_{i,j}}.
\end{align*}
If we further assume $\rho_j\leq \rho_i e^{\vartheta(j-i)}$,
then
\begin{align*}
\rho_i\rho_j\geq\rho_j^2 e^{-\vartheta(j-i)}
\geq e^{-(2/\Theta+\vartheta)(j-i)}.
\end{align*}
This implies that 
\begin{align*}
  e^{-c\tau (j-i)}\rho_i^{-L}\rho_j^{-L}  
  \leq e^{(-c\tau+(2/\Theta+\vartheta)(L+4)) (j-i)} \rho_i^{4}\rho_j^{4}= e^{-\frac{c\tau}{2}(j-i)}\rho_i^{4}\rho_j^{4},
\end{align*}
where the last equality holds since
$\bigl(\frac{2}{\Theta}+\vartheta\bigr)(L+4) =\frac{c\tau}{2}$,
which follows from \eqref{equ:thetadef}. Moreover, using the polynomial lower bound \eqref{eq: lb and up for r(k)} we have $e^{-c\tau j}\rho_j^{-L}\ll e^{-\frac{c\tau}{2}j}\rho_j^4$ and $e^{-c\tau i}\rho_i^{-L}\ll e^{-\frac{c\tau}{2}i}\rho_i^4$. Using also $\rho_k^4\ll f_K(\rho_k)\asymp b_k$ for all $k\geq k_0$, we get from \eqref{eq: double mixing} that 
\begin{align*}
|b_{i,j}|\ll e^{-\frac{c\tau}{2}D_{i,j}}b_ib_j\ll e^{-\frac{c\tau}{2}D_{i,j}}b_j.
\end{align*}
If instead $\rho_j\geq  \rho_i e^{\vartheta(j-i)}$, then we have by \eqref{equ:bijtbd} that
\begin{align}
\frac{|b_{i,j}|}{b_j}\leq \frac{b_i}{b_j}\asymp \frac{f_K(\rho_i)}{f_K(\rho_j)}\ll \frac{\rho_i}{\rho_j}\leq e^{-\vartheta(j-i)}\leq e^{-\vartheta D_{i,j}}.  
\end{align}
Here for the last estimate we used \eqref{equ:fkest2} and that in this case $0<\rho_i\leq \rho_j$.
Combining the above two estimates and noting that $\vartheta<\frac{c\tau}{2}$, we get
\begin{align}\label{eq: bound 1 for bij}
    |b_{i,j}|\ll e^{-\vartheta D_{i,j}}b_j, \qquad \text{whenever
$D_{i,j}\ge \Theta \log \left(\frac{1}{\rho_j}\right) $}.
\end{align}
For $D_{i,j}<\Theta \log (\frac{1}{\rho_j})$, if $D_{i,j}=i<\Theta\log (\frac{1}{\rho_j})$, then $j-i\ge \frac{3j}{4}$ for  $j\geq k_0$, after possibly enlarging $k_0$. Then similarly  \eqref{eq: lb and up for r(k)}, \eqref{eq: asymptotic for b_k} 
and \eqref{eq: double mixing} give
\begin{align}\label{eq: bound 2 for bij}
    |b_{i,j}|\ll e^{-\frac{c{\tau}}{2}i}b_j, \qquad \text{whenever
$D_{i,j}=i<\Theta\log \left(\frac{1}{\rho_j}\right)$}.
\end{align}
It remains to treat the case when $D_{i,j}= j-i\le \Theta\log (\frac{1}{\rho_j})$.  In this case define
\begin{align}\label{equ:phiijdef}
\Phi_{i,j}(x):=\underline{\phi}_{\rho_i}(x)\underline{\phi}_{\rho_j}(g_{\tau(j-i)}x).
\end{align}
Using the definition of $b_{i,j}$ and applying Theorem~\ref{thm: single equidistribution} to $\Phi_{i,j}$ we get
$$
b_{i,j}\leq \int_{\C/\cO_K} h_i(z)h_j(z)\,\mathrm{d}z
=\int_{\C/\cO_K} \Phi_{i,j}(g_{\tau i}\Lambda_z)\, \mathrm{d} z =
\mu_K(\Phi_{i,j})
+
O\!\left(e^{-c\tau i}\|\Phi_{i,j}\|_{C^7}\right).
$$
Using the $C^7$-norm estimate for $\Phi_{i,j}$ under left translation (cf. \cite[Lemma 3.2 and Remark 3.3]{StrombergssonYu2026}), the lower bound in  \eqref{eq: lb and up for r(k)} on $\rho_i,\rho_j$, the estimates \eqref{eq: asymptotic for b_k}, \eqref{equ:thickeningmeest}  and the estimate $i=j-D_{i,j}>\frac{3}{4}j$ in this case, 
the above error term is seen to be bounded by $O(e^{-\frac{c{\tau}}{2} j}b_j)$. Thus
\begin{align}\label{eq: bound 3 for bij}
   b_{i,j}\le \mu_K(\Phi_{i,j})+O(e^{-\frac{c{\tau}}{2} j}b_j),
\qquad \text{whenever $D_{i,j}=j-i<\Theta\log \left(\frac{1}{\rho_j}\right)$}.
\end{align}

\subsubsection{Admissible indices}
Following the ideas in \cite{StrombergssonYu2026}, we introduce the following notion of admissible indices to single out the times $j$ whose neighboring targets in the relevant short range are not substantially larger than the target at the time $j$. 
Let $\Theta$ be as in \eqref{equ:thetadef}. 
For any integer $j\geq k_0$, we say 
that $j$ is \textit{admissible}
if 
\begin{align}\label{equ:alphaaddm}
 \rho_i\leq \rho_j^{1/2},\qquad \forall\, \lceil\max\{k_0,j-\Theta\log(1/\rho_j)\}\rceil\leq i<j.
\end{align}
Denote by $\mathcal{G}\subset [k_0,\infty)\cap \N$ the set of all admissible indices. 
\begin{Lem}
We have 
\begin{align}\label{eq: bk divergentad}
    \sum_{k\in \cG}b_k=\infty.
\end{align}
\end{Lem}
\begin{proof}
The proof is almost identical to that of \cite[Lemma 4.5]{StrombergssonYu2026}. For completeness, we include the details.

Let $\mathcal{B}:=([k_0,\infty)\cap \N)\setminus \cG$ be the set of non-admissible indices.
    It follows from the definition in \eqref{equ:alphaaddm} that $j\in \cB$ holds if and only if there exists some integer $i$ in the interval $\lceil\max\{k_0, j-\Theta \log(1/\rho_j)\}\rceil\leq i< j$ such that $\rho_i> \rho_j^{1/2}$;
    then let $w(j)$ be a fixed choice of such an integer $i$.
    Next, for each $i\geq k_0$, let $\cS(i)$ be the set of all integers $j>i$ satisfying 
    $w(j)=i$.
    Then by the estimate that $b_k\asymp f_K(\rho_k)$ (see \eqref{eq: asymptotic for b_k}), for any $k'\geq k_0$ we have 
    \begin{align}\label{GsumstillinfiniteLEMpf1}
        \sum_{j\in \cB,\: j\leq k'}b_j\leq \sum_{k_0\leq i<k'}\sum_{j\in\cS(i)}
        b_j\asymp \sum_{k_0\leq i<k'}\sum_{j\in\cS(i)}
f_K(\rho_j).
    \end{align}
Now for any $i\geq k_0$ and $j\in\cS(i)$
we have $i<j\leq i+\Theta\log(1/\rho_j)$ and $\rho_j<\rho_i^2$,
and hence for any integer $V\geq 0$, if $i\geq k_0$, $j\in\cS(i)$ and $e^{-V-1}<\rho_j\leq e^{-V}$, then 
$e^{-V-1}<\rho_i^2$, i.e.\ $V>2\log(1/\rho_i)-1$,
and also $i<j
<i+\Theta(V+1)$.
Consequently, for each fixed $i$ and $V$, there are at most $\Theta(V+1)=O(V+1)$ possible values of $j$.
Since $f_K$ is strictly increasing, we get
\begin{align}\label{GsumstillinfiniteLEMpf2}
\sum_{j\in\cS(i)}f_K(\rho_j)
\ll \sum_{V>2\log(1/\rho_i) -1}  f_K(e^{-V})(V+1)\ll 
\begin{cases}
\rho_i^8\log(1/\rho_i) &K=\Q(\sqrt{-1}),\\[2mm]
\rho_i^4\log^2(1/\rho_i) &K=\Q(\sqrt{-3}),
\end{cases}
\ll b_i\rho_i.
\end{align}
Here the second bound follows from a direct computation, and the last bound holds again by the estimate that $b_k\asymp f_K(\rho_k)$.
It then follows from \eqref{GsumstillinfiniteLEMpf1} and \eqref{GsumstillinfiniteLEMpf2}
that there is a constant $C>0$ such that for any $k'\geq k_0$,
\begin{align*}
\sum_{j\in \cB,\: j\leq k'}b_j
\leq C\sum_{k_0\leq j<k'}b_j\rho_j.
\end{align*}
Since $\rho_j\to 0$, the above estimate shows that 
$$
\sum_{j\in \cB,\: j\leq k'}b_j=o\left(\sum_{k_0\leq j\leq k'}b_j\right),\qquad \text{as $k'\to\infty$},
$$
which, together with the fact that $\sum_{k\geq k_0}b_k=\infty$, then clearly implies \eqref{eq: bk divergentad}.
\end{proof}

\subsubsection{Proof of the divergence case}
For $k_2>k_1\ge k_0$ with $k_1,k_2\in \cG$, define the correlation function
\begin{align}\label{eq: defn of correlation}
    Q^{\cG}_{k_1,k_2}:=
\int_{\mathbb C/\mathcal O_K}
\left(
\sum_{\substack{k_1\leq i\leq k_2\\ i\in \cG}}h_i(z)-\sum_{\substack{k_1\leq i\leq k_2\\ i\in \cG}}b_i
\right)^2 \mathrm{d}z .
\end{align}

{Under the assumption \eqref{eq: bk divergentad}, a divergent Borel-Cantelli lemma (see e.g. \cite[Chapter I, Lemma 10]{Sp79})} shows that if 
\begin{align}\label{eq: required correlation bound}
     Q^{\cG}_{k_1,k_2}\ll \sum_{\substack{k_1\leq j\leq k_2\\ j\in \cG}} b_j,\qquad \forall \, k_2\geq k_1 \gg k_0\quad \text{with $k_1,k_2\in \cG$},
\end{align}
then for $\mathrm{Leb}$-a.e. $z\in \C/\cO_K$, $h_k(z)>0$ for infinitely many $k\geq k_0$ with $k\in \cG$. This then clearly implies \eqref{equ:equistate}. 
It thus suffices to prove the bound \eqref{eq: required correlation bound}.

For the remainder of this section, we prove \eqref{eq: required correlation bound}.
Expanding the square in \eqref{eq: defn of correlation} gives
\begin{align}\label{eq: expanding correlation}
    Q^{\cG}_{k_1,k_2}
\le
\sum_{\substack{k_1\leq i\leq k_2\\ i\in \cG}}b_i
+
2\sum_{\substack{k_1\le i<j\le k_2\\ i,j\in \cG}} b_{i,j}.
\end{align}

By \eqref{eq: bound 1 for bij} and \eqref{eq: bound 2 for bij}, one verifies that the contribution to $\sum_{\substack{k_1\le i<j\le k_2\\ i,j\in \cG}} b_{i,j}$ of those  $k_1\leq i<j\leq k_2$ with $i,j\in \cG$ and  satisfying either $D_{i,j}\geq \Theta\log (1/\rho_j)$ or $D_{i,j}=i< \Theta\log (1/\rho_j)$ is bounded by
\begin{align}\label{eq: contribution for long range}
    O\!\left(\sum_{\substack{k_1\leq j\leq k_2\\ j\in \cG}}b_j\right).
\end{align}
Then in view of \eqref{eq: bound 3 for bij}, it remains to get the following bound
\begin{align*}
    \sum_{\substack{k_1\leq j\leq k_2\\ j\in \cG}}\sum_{a_j\le i<j}\mu_K(\Phi_{i,j})\ll \sum_{\substack{k_1\leq j\leq k_2\\ j\in \cG}} b_j. 
\end{align*}
Here $a_j:=\lceil\max\{k_1,j-\Theta\log(1/\rho_j)\}\rceil$. 
For any $j\geq k_0$, define
\begin{align*}
    R_j:=\rho_j^{1/2}.
\end{align*}
Using the assumption that $k_1\geq k_0$ and the definition of admissible indices, we have for any $j\geq k_0$,
\begin{align}\label{equ:disjarg}
    j\in \cG\quad \Rightarrow\quad  \rho_i\leq R_j,\quad \forall\, a_j \leq i< j.
\end{align}
Now for any $k_1\leq j\leq  k_2$ with $j\in \cG$, by \eqref{equ:smapplphi}, the definition \eqref{equ:phiijdef} of $\Phi_{i,j}$ and  the invariance of the Haar measure $\mu_K$ we get  
$$
\sum_{a_j\le i<j}\mu_K(\Phi_{i,j})
\le \sum_{a_j\le i<j}\mu_K\left(g_{\tau (j-i)} \widetilde\Delta'_\tau(2\rho_i)
\cap
\widetilde\Delta'_\tau(2\rho_j)\right).
$$
For any $k_1\leq j\leq k_2$ with $j\in \cG$, let
$$
J_{j}:=\left\lfloor \frac{1}{4\tau}\log\left(\frac1{2R_j}\right)\right\rfloor.
$$
By the assumption $2\rho_k^{1/2}< e^{-4\tau}$, we have $J_j\geq 1$ for all $j\geq k_0$ with $j\in \cG$. By assumption \eqref{equ:sqebd} we see that  Proposition~\ref{prop:disjoint} applies for the parameter $2R_j$. Thus by Proposition~\ref{prop:disjoint} and \eqref{equ:disjarg} we see that every block of length $J_j$ in 
\begin{align*}
    g_{\tau (j-a_j)}\widetilde\Delta'_\tau(2\rho_{a_j}), \ g_{\tau (j-a_j-1)}\widetilde\Delta'_\tau(2\rho_{a_j+1}),\ \cdots, \ g_{\tau }\widetilde\Delta'_\tau(2\rho_{j-1})
\end{align*}
is pairwise disjoint.  Therefore
\begin{align}\label{eq: bound due to disjointness}
    \sum_{a_j\le i<j}\mu_K\left( g_{\tau (j-i)}\widetilde\Delta'_\tau(2\rho_i)
\cap
\widetilde\Delta'_\tau(2\rho_j)\right)
{\leq }
\left(\frac{j-a_j}{J_j}+1\right)
\mu_K(\widetilde\Delta'_\tau(2\rho_j)).
\end{align}

Since $j-a_j\leq \Theta\log (1/\rho_j)$ and $J_j\asymp \log(1/\rho_j)$ (since $R_j=\rho_j^{1/2}$), we have 
$
    \frac{j-a_j}{J_j}\ll 1. 
$
Thus \eqref{eq: bound due to disjointness} implies that
\begin{align}\label{eq: bound for Phi ij}
 \sum_{\substack{k_1\leq j\leq k_2\\ j\in \cG}}\sum_{a_j\le i<j}\mu_K(\Phi_{i,j})\ll \sum_{\substack{k_1\leq j\leq k_2\\ j\in \cG}}\mu_K(\widetilde\Delta'_\tau(2\rho_j))
\ll \sum_{\substack{k_1\leq j\leq k_2\\ j\in \cG}} b_j .
\end{align}
Combining \eqref{eq: contribution for long range} and \eqref{eq: bound for Phi ij}, in view of \eqref{eq: expanding correlation}, we obtain the desired upper bound \eqref{eq: required correlation bound}. This completes the proof of the divergence case of Theorem \ref{thm: zero-one laws}.

\bibliographystyle{abbrv}
\bibliography{DKbibliog}

\end{document}